\documentclass[reqno,11pt]{amsart}

\usepackage{amssymb}
\usepackage{amssymb, latexsym}
\usepackage{hyperref}
\usepackage{amsmath}
\usepackage{amscd}
\usepackage{enumerate}
\usepackage{amsfonts}
\usepackage{graphicx}
\usepackage[all]{xypic}
\usepackage{mathrsfs}

\allowdisplaybreaks

\newcommand{\Z}{{\mathbb{Z}}}
\newcommand{\C}{{\mathbb{C}}}
\newcommand{\R}{{\mathbb{R}}}
\newcommand{\N}{{\mathbb{N}}}

\newcommand{\eps}{{\varepsilon}}

\theoremstyle{plain}
\newtheorem{theorem}{Theorem}
\newtheorem{proposition}[theorem]{Proposition}
\newtheorem{lemma}[theorem]{Lemma}
\newtheorem{corollary}[theorem]{Corollary}

\theoremstyle{definition}

\newtheorem{remark}[theorem]{Remark}

\numberwithin{equation}{section}
\numberwithin{theorem}{section}

\numberwithin{equation}{section}

\newcommand{\LLL}{\mathcal{L}}
\newcommand{\MMM}{\mathcal{M}}

\newcommand{\UUU}{\mathcal{U}}

\newcommand{\YYY}{\mathcal{Y}}
\newcommand{\ZZZ}{\mathcal{Z}}

\begin{document}

\title[Dynamics for the focusing Hartree equation]
{Dynamics for the focusing, energy-critical nonlinear Hartree
equation}

\author[Miao]{Changxing Miao}
\address{\hskip-1.15em Changxing Miao:
\hfill\newline Institute of Applied Physics and Computational
Mathematics, \hfill\newline P. O. Box 8009,\ Beijing,\ China,\
100088,}
\email{miao\_changxing@iapcm.ac.cn}

\author[Wu]{Yifei Wu}
\address{\hskip-1.15em Yifei Wu \hfill\newline Institute of
Applied Physics and Computational Mathematics, \hfill\newline  P. O.
Box 8009,\ Beijing,\ China,\ 100088, } \email{yerfmath@yahoo.cn}

\author[Xu]{Guixiang Xu}
\address{\hskip-1.15em Guixiang Xu \hfill\newline Institute of
Applied Physics and Computational Mathematics, \hfill\newline P. O.
Box 8009,\ Beijing,\ China,\ 100088, }
\email{xu\_guixiang@iapcm.ac.cn}



\begin{abstract}In \cite{LiMZ:e-critical Har,
MiaoXZ:09:e-critical radial Har}, the dynamics of the solutions for
the focusing energy-critical Hartree equation have been classified
when $E(u_0)<E(W)$, where $W$ is the ground state. In this paper, we
continue the study on the dynamics of the radial solutions with the
threshold energy. Our arguments closely follow those in
\cite{DuyMerle:NLS:ThresholdSolution,
DuyMerle:NLW:ThresholdSolution, DuyRouden:NLS:ThresholdSolution,
LiZh:NLS, LiZh:NLW}. The new ingredient is that we show that the
positive solution of the nonlocal elliptic equation in
$L^{\frac{2d}{d-2}}(\R^d)$ is regular and unique by the moving plane
method in its global form,  which plays an important role in the
spectral theory of the linearized operator and the dynamics behavior
of the threshold solution.
\end{abstract}

\maketitle


%
%
%
%

\section{Introduction}

We consider the focusing, energy-critical Hartree equation
\begin{equation}\label{har}
i \partial_t u + \Delta u  + \left( |x|^{-4}*|u|^2\right)u =0,\quad
(t,x)\in \R \times \R^d
\end{equation}
where $ d\geq 5$ in this context. The Hartree equation arises in the
study of Boson stars and other physical phenomena, please refer to
\cite{pi}. In chemistry, it appears as a continuous-limit model for
the mesoscopic structures, see \cite{grc}. Moreover, \eqref{har}
enjoys several symmetries: If $u(t,x)$ is a solution, then
\begin{enumerate}
\item[(a)] by scaling: so is $\lambda^{-\frac{d-2}{2}}u(\lambda^{-2} t,
\lambda^{-1} x)$, $\lambda> 0$;
\item[(b)] by time translation invariance: so is $u(t+t_0, x)$ for $t_0\in \R$;
\item[(c)] by spatial translation invariance: so is $u(t, x+x_0)$ for $x_0\in \R^d$;
\item[(d)] by phase rotation invariance: so is $e^{i\theta_0}u(t,x)$, $\theta_0\in \R$;
\item[(e)] by time reversal invariance: so is $\overline{u(-t, x)}$.
\end{enumerate}

The local well-posedness for the Cauchy problem of \eqref{har} was
developed in \cite{Caz:book, MiaoXZ:08:LWP for Har}. Namely, if $u_0
\in \dot H^1(\R^d)$, there exists a unique solution defined on a
maximal interval $I=\left(-T_-(u), T_+(u)\right)$ and the energy
\begin{equation}\label{energy}
\aligned  E(u(t)):=\frac12 \int_{\R^d} \big| \nabla u(t,x)\big|^2 dx
-\frac{1}{4} \iint_{\R^d\times\R^d} \frac{\big|u(t,x)\big|^2
\big|u(t,y)\big|^2}{|x-y|^{4}} \ dxdy=E(u_0)
\endaligned
\end{equation}
is  conserved on $I$. The name ``energy critical'' refers to the
fact that the scaling
\begin{equation}\label{scaling}
u(t,x)\rightarrow
u_{\lambda}(t,x)=\lambda^{-\frac{d-2}{2}}u\left(\lambda^{-2} t,
\lambda^{-1} x\right), \; \lambda>0
\end{equation}
makes the equation \eqref{har} and its  energy \eqref{energy}
invariant.

There are many results for the energy-critical Hartree equation. For
the defocusing case, Miao, etc, take advantage of the term
$\displaystyle - \int_{I}\int_{|x|\leq A|I|^{1/2}}|u|^{2}\Delta
\Big(\frac{1}{|x|}\Big)dxdt$ in the localized Morawetz identity,
which is related to the linear operator $i\partial_t + \Delta$ to
rule out the possibility of energy concentration, instead of the
classical Morawetz estimate dependent of the nonlinearity and thus
obtain the global well-posedness and scattering of the radial
solution in \cite{MiaoXZ:07:F e-critical radial Har}. Subsequently,
Miao, etc use the induction on energy argument in both the frequency
space and the spatial space simultaneously and the
frequency-localized interaction Morawetz estimate to remove the
radial assumption in \cite{MiaoXZ:09:DF e-critical nonradial Har}.

 For the focusing case,  the dynamics behavior becomes complicated. It turns
out that the explicit ground state
\begin{equation}\label{w function}
\aligned W(x)= c_0\left(\frac{t}{t^2+|x|^2}\right)^{ -\frac{d-2}{2}}
 \; \text{with}\; c_0>0, t>0
\endaligned
\end{equation}
plays an important role in the dynamical behavior of solutions for
\eqref{har}. Miao, etc, make use of the concentration compactness
principle and the rigidity argument, which are first introduced in
NLS and NLW by C.~Kenig and F.~Merle in \cite{Kenigmerle:H1 critical
NLS, Kenig-merle:wave},  to show that
\begin{theorem}[\cite{LiMZ:e-critical Har, MiaoXZ:09:e-critical radial Har}]\label{belowthreshold}
 Let $u$ be a solution of \eqref{har} with
\begin{equation*}
\aligned u_0 \in \dot H^1(\R^d), \;\; E(u)< E(W).
\endaligned
\end{equation*}
Then
\begin{enumerate}
\item[\rm (a)] if $\big\|\nabla u_0\big\|_{L^2 }< \big\|\nabla W\big\|_{L^2}$, then $I=\R$ and $
\big\|u\big\|_{L^6_t\big(\R; L^{\frac{6d}{3d-8}}_x\big)}<\infty$;

\item[\rm (b)] if $\big\|\nabla u_0\big\|_{L^2}> \big\|\nabla W\big\|_{L^2}$,
and either $u_0 \in L^2$ is radial  or $x\cdot u_0 \in L^2$, then
$T_{\pm} < \infty$.
\end{enumerate}
\end{theorem}
For other dynamics results of the Hartree equation, please refer to
\cite{CaoG, GiV00, KriLR:m-critcal Har, KriMR:mass-subcritical har,
LiMZ:e-critical Har, LiZh:har:class, MiaoXZ:08:LWP for Har,
MiaoXZ:09:e-critical radial Har, MiaoXZ:09:m-critical Har,
MiaoXZ:09:p-subcritical Har, MiaoXZ:10:blowup, Na99d}.

 In this paper, we
continue the study on the dynamics of the radial solutions with the
threshold energy $E(W)$. Our goal is to give the classification of
the solutions, that is, the initial data $u_0\in \dot H^1(\R^d)$ is
radial and satisfies
\begin{equation}\label{threshold energy}
\aligned  E(u_0)=E(W).
\endaligned
\end{equation}
In this case, the classification is more abundant than  that in
Theorem \ref{belowthreshold}. Clearly, $W$ is a new solution which
doesn't satisfy neither conclusion in Theorem \ref{belowthreshold}.
Besides  $W$,  there also exist two other special radial solutions
$W^{\pm}$.

\begin{theorem}\label{threholdsolution} There exist two radial solutions $W^\pm$ of \eqref{har} with
initial data $W^{\pm}_{0}$ such that
\begin{enumerate}
\item[\rm(a)] $E(W^\pm)=E(W)$, $T_+(W^{\pm})=\infty$ and
\begin{equation*} \aligned
\lim_{t\rightarrow +\infty}W^{\pm}(t) =  W \; \text{in} \; \dot H^1.
\endaligned
\end{equation*}

\item[\rm(b)] $\big\|\nabla W^-_0\big\|_{2} < \big\|\nabla W\big\|_{2}$, $T_-(W^{-})=\infty$ and $W^-$ scatters for the negative time.

\item[\rm(c)] $\big\|\nabla W^+_0\big\|_{2} > \big\|\nabla W\big\|_{2}$, and $T_-(W^{+})<\infty$.
\end{enumerate}
\end{theorem}

Next, we characterize all radial solutions with the threshold energy
as follows:

\begin{theorem}\label{classification} Suppose $u_0 \in \dot
H^1\big(\R^d\big)$ is radial, and such that
\begin{equation*}
\aligned E(u_0 )=E(W ).
\endaligned
\end{equation*}
Let  $u$  be a solution of \eqref{har} with initial data $u_0 $
and $I$ be the maximal interval of existence. Then the following
holds:
\begin{enumerate}
\item[\rm(a)] If $\big\|\nabla u_0\big\|_{2} < \big\|\nabla W \big\|_{2}$, then $I=\R$. Furthermore,
either $u=W^{-}$ up to the symmetries of the equation, or
$\big\|u\big\|_{L^{6}_tL^{\frac{6d}{3d-8}}_{x}}<\infty$.

\item[\rm(b)] If $\big\|\nabla u_0\big\|_{2} = \big\|\nabla W \big\|_{2}$, then  $u= W$ up to the symmetries of the equation.

\item[\rm(c)] If $\big\|\nabla u_0\big\|_{2} > \big\|\nabla W \big\|_{2}$, and $u_0 \in L^2(\R^d)$, then
either $u=W^{+}$ up to the symmetries of the equation, or $I$ is
finite.
\end{enumerate}
\end{theorem}
Because of the radial assumption, the symmetries of the equation in
the above theorem refer to the symmetries under scaling, time
translation, phase rotation and time reverse.

Now we begin with a brief recapitulation of some important dynamics
results for NLS, NLW and NLKG that have been derived so far. The
orbital stability of the soliton for the $L^2$-subcritical NLS in
the energy space was settled by Weinstein
\cite{Wein:85:Modulationary stability,Wein:86:Lyapunov stability},
Berestycki and Cazenave \cite{BerCaz}, Cazenave and Lions
\cite{CazLio}. In detail, Weinstein obtained the quantitative
analysis (modulation stability analysis), while Berestycki, Cazenave
and Lions gave the qualitative analysis. After the successful
applications of the concentration compactness principle (the profile
decomposition \cite{Bahouri-gerard, IbrMN:f:NLKG, Keraani:energy-c
NLS prfile}) into the global existence and scattering theory for the
$\dot H^1$-critical NLS and NLW with the energy less than that of
the ground state in \cite{Kenigmerle:H1 critical NLS,
Kenig-merle:wave}, Duyckaerts and Merle combined the spectral theory
of the linearized
 operator, the modulational stability of the soliton with the concentration compactness argument to classify the solutions
 with the threshold energy for the $\dot H^1$-critical NLS and NLW
 in \cite{DuyMerle:NLS:ThresholdSolution, DuyMerle:NLW:ThresholdSolution}. Subsequently,
  Duyckaerts and Roudenko dealt with the 3D cubic NLS case in \cite{DuyRouden:NLS:ThresholdSolution}.
 Li and Zhang obtained the dynamics of threshold solutions for the focusing, $\dot H^1$-critical NLS and NLW
 in the higher dimensions in \cite{LiZh:NLS, LiZh:NLW}. For the more recent progresses on the global dynamics above the ground state energy of NLS , NLW and
NLKG, please refer to \cite{KriNS:e-critical NLW, KriNakSch:1D KG,
NakSch:cubic KG:Rigidity, NakSch:cubic NLS:Rigidity, NakSch:cubic
KG:nonradial, NakSch:KG:hadammeth}.

The paper is organized as follows. The main structure of the paper
is reminiscent of that for the NLS and NLW in
\cite{DuyMerle:NLS:ThresholdSolution,
DuyMerle:NLW:ThresholdSolution, DuyRouden:NLS:ThresholdSolution,
LiZh:NLS, LiZh:NLW}. The new ingredient is that we show that the
positive solution of the nonlocal elliptic equation in
$L^{\frac{2d}{d-2}}(\R^d)$ is regular and unique by the moving plane
method in its global form,  which plays an important role in the
spectral theory of the linearized operator and the dynamics behavior
of the threshold solution.

In Section \ref{S:preli}, we recall the Cauchy problem, the
properties of the ground state. We also state  the spectral
properties of the linearized operator $\mathcal{L}$ around $W$,
which is deduced from the property of the null space of the
linearized operator in $L^2_{rad}$; Under the condition
$E(u_0)=E(W),$ we can identify a quadratic form $B$ associated to
the linearized operator $\mathcal{L}$ and use the property of the
null space of the linearized operator to find two subspaces
$H^{\bot}\cap \dot H^1_{rad}$ and $G_{\bot}\cap \dot H^1_{rad}$ in
$\dot H^1$, where the linearized energy $\Phi$ is positive
(coercive), avoiding the vanishing and negative directions. These
decompositions will play an important role in establishing the
modulational stability in Section \ref{S:modulation} and analyzing
the uniqueness of the exponential decaying solutions to the
linearized equation in Section \ref{S:uniqueness}, respectively.

In Section \ref{S:existence}, we construct two special solutions
$W^{\pm}$ of \eqref{har} except for the negative time behavior by
use of the knowledge about the real eigenvalues of the linearized
operator $\mathcal{L}$ and the fixed point argument, which gives the
proof of Theorem \ref{threholdsolution} except for the negative time
behavior.

In Section \ref{S:modulation}, we make use of the variational
characterization of $W$ and the implicit function theorem to discuss
the modulational stability around $W$, then we make use of the
positivity of the linearized energy $\Phi$ in $H^{\bot}\cap \dot
H^1_{rad}$ to identify the scaling and phase parameters in the
modulational stability which are closely related with the gradient
variant $\delta(t)$ of the solution away from $W$. In particular,
there parameters are linearly dependent of $\delta(t)$ in the
interval with small gradient variant.

In Section \ref{S:convergence:sup} and Section
\ref{S:convergence:sub}, we study the solutions with initial data
satisfying Theorem \ref{classification} part (a) and (c). Main
techniques are to make use of the virial argument and the
concentration compactness argument to obtain the exponential decay
\eqref{gradient:sup:expdecay} and \eqref{expdecay:pointwise} of the
gradient variant $\delta(t)$ for the large (positive) time, which
will imply the exponential convergence in the (positive) time
direction to $W$ (up to scaling and phase rotation) by the
modulational stability, and also obtain the proof of Theorem
\ref{threholdsolution} for the negative time behavior.

In Section \ref{S:uniqueness}, we first use the positivity of the
linearized energy $\Phi$ in $G_{\bot}\cap \dot H^1_{rad}$ to analyze
the property of the exponentially decay solution of the linearized
equation, then apply it to establishing the uniqueness of the
special solutions, this will imply the proof of Theorem
\ref{classification}.

Appendix \ref{S:movplan} contains proofs of the uniqueness of the
ground state in $L^{\frac{2d}{d-2}}$ by the moving plane method in
its global form. Appendix \ref{appen CoerH} and
\ref{appen-spectralprop} contain proofs of the spectral properties
and positivity of the linearized operator in Proposition \ref{coerH}
and Proposition \ref{spectral}.

\subsection*{Acknowledgements.}The authors are partly supported by the NSF
of China (No. 10725102, No. 10801015). The authors would like to
thank Professor W.~Chen, F.~Lin and K.~Nakanishi for their valuable
suggestions.

%
%
%
%

\section{Preliminaries}\label{S:preli}

\subsection{The linear estimates and the Cauchy problem.} In this section, we recall some
results on the Cauchy problem of \eqref{har}. Let $I$ be an
interval, and denote
\begin{align*}
 Z(I):=L^6\left(I; L^{\frac{6d}{3d-8}}({\R^d})\right), \quad S(I):=
L^3 & \left(I;  L^{\frac{6d}{3d-4}}({\R^d})\right), \quad N(I):=
L^{\frac32}\left(I; L^{\frac{6d}{3d+4}}({\R^d})\right),\\
 l(I)=Z(I)\cap L^{3}\left(I;
\dot{W}^{1,\frac{6d}{3d-4}}({\R^d})\right)&,\quad
\big\|u\big\|_{l(I)}:= \big\|u\big\|_{Z(I)}+\big\|\nabla
u\big\|_{S(I)}.
 \end{align*}

A solution of \eqref{har} on an interval $I$ with $0\in I$ is a
function $u\in C^0(I, \dot H^1(\R^d))$ such that $u\in Z(J)$ for all
interval $J\Subset I$ and
\begin{equation*}
\aligned u(t)= e^{it\Delta } u_0 + i \int^t_0 e^{i(t-s)\Delta
}\left( \frac{1}{|\cdot|^{4}}*|u(s,\cdot)|^2\right)(x)\; u(s,x)\;
ds.
\endaligned
\end{equation*}

\begin{lemma}[\cite{Caz:book, tao:book}]\label{Striestimate} Consider
\begin{equation}\label{lin}
\left\{
\aligned & i\partial_t u + \Delta u = f,\quad x\in \R^d,\; t\in [0, T),\\
&u(0)= u_0 \in \dot H^1,
\endaligned\right.
\end{equation}
where $ \nabla  f \in N(0, T)$, then we have
\begin{equation*}
\aligned \sup_{t\in [0,T)}\big\| u\big\|_{\dot H^1} +
\big\|u\big\|_{l(0,T)}   \leq C \left( \big\|u_0\big\|_{\dot H^1} +
\big\|\nabla f\big\|_{N(0,T)}\right).
\endaligned
\end{equation*}
\end{lemma}

\begin{lemma}[\cite{Gra04:book, Ste:70:book}]\label{L:hardy}
For $\alpha \in (0, d)$,  there exists a constant $C(d,\alpha) $
such that for any $r\in (\frac{d}{d-\alpha}, \infty)$,
\begin{equation*}
\left\|\int_{\R^n}\frac{f(y)}{|x-y|^{d-\alpha}} \; dy
\right\|_{L^r(\R^d)} \leq C(d,\alpha)
\big\|f\big\|_{L^{\frac{dr}{d+\alpha r}}(\R^d)}.
\end{equation*}
\end{lemma}

\begin{lemma}[\cite{Bahouri-gerard, GMO, KillipVisan:clay}]\label{L:sobolev}
For every $f\in\dot H^1(\R^d)$, there exists a constant $C$ such
that
\begin{equation*}
\big\|f\big\|_{L^{\frac{2d}{d-2}}(\R^d)} \leq C
\big\|f\big\|^{\frac{d-2}{d}}_{\dot H^1(\R^d)}
\big\|f\big\|^{\frac{2}{d}}_{\dot B^1_{2,\infty}(\R^d)}.
\end{equation*}
\end{lemma}

\begin{theorem}[\cite{Caz:book, MiaoXZ:08:LWP for Har}]\label{T:local}
For any $u_0 \in \dot H^1(\R^d)$ and $t_0\in \R$, there exists a
unique maximal-lifespan solution $u :I\times \R^d \rightarrow \C$ to
\eqref{har} with   $u(t_0)=u_0$. This solution also has the
following properties:

\begin{enumerate}
\item[\rm(a)]   $I$ is an open neighborhood of $t_0$.
\item[\rm(b)]  The energy of the solution $u$ are conserved,
that is, for all $t\in I$, we have
\begin{equation*}
\aligned E(u(t))=E(u_0).
\endaligned
\end{equation*}
\item[\rm(c)]   If $u^{(n)}_0$ is a
sequence converging to $u_0$ in $\dot H^1 $ and
$u^{(n)}:I^{(n)}\times \R^d \rightarrow \C$ are the associated
maximal-lifespan solutions, then $u^{(n)}$ converges locally
uniformly to $u$.
\item[\rm(d)]   There exists $\eta_0$, such that if $\big\|u_0\big\|_{\dot H^1}<\eta_0$,
 then $u$ is a global solution.  Indeed, the solution also scatters to $0$
in $\dot H^1$.
\end{enumerate}
\end{theorem}

\subsection{Properties of the ground state.}

Consider the nonlocal elliptic equation
\begin{equation}\label{gs}
\aligned -\Delta W (x)= \int_{\R^d} \frac{|W(y)|^2}{|x-y|^4} \; dy
\; W(x), \quad x\in \R^d.
\endaligned
\end{equation}
\eqref{w function} is an explicit solution of \eqref{gs}. Using the
moving plane method in its global form, we will show that $W$ is the
only positive solution of \eqref{gs} in $L^{\frac{2d}{d-2}}(\R^d)$
in Appendix \ref{S:movplan}, up to the symmetries of \eqref{gs}. Moreover, the uniqueness still holds for the positive solution in
$L^{\frac{2d}{d-2}}_{loc}(\R^d)$. Hence, we have

\begin{lemma}\label{GS:unique} The elliptic equation
\eqref{gs} has a unique positive, radial decreasing solution
$W(x)=W(|x|)$ in $L^{\frac{2d}{d-2}}(\R^d)$, up to the spatial
translation,  scaling and the Kelvin transform.
\end{lemma}

Combining the sharp Sobolev inequality \cite{Aubin, LiebL:book,
Talenti:best constant} with the sharp Hardy-Littlewood-Sobolev
inequality \cite{Lieb:sharp constant for HLS, LiebL:book}, we have
the following variational characterization of $W$.
\begin{lemma}[\cite{LiebL:book}]\label{SharpConstant} For any $\omega \in \dot H^1(\R^d), $ we have
\begin{align*}
 \left(\iint_{\R^d\times \R^d}
\frac{1}{|x-y|^{4}} \big|\omega(x)\big|^2 \big|\omega(y)\big|^2\
dxdy \right)^{1/4}\leq & \; C_* \big\|\nabla \omega \big\|_2,
\end{align*}
where $C_*=C_*(d)$ is the best constant. Moreover if
\begin{equation*}\aligned \left(\iint_{\R^d\times \R^d} \frac{1}{|x-y|^{4}}
\big|\omega(x)\big|^2 \big|\omega(y)\big|^2\ dxdy \right)^{1/4} = \;
C_* \big\|\nabla u \big\|_2,  \endaligned
\end{equation*}
then there exist $\lambda_0>0, x_0 \in \R^d, \theta_0 \in [0,
2\pi),$ such that
\begin{equation*}
\aligned \omega (x)= e^{i\theta_0} \lambda^{-\frac{d-2}{2}}_0W\left(
\lambda^{-1}_0(x+x_0)\right).
\endaligned
\end{equation*}
\end{lemma}

From \eqref{gs} and Lemma \ref{SharpConstant}, we have
\begin{equation}\label{KinEnergy}
\aligned \big\|\nabla W \big\|^2_{2}=C^{-4}_*, \quad E(W)=C^{-4}_*/4
.
\endaligned
\end{equation}

Using the characterization of $W$ in Lemma \ref{SharpConstant}, the
refined Sobolev inequality in Lemma \ref{L:sobolev} and the similar
concentration compactness principle (profile decomposition in $\dot
H^1_{rad}$) to the proof of Proposition 3.1 in
\cite{MiaoXZ:10:blowup}, we can show that
\begin{proposition}\label{P:static stability} Let  $u \in \dot H^1(\R^d)$ be radial and $E(u) = E (W)$.
Then there exists a function $\varepsilon=\varepsilon (\rho)$, such that
\begin{equation*}
\aligned  \inf_{ \theta \in \R,\; \mu>0}  \big\| u_{\theta, \mu} -W
\big\|_{\dot H^1} \leq \varepsilon (\delta(u)), \quad \lim_{\rho
\rightarrow 0}\varepsilon (\rho)=0,
\endaligned
\end{equation*}
where $u_{\theta, \mu}(x) = e^{i\theta} \mu^{-\frac{d-2}{2}}u(
\mu^{-1} x)$, and $\displaystyle \delta(u)=\Big| \int_{\R^d} \left(
\big|\nabla u \big|^2 -\big|\nabla W \big|^2 \right) \; dx \Big| $.
\end{proposition}

\begin{proof} Suppose the contrary holds. Let $u_n \in \dot H^1$ be any radial sequence and satisfy
\begin{equation}\label{kinetic:converg}
E(u_n)=E(W), \; \text{and}\; \Big| \int_{\R^d} \left( \big|\nabla
u_n \big|^2 -\big|\nabla W \big|^2 \right) \; dx \Big|
\longrightarrow 0.
\end{equation}
Hence, there exists a subsequence of $\{u_n\}^{\infty}_{n=1}$ and a
sequence $\{U^{(j)}\}_{j\geq 1} $ in $\dot H^1_{rad}$ and for any
$j\geq 1$, a family $\{\lambda^j_n\}$ such that
\begin{enumerate}

\item If $j\not= k$, we have the asymptotic orthogonality between  $\{\lambda^j_n\}$ and $\{\lambda^k_n\}$, i.e.
\begin{align*}
\frac{\lambda^j_n}{\lambda^k_n} + \frac{\lambda^k_n}{\lambda^j_n}
\longrightarrow +\infty, \quad \text{as}\; n\rightarrow +\infty.
\end{align*}
\item For every $l\geq 1$, we have
\begin{align*}
u_n(x) = \sum^l_{j=1} \frac{1}{\left(  \lambda^j_n
\right)^{(d-2)/2}} U^{(j)} \left( \frac{x}{\lambda^j_n}\right) +
r^l_n(x), \; \text{with}\; \lim_{l\rightarrow
+\infty}\limsup_{n\rightarrow+\infty} \big\|r^l_n\big\|_{\dot
B^1_{2, \infty}} \rightarrow 0.
\end{align*}
\item We have
\begin{align*}
\big\|\nabla u_n\big\|^2_{L^2} = \sum^{l}_{j=1} \big\|\nabla
U^{(j)}_n\big\|^2_{L^2} + \big\|\nabla r^l_n\big\|^2_{L^2} + o_n(1).
\end{align*}
\end{enumerate}
By Lemma \ref{L:hardy}, Lemma \ref{L:sobolev} and the orthogonality
between $\lambda^j_n$ and $\lambda^k_n$ for $j\not = k$, we have
\begin{align*}
\lim_{l\rightarrow+\infty}\limsup_{n\rightarrow +\infty}
\iint_{\R^d\times\R^d}&
\frac{\left|r^l_n(x)\right|^2\left|r^l_n(y)\right|^2}{|x-y|^4} \;
dxdy = 0, \\
\iint_{\R^d\times\R^d} \frac{\left| u_n (x)\right|^2\left| u_n
(y)\right|^2}{|x-y|^4} \; dxdy = &
\lim_{j\rightarrow+\infty}\sum^l_{j=1} \iint_{\R^d\times\R^d}
\frac{\left| U^{(j)}  (x)\right|^2\left| U^{(j)}
(y)\right|^2}{|x-y|^4} \; dxdy .
\end{align*}

By Lemma \ref{SharpConstant}, we have
\begin{align*}
C^{-4}_* \iint_{\R^d\times\R^d} \frac{\left| U^{(j)}
(x)\right|^2\left| U^{(j)} (y)\right|^2}{|x-y|^4} \; dxdy  \leq
\big\|\nabla U^{(j)}\big\|^4_{L^2},
\end{align*}
thus,
\begin{align*}
C^{-4}_* \sum^{l}_{j=1}\iint_{\R^d\times\R^d} \frac{\left| U^{(j)}
(x)\right|^2\left| U^{(j)} (y)\right|^2}{|x-y|^4} \; dxdy  \leq
\sum^{l}_{j=1} \big\|\nabla U^{(j)}\big\|^4_{L^2}.
\end{align*}
This yields that
\begin{align*}
C^{-4}_*  \leq &
\lim_{n\rightarrow+\infty}\limsup_{l\rightarrow+\infty}\frac{
\displaystyle \sum^{l}_{j=1} \big\|\nabla
U^{(j)}\big\|^4_{L^2}}{\displaystyle \sum^{l}_{j=1}
\iint_{\R^d\times\R^d} \frac{\left| U^{(j)} (x)\right|^2\left|
U^{(j)} (y)\right|^2}{|x-y|^4} \; dxdy} \\
\leq  & \lim_{n\rightarrow+\infty}\limsup_{l\rightarrow+\infty}
\frac{ \displaystyle \left(\sum^{l}_{j=1} \big\|\nabla
U^{(j)}\big\|^2_{L^2}\right)^2}{\displaystyle \sum^{l}_{j=1}
\iint_{\R^d\times\R^d} \frac{\left| U^{(j)} (x)\right|^2\left|
U^{(j)} (y)\right|^2}{|x-y|^4} \; dxdy}\\
\leq & \lim_{n\rightarrow+\infty} \frac{ \displaystyle \big\|\nabla
u_n \big\|^4_{L^2} }{\displaystyle \iint_{\R^d\times\R^d}
\frac{\left| u_n (x)\right|^2\left| u_n (y)\right|^2}{|x-y|^4} \;
dxdy} = C^{-4}_*,
\end{align*}
where we use \eqref{kinetic:converg}   in the last step. Therefore,
we conclude that only one profile  $U^{(j)}$ is nonzero, and
\begin{align*}
C^{-4}_*  = \frac{ \displaystyle \big\|\nabla
U^{(j_0)}\big\|^4_{L^2}}{\displaystyle \iint_{\R^d\times\R^d}
\frac{\left| U^{(j_0)} (x)\right|^2\left| U^{(j_0)}
(y)\right|^2}{|x-y|^4} \; dxdy }.
\end{align*}
By Lemma \ref{SharpConstant}, there exist $\theta_0\in [0, 2\pi)$,
$\lambda_0>0$ such that
\begin{equation*}
U^{(j_0)}(x) = e^{i\theta_0} \lambda^{-\frac{d-2}{2}}_0W (
\lambda^{-1}_0 x ).
\end{equation*}
This contradicts the assumption and we completes the proof.
\end{proof}

\subsection{The gradient separation.}
By the convex analysis in \cite{MiaoXZ:09:e-critical radial Har}, we
first have
\begin{align}
\label{convexity} \forall u \in \dot H^1 , \;\; E(u)\leq E(W), \;\;
\big\|\nabla u \big\|_2 \leq \big\| \nabla W \big\|_2
\Longrightarrow \frac{\big\|\nabla u \big\|^2_2}{\big\|\nabla W
\big\|^2_2} \leq \frac{E(u)}{E(W)},\\
\forall u \in \dot H^1 , \;\; E(u)\leq E(W), \;\; \big\|\nabla u
\big\|_2 \geq \big\| \nabla W \big\|_2 \Longrightarrow
\frac{\big\|\nabla u \big\|^2_2}{\big\|\nabla W \big\|^2_2} \geq
\frac{E(u)}{E(W)}. \notag
\end{align}
This, together with  the energy conservation, the variational
characterization of $W$ and the continuity argument, implies that
\begin{lemma}\label{energytrapp}
Let $u \in \dot H^1(\R^d)$ be a radial solution of \eqref{har} with
initial data $u_0$, and $I=(-T_-, T_+)$ its maximal interval of
existence. Assume that $ E(u_0)=E(W)$,   then
\begin{enumerate}
\item[\rm (a)] if $\big\|\nabla u_0\big\|_{2}< \big\|\nabla W\big\|_{2}$,
then $\big\|\nabla u(t)\big\|_{2}< \big\|\nabla W\big\|_{2}$ for
$t\in I$.
\item[\rm (b)]  if $\big\|\nabla u_0\big\|_{2} = \big\|\nabla W\big\|_{2}$,
then $u =W $ up to the symmetry of the equation.
\item[\rm (c)] if $\big\|\nabla u_0\big\|_{2} >  \big\|\nabla W\big\|_{2}$,
then $\big\|\nabla u(t)\big\|_{2} > \big\|\nabla W\big\|_{2}$ for $
t\in I$.
\end{enumerate}
\end{lemma}
\begin{proof}The proof is analogue to  that of Proposition 3.1 in \cite{MiaoXZ:09:e-critical radial Har}.
\end{proof}

\subsection{Monotonicity formula}
Let $\phi(x)$ be a smooth radial function such that
$\phi(x)=|x|^2$ for $|x|\leq 1$ and $\phi(x)=0$ for $|x|\geq 2$ .
For $R>0$, define
\begin{equation}\label{localV}
\aligned V_R(t)=\int_{\R^d} \phi_R(x)\big|u(t,x)\big|^2 dx,\;
\text{where}\; \phi_R(x)=R^2\phi\left(\frac{x}{R}\right).
\endaligned
\end{equation}

\begin{lemma}[\cite{LiMZ:e-critical Har, MiaoXZ:09:e-critical radial Har}]\label{L:local virial}Let $u(t,x)$ be a radial solution to \eqref{har}, $V_R(t)$ be
defined by \eqref{localV}, then
\begin{align*}
\partial_t V_R(t)=&\;  2\Im \int_{\R^d} \overline{u}\; \nabla u \cdot \nabla
\phi_R\; dx, \\
\partial^2_t V_R(t)=&\; 8 \int_{\R^d} \Big|\nabla u (t,x)\Big|^2 dx -8
\iint_{\R^d\times\R^d} \frac{|u(t,x)|^2|u(t,y)|^2}{|x-y|^4} \; dxdy
+ A_R\big(u(t)\big),
\end{align*}
where
\begin{align*} A_R\big(u(t)\big)=&\;    \int_{\R^d} \left( 4\phi''\Big(\frac{|x| }{R }\Big)-8 \right) \big|\nabla u (t,x) \big|^2 dx
  +  \int_{\R^d} \big(-\Delta \Delta \phi_R(x) \big) \big| u(t,x)\big|^2 dx \\
& + 8 \iint_{\R^d\times\R^d} \left(1-\frac12
\frac{R}{|x|}\phi'\big(\frac{|x| }{R }\big) \right) \;
\frac{x(x-y)}{|x-y|^6}\;|u(t,x)|^2|u(t,y)|^2 \;dxdy \\
& - 8 \iint_{\R^d\times\R^d} \left(1-\frac12
\frac{R}{|y|}\phi'\big(\frac{|y| }{R }\big) \right) \;
\frac{y(x-y)}{|x-y|^6}\;|u(t,x)|^2|u(t,y)|^2 \;dxdy.
\end{align*}
\end{lemma}

\subsection{Preliminary properties of the linearized operator.} We consider a radial solution $u$ of \eqref{har} close to $W$ and
write $u$ as
\begin{equation*}
\aligned u(t,x)= W(x)+h(t,x),
\endaligned
\end{equation*}
then $h$ satisfies that
\begin{align}
i\partial_t h+\Delta h= V h+  i R(h ), \nonumber
\end{align}
where the linear operator $V$ and the remainder $R(h)$ are defined
by
\begin{align}
\label{linearterm} Vh: = &\; - \Big( \frac{1}{|x|^{4}}*|W|^2 \Big) h
- 2\Big(\frac{1}{|x|^4}*\big(W\Re h\big) \Big) W ,\\
\label{remainder} R(h):=&\; i \Big(|\cdot|^{-4}*|h|^2\Big)(W+h) + 2i
\Big(|\cdot|^{-4}*(W\Re h)\Big) h .
\end{align}

In the following context, we will
always denote the complex value function $h=h_1+ih_2=(h_1, h_2)$
without confusion. Let $h_1=\Re h$ and $h_2=\Im h$.
 Then $h$ is a solution of the equation
\begin{equation}\label{linearequat}
\aligned
\partial_t h + \mathcal{L} h=R(h), \quad \mathcal{L}:=\begin{pmatrix} 0 &  -L_-\\
L_+  & 0 \end{pmatrix},
\endaligned
\end{equation}
where the self-adjoint operators $L_{\pm}$ are defined by
\begin{align}
 L_+ h_1:=& -\Delta h_1 -\Big(\frac{1}{|x|^{4}}*|W|^2\Big) h_1 - 2\Big(\frac{1}{|x|^{4}}* (Wh_1)\Big) W, \label{Lpos}\\
   L_- h_2:= &-\Delta h_2-\Big(\frac{1}{|x|^{4}}*|W|^2\Big) h_2.
   \label{Lneg}
\end{align}

Now we first give some preliminary estimates about the linearized
equation \eqref{linearequat}.

\begin{lemma}\label{linearoperator:prelimestimate}
Let $V$, $R$ be defined by \eqref{linearterm} and \eqref{remainder},
respectively, $I$ be a interval with $|I|\leq 1$, and $g,h \in
l(I)$, $u, v \in L^{\frac{2d}{d-2}}(\R^d)$, then
\begin{align}
\big\|\nabla \big(V h\big) \big\|_{N(I)} \lesssim &\; |I|^{\frac13}
\big\|h\big\|_{l(I)},\label{linearestimate}
\\
 \big\|\nabla \big( R(g)-R(h)\big)
\big\|_{N(I)}\lesssim & \; \big\|g-h\big\|_{l(I)} \Big(
 |I|^{\frac16} \big(\big\|g\big\|_{l(I)} + \big\|h\big\|_{l(I)}\big)
+ \big\|g\big\|^2_{l(I)} +
\big\|h\big\|^2_{l(I)}\Big),\label{nonlinearestimates}\\
\big\|  R(u)-R(v) \big\|_{L^{\frac{2d}{d+2}}(\R^d)}\lesssim & \;
\Big(\big\|u\big\|_{L^{\frac{2d}{d-2}}} +
\big\|v\big\|_{L^{\frac{2d}{d-2}}}+
\big\|u\big\|^2_{L^{\frac{2d}{d-2}}} +
\big\|v\big\|^2_{L^{\frac{2d}{d-2}}}\Big)\big\|u-v\big\|_{L^{\frac{2d}{d-2}}}.
\label{nonlinearestimate:dual}
\end{align}
\end{lemma}
\begin{proof} By Lemma \ref{L:hardy} and
H\"{o}lder's inequality, we have
\begin{equation*}
\aligned \big\|\nabla \big( V h \big) \big\|_{N(I)} \lesssim
\big\|W\big\|^2_{Z(I)} \big\|\nabla h\big\|_{S(I)} +
\big\|W\big\|_{Z(I)}\big\|\nabla W \big\|_{S(I)}
\big\|h\big\|_{Z(I)}.
\endaligned
\end{equation*}
Since $W \in L^{\frac{6d}{3d-8}}_x \cap \dot W^{1,
\frac{6d}{3d-4}}_x,$ we have that
\begin{equation*}
\aligned \big\|W\big\|_{Z(I)} \lesssim |I|^{\frac16}, \;\;
\big\|\nabla W \big\|_{S(I)}\lesssim |I|^{\frac13}.
\endaligned
\end{equation*}
This  implies \eqref{linearestimate}. The analogue argument gives
that
\begin{align*}
& \big\|\nabla \big( R(g)-R(h)\big) \big\|_{N(I)}\\
\lesssim  &\;  \big\|g-h\big\|_{Z(I)} \Big( \big\| \nabla
(g+h)\big\|_{S(I)} |I|^{\frac16} + \big\|  g+h \big\|_{Z(I)}
|I|^{\frac13} \Big) \\
& + \big\|\nabla (g-h) \big\|_{S(I)}  \big\|
g+h\big\|_{Z(I)} |I|^{\frac16} \\
& + \big\|g-h\big\|_{Z(I)} \Big( \big\|\nabla g\big\|_{S(I)}
\big\|g\big\|_{Z(I)} + \big\|\nabla ( g+h ) \big\|_{S(I)}
\big\|h\big\|_{Z(I)} \Big) \\
& + \big\|\nabla (g-h) \big\|_{S(I)} \Big( \big\|g\big\|^2_{Z(I)} +
\big\|g+h\big\|_{Z(I)}\big\|h\big\|_{Z(I)} \Big) \\
& + \big\|g-h\big\|_{Z(I)} \Big( |I|^{\frac13}
\big(\big\|g\big\|_{Z(I)}+   \big\|h\big\|_{Z(I)}\big)+
|I|^{\frac16} \big(\big\|\nabla g\big\|_{S(I)} +
 \big\|\nabla h\big\|_{S(I)}\big)\Big)\\
& + \big\|\nabla (g-h) \big\|_{S(I)}  \big\|
g \big\|_{Z(I)} |I|^{\frac16} \\
\lesssim & \; \big\|g-h\big\|_{l(I)} \Big(
\big(|I|^{\frac16}+|I|^{\frac13}\big)\big(\big\|g\big\|_{l(I)} +
\big\|h\big\|_{l(I)}\big)  + \big\|g\big\|^2_{l(I)} +
\big\|h\big\|^2_{l(I)}\Big).
\end{align*}
In addition, \eqref{nonlinearestimate:dual} holds by Lemma
\ref{L:hardy}. This completes the proof. \end{proof}

Due to the emergence of the linear operator $V$ in the
linearized equation \eqref{linearequat}, we will often use the
following integral summation argument.

\begin{lemma}[\cite{DuyMerle:NLS:ThresholdSolution}]
\label{summation} Let $t_0>0$, $p\in [1,+\infty[$, $a_0 \neq 0$, $E$
a normed vector space, and $f\in L^p_{\rm loc}(t_0,+\infty;E)$ such
that
\begin{equation}
\label{small.tau} \exists \tau_0>0,\; \exists C_0>0, \;\forall\;
t\geq t_0, \quad \|f\|_{L^p(t,t+\tau_0,E)}\leq C_0e^{a_0 t}.
\end{equation}
Then for $t\geq t_0$,
\begin{align}
 \|f\|_{L^p(t,+\infty,E)}\leq &\;
  \frac{C_0e^{a_0 t}}{1-e^{a_0\tau_0}},\;  \text{if
}\; a_0<0; \label{conclu.summation} \\
 \|f\|_{L^p(t_0, t,E)} \leq
&\; \frac{C_0e^{a_0 t}}{1-e^{-a_0\tau_0}},\;  \text{if }\; a_0>0.
\label{conclu.summationless}
 \end{align}
\end{lemma}

By the Strichartz estimate, Lemma
\ref{linearoperator:prelimestimate} and Lemma\ref{summation}, we
have
\begin{lemma}\label{linearstrich}
Let $v$ be a solution of \eqref{linearequat} with
\begin{equation*}
\big\| v(t) \big\|_{\dot H^1} \leq C e^{-c_1 t}
\end{equation*}
for some $C$ and $c_1>0$, then for any admissible pair $(q,r)$, i.e.
$$\frac2q=d
\Big(\frac12 -\frac1r\Big),\qquad  q\in [2, +\infty],$$ we have for large $t$
\begin{align*}
\big\|v\big\|_{l(t,+\infty)}+\big\|\nabla v\big\|_{L^q(t,+\infty;
L^r)} \leq C e^{-c_1t}.
\end{align*}
\end{lemma}

\begin{proof}For small $\tau_0$, by the Strichartz estimate and
Lemma \ref{linearoperator:prelimestimate}, we have on $I=[t,
t+\tau_0]$
\begin{align*}
\big\|v\big\|_{l(I)}+\big\|\nabla v\big\|_{L^q(I; L^r)} \leq & C
e^{-c_1 t } + C\big\|\nabla (V h) \big\|_{N(I)} +  C\big\|\nabla R(
h) \big\|_{N(I)} \\
\leq & C e^{-c_1 t } + C |I|^{\frac{1}{3}} \big\|h\big\|_{l(I)} +
\big\|h\big\|^2_{l(I)} + \big\|h\big\|^3_{l(I)}.
\end{align*}
By choosing sufficiently small $\tau_0$, the continuous argument
gives that
\begin{equation*}
\big\|v\big\|_{l(I)}+\big\|\nabla v\big\|_{L^q(I; L^r)} \leq C
e^{-c_1 t }.
\end{equation*}
This implies the desired result by Lemma \ref{summation}.
\end{proof}

\subsection{Spectral properties of the linearized operator.}
Due to the symmetries of \eqref{har} under the phase rotation and
the scaling, we know that the elements
$$iW \in L^2(\R^d), \quad \widetilde{W}=\dfrac{d-2}{2}W+x\cdot \nabla W \in L^2(\R^d)$$
belongs to the null-space of $\mathcal{L}$ in $L^2_{rad}$. Indeed,
they are the only elements of the null-space of $\mathcal{L}$ in
$L^2_{rad}$.

\begin{lemma}\label{keyassumption}Let $\mathcal{L}$ be defined by
\eqref{linearequat}. Then
\begin{align*}
\big\{u\in   L^2_{rad}(\R^d),
\mathcal{L}u=0\big\}=\text{span}\big\{iW, \widetilde{W}
 \big\}.
\end{align*}Namely,
\begin{equation*}
\big\{u\in   L^2_{rad}(\R^d), L_-u=0\big\}=\text{span}\big\{W
 \big\};\quad \big\{u\in   L^2_{rad}(\R^d),  L_+u=0\big\}=\text{span}\big\{\widetilde{W}
 \big\}.
\end{equation*}
\end{lemma}

\begin{proof} We prove it by the oscillation properties of
Sturm-Liouville eigenvalue problems. Note that  $L_-W=0$ and
 the ground state $W$ is non-degenerate (Lemma \ref{GS:unique}),
we know that $ \big\{u\in   L^2_{rad},
L_-u=0\big\}=\text{span}\big\{W
 \big\}.
$ Hence it suffices to show that\begin{equation*}  \big\{u\in
L^2_{rad}, L_+u=0\big\}=\text{span}\big\{\widetilde{W}
 \big\}.
\end{equation*}For the radial function, by the spherical coordinates (see \cite{Gra04:book, SteinW:book:Fourier
analysis}), $L_+v=0$ can be written as
\begin{equation*}
\aligned A_0 v(r) = -\frac{d^2}{dr^2}  v - \frac{d-1}{r}\frac{d}{dr}
v - & \int^{\infty}_0    \rho^{d-5} V\left(\frac{r}{\rho}\right)
W(\rho) ^2 \ d\rho \;v(r) \\
-&  2  \int^{\infty}_0 \rho^{d-5}V\left(\frac{r}{\rho}\right)
W(\rho)v(\rho) \ d\rho \; W(r),
\endaligned
\end{equation*}where
\begin{equation*}
\aligned V(\rho)=\int_{\mathbb{S}^{d-1}} \frac{1}{|\rho - y'|^4}
d\sigma_{y'} = \omega_{d-2} \int^{1}_{-1} \frac{\left(
1-s^2\right)^{\frac{d-3}{2}}}{\left( \rho^2 -2\rho s  +1 \right)^2}
ds,
\endaligned
\end{equation*}
and $\mathbb{S}^{d-1}$ denotes the unit sphere in $\R^d$ and
$\omega_{d-2}$ denotes the area of the unit sphere
$\mathbb{S}^{d-2}$ in $\R^{d-1}$.

From  $ \left( W, A_0 W\right) < 0$,  we know that the first
eigenvalue is negative.  By the nonnegative of $L_+$ on $\{\Delta
W\}^{\bot}$ in Step 1 of Appendix \ref{appen CoerH} and the
Courant's min-max principle \cite{LiebL:book}, we conclude that the
second eigenvalue is nonnegative. Note that $ A_0 \widetilde{W}(r)
=0 $ and $\widetilde{W} \in L^2$ have only one positive zero, we
know that $0$ is the second eigenvalue, and by the Sturm-Liouville
theory, we conclude the desired result.
\end{proof}

Now we define the linearized energy $\Phi$ by
\begin{align}\label{linearized energy}
 \Phi(h):=& \frac12 \int_{\R^d} (L_+ h_1) h_1 \; dx + \frac12
 \int_{\R^d}
(L_- h_2) h_2 \; dx, \\
= & \frac12 \int_{\R^d} \big| \nabla h_1 \big| ^2 - \frac12
\iint_{\R^d\times \R^d} \frac{|W(x)|^2|h_1(y)|^2}{|x-y|^4} -
\iint_{\R^d\times \R^d}
\frac{W(x)h_1(x)W(y)h_1(y)}{|x-y|^4} \nonumber\\
& + \frac12 \int_{\R^d} \big| \nabla h_2 \big| ^2 - \frac12
\iint_{\R^d\times\R^d} \frac{|W(x)|^2|h_2(y)|^2}{|x-y|^4} .
\nonumber
 \end{align}
We denote by $B(g,h)$ the bilinear symmetric form associated to
$\Phi$,
\begin{equation}\label{bilinear form}
\aligned B(g,h):=\frac12 \int_{\R^d} (L_+ g_1) h_1 \; dx +
\frac12\int_{\R^d} (L_- g_2) h_2 \; dx, \;\; \forall \; g,h \in \dot
H^1(\R^d).
\endaligned
\end{equation}

Let us specify the important coercivity properties of $\Phi$.
Consider the three orthogonal directions $W, iW,$ and
$\widetilde{W}$ in the Hilbert space $\dot H^1(\R^d, \C)$. Let
$H:=\text{span} \{W, iW, \widetilde{W}\},$ and
\begin{equation*}\label{H orth}
\aligned H^{\bot} := \left\{v\in \dot H^1, (iW, v)_{\dot
H^1}=(\widetilde{W}, v)_{\dot H^1}  = (W, v)_{\dot H^1}=0 \right\}
\endaligned
\end{equation*}
denotes the orthogonal subspace of $H$ in $\dot{H}^1$.

\begin{proposition}\label{coerH}There exists a constant $c>0$ such
that for all radial function $ h \in H^{\bot}$, we have
\begin{equation}
\aligned  \Phi(h)\geq c\big\| h \big\|^2_{\dot H^1}.
\endaligned
\end{equation}
\end{proposition}

\begin{proposition}\label{spectral}Let $\sigma(\mathcal{L})$ be the spectrum of the operator
$\mathcal{L}$, defined on $L^2(\R^d)$ with domain $ H^2$ and let
$\sigma_{ess}(\mathcal{L})$ be its essential spectrum. Then we have
\begin{enumerate}
\item[\rm (a)]The operator $\mathcal{L}$ admits two radial
eigenfunctions $\mathcal{Y}_{\pm} \in \mathcal{S}(\R^d) $ with real
eigenvalues $\pm e_0$, $e_0
> 0$, that is, $ \mathcal{L} \mathcal{Y}_{\pm} = \pm e_0
\mathcal{Y}_{\pm},\; \mathcal{Y}_+ = \overline{\mathcal{Y}}_- ,\;
e_0
> 0.
$

\item[\rm (b)] There exists a constant $c>0$ such that  for all radial function $ h
\in G_{\bot}$, we have
\begin{equation*}
\aligned  \Phi(h)\geq c\big\| h \big\|^2_{\dot H^1},
\endaligned
\end{equation*}
where \begin{equation*}
 \label{G orthy} G_{\bot}  =\left\{v\in \dot H^1, (iW, v)_{\dot
H^1}=(\widetilde{W}, v)_{\dot H^1}  = B(\mathcal{Y}_\pm, v)  =
0\right\}.
\end{equation*}

\item[\rm (c)]
$\sigma_{ess}(\mathcal{L})=\{i\xi: \xi\in \R, \}, \quad
\sigma(\mathcal{L})\cap \R =\{-e_0, 0, e_0\}. $
\end{enumerate}
\end{proposition}

The proof of Proposition \ref{coerH} and Proposition \ref{spectral}
are analogue to that of Claim 3.5,  Lemma 5.1 and Corollary 5.3 in
\cite{DuyMerle:NLS:ThresholdSolution}. For the sake of completeness,
We prove it in Appendix \ref{appen CoerH} and Appendix
\ref{appen-spectralprop}.

\begin{remark}\label{r:bilinear form prop}
As a consequence of the definition of $\Phi$ and $B$, we have for
any $ f, g \in \dot H^1, \mathcal{L}f, \mathcal{L}g \in \dot H^1 $,
and $ h\in H^{\bot}$
\begin{align}
\Phi(iW)=\Phi(\widetilde{W})=&\Phi(\mathcal{Y}_\pm)=0, \;\; \Phi(W)=-\big\|\nabla W\big\|^2_2<0,  \label{Phi}\\
B(f,g)=&B(g,f), \;\; B(\mathcal{L}f, g)=-B(f,\mathcal{L}g), \label{B-pro}\\
B(iW, f)=& B(\widetilde{W}, f)=0,\;\;  B(W, h)=0
.\label{B-orthEnerHbot}
\end{align}
\end{remark}

\begin{corollary}
As a consequence of  Proposition \ref{coerH}, we have
\begin{align}
B(\mathcal{Y}_+, \mathcal{Y}_-)\not = 0. \label{B-orthEigen}
 \end{align}
\end{corollary}
\begin{proof} If $B(\mathcal{Y}_+, \mathcal{Y}_-) = 0$, then for any
$h\in \text{span} \{iW, \widetilde{W}, \mathcal{Y}_{\pm}\}$, which
is of dimension 4, we have
$$
\Phi(h)=0.
$$
But by Proposition \ref{coerH}, we know that $\Phi$ is positive on
$H^{\bot} \cap \dot H^1_{rad}$, which is of codimension 3. It is a
contradiction.
\end{proof}

\begin{remark}
As a consequence of Proposition \ref{spectral}, we have Lemma
\ref{keyassumption}. In fact, it suffices to show the inclusion
relation ``$\subseteq$''. If the dimension of $\big\{u\in L^2_{rad},
\mathcal{L}u=0\big\}$ was strictly higher than two, we could find
$\ZZZ \in L^2_{rad}$, $\ZZZ\not =0$, such that $\mathcal{L}\ZZZ=0$.
Let \begin{equation*}\ZZZ'=\ZZZ- (\ZZZ, iW)_{\dot H^1} \frac{iW}
{\big\|W\big\|^2_{\dot H^1}} - (\ZZZ, \widetilde{W})_{\dot H^1}
\frac{\widetilde{W}}{\big\|\widetilde{W}\big\|^2_{\dot H^1}} \not
=0,
\end{equation*} which implies that $\mathcal{L}\ZZZ'=0$ and
\begin{align*}
(\ZZZ', iW)_{\dot H^1}& =0,\;  (\ZZZ', \widetilde{W})_{\dot H^1}=0,
\;  B(\YYY_{\pm}, \ZZZ')  =\pm\frac{ 1}{e_0}
B(\mathcal{L}\YYY_{\pm}, \ZZZ') =\mp B(\YYY_{\pm}, \mathcal{L}
\ZZZ')=0.
\end{align*}
Thus $\ZZZ'\in G_{\bot} \backslash \{0\}$. While $ \Phi(\ZZZ')=
(\mathcal{L}\ZZZ', \ZZZ')_{L^2}=0, $ which contradicts the
coercivity of $\Phi$ on $G_{\bot} \cap \dot H^1_{rad}$.
\end{remark}

\begin{corollary}As a consequence of Proposition
\ref{spectral}, we have
\begin{equation}\label{WYrelation}
\aligned \int \nabla W \cdot \nabla \mathcal{Y}_1 \; dx \not =0,
\quad \text{where}\; \mathcal{Y}_1=\Re \mathcal{Y}_+.
\endaligned
\end{equation}
\end{corollary}
\begin{proof}
On one hand, we have
$$(W, iW)_{\dot H^1}=(W, \widetilde{W})_{\dot H^1}=0.$$
On the other hand, by \eqref{gs}, we know $ L_+(W)= 2 \Delta W$. If
$ (W, \mathcal{Y}_1)_{\dot H^1} =0,$ then by \eqref{B-pro}, we
obtain
\begin{equation*}
\aligned e_0 B(\mathcal{Y}_{\pm}, W) = &  B(\mathcal{L}
\mathcal{Y}_{\pm}, W) = - B(\mathcal{Y}_{\pm}, \mathcal{L} W) \\
=& -\frac12 \int L_{-}\mathcal{Y}_2 \cdot L_+(W) = \frac12 \int e_0
\mathcal{Y}_1 \cdot  L_+(W)  =  \int e_0 \mathcal{Y}_1 \cdot \Delta
W =0,
\endaligned
\end{equation*}
which means that $W \in G_{\bot}$. By Proposition \ref{spectral}, we
have $\Phi(W) \gtrsim \big\|W \big\|^2_{\dot H^1}$. This contradicts
the fact \eqref{Phi}. This completes the proof.
\end{proof}

%
%
%
%

\section{Existence of special solutions}\label{S:existence}

In this section, we first construct a family of approximate
solutions to \eqref{har} by making use of the knowledge about the real
eigenvalues of the linearized operator $\mathcal{L}$, and then prove
the existence of $U^a$ by a fixed point argument around an
approximate solution.

\subsection{A family of approximate solutions converging to $W$ as $t\rightarrow +\infty$.}

\begin{proposition}\label{constru:approximatesolution}
Let $a\in \R$. There exists a sequence of functions
$(\ZZZ^a_j)_{j\geq 1}$ in $\mathcal{S}(\R^d)$ such that
$\ZZZ^a_1=a\mathcal{Y}_+$,   and if
\begin{equation}\label{h app}
\aligned
 h^a_k(t,x):=  \sum^k_{j=1}e^{-je_0t}\mathcal{Z}^a_j,\; k\in \Z^+,
\endaligned\end{equation}
then as $t\rightarrow +\infty$,
\begin{equation}\label{approxerror}
\aligned \epsilon_k:=&  \partial_t  h^a_k + \mathcal{L}  h^a_k  -
R(h^a_k) = O(e^{-(k+1)e_0t}) \quad \text{in }\;\mathcal{ S }(\R^d).
\endaligned
\end{equation}
\end{proposition}

\begin{remark}
Let
\begin{equation}\label{u app}
\aligned
 U^a_k(t,x):= W(x)+ h^a_k(t,x),
\endaligned\end{equation}
then as $t\rightarrow +\infty$,
\begin{equation*}
\aligned \epsilon_k:=& i \partial _t U^a_k + \Delta  U^a_k  + \Big(
\frac{1}{|x|^{4}}* |U^a_k |^2 \Big) U^a_k = O(e^{-(k+1)e_0t}) \quad
\text{in }\; \mathcal{ S }(\R^d).
\endaligned
\end{equation*}
\end{remark}

\begin{proof}[Proof of Proposition \ref{constru:approximatesolution}]
We prove it by induction. Note that $h^a_1:=e^{-e_0t}\ZZZ^a_1$, we
have
$$
\partial_t h^a_1 + \LLL h^a_1-R(h^a_1)= -R(h^a_1)
= -R \left(e^{-e_0t}\ZZZ^a_1\right).
$$
By the definition of \eqref{remainder}, we know that $R$ has at
least the quadratic term, this implies that  $R
\left(e^{-e_0t}\ZZZ^a_1\right) = O(e^{-2e_0t})$.  Therefore, we
conclude \eqref{approxerror} for $k=1$.

Let $k\geq 1$ and assume that there exist $\ZZZ^a_1,\ldots,\ZZZ^a_k$
such that $h^a_k$ satisfies \eqref{approxerror}, then there exists
$\UUU^a_{k+1}\in \mathcal{S}$ such that, as $t\rightarrow +\infty$,
$$
\partial_t h^a_k + \LLL h^a_k = R(h^a_k) + e^{-(k+1)e_0 t} \UUU^a_{k+1}
+ O \left(e^{-(k+2)e_0t} \right)\; \text{ in }\; \mathcal{ S }.
$$
By Proposition \ref{spectral}, $(k+1)e_0$ is not in the spectrum of
$\LLL$. Define $$\ZZZ^a_{k+1}:= -\left(\LLL-(k+1)e_0\right)^{-1}
\UUU^a_{k+1}.$$ By the similar argument as in the proof in \cite[Remark
7.2]{DuyMerle:NLS:ThresholdSolution}, we know that that
$\ZZZ^a_{k+1}\in \mathcal{ S }$. Then we have
\begin{equation}
 \label{eqVk1}
\partial_t\left(h^a_k +e^{-(k+1)e_0t}\ZZZ^a_{k+1}\right)
+ \LLL \left(h^a_k+e^{-(k+1)e_0t}\ZZZ^a_{k+1}\right)\\
=R(h^a_k)+ O\left(e^{-(k+2)e_0t} \right).
\end{equation}

Denote $$h^a_{k+1}:= h^a_k+e^{-(k+1)e_0 t} \ZZZ^a_{k+1}.$$ By
\eqref{eqVk1}, $h^a_{k+1}$ satisfies
\begin{equation}
 \label{eqVk+1}
\partial_t h^a_{k+1}+\LLL h^a_{k+1} - R(h^a_{k+1})=
R(h^a_{k})-R(h^a_{k+1})+O\left(e^{-(k+2)e_0 t}\right) \; \text{as}\;
t\rightarrow +\infty.
\end{equation}

Since
\begin{equation*}
h^a_j=  O\left(e^{-e_0t}\right) \;\; \text{for}\;\; j=k,k+1,
\;\;\text{and}\;\; h^a_k-h^a_{k+1}=  O(e^{-(k+1)e_0t}) \;
\text{as}\;t\rightarrow +\infty,
\end{equation*} we obtain
\begin{equation*}
R(h^a_{k})-R(h^a_{k+1}) = O\left(e^{-(k+2)e_0 t}\right) \;
\text{as}\; t\rightarrow +\infty.
\end{equation*}
This together with \eqref{eqVk+1} shows the desired estimate
\eqref{approxerror} for $k+1$, so we completes the proof.
\end{proof}

\subsection{Construction of special solutions near an approximate solution.} Now we prove the
existence of the special solution with the threshold energy by a
fixed point argument.
\begin{proposition}\label{existence:thresholdsolution}
Let $a \in \R$. There exist $k_0>0$ and $t_0 \geq 0$ such that for
any $k\geq k_0$, there exists a radial solution $U^a$ of \eqref{har}
such that for $t\geq t_0$
\begin{equation}\label{difference:Stri} \aligned
\left\|  U^a-U^a_k  \right\|_{l(t, +\infty)}  \leq
e^{-(k+\frac{1}{2})e_0t}.
\endaligned
\end{equation}
Furthermore $U^a$ is the unique solution of \eqref{har} satisfying
\eqref{difference:Stri} for large $t$. Finally, $U^a$ is independent
of $k$ and satisfies, for $t\geq t_0$,
\begin{equation}\label{difference:Energy}
\aligned   \big\|  U^a(t)- W - a e^{-e_0t}\mathcal{Y}_+ \big\|_{\dot
H^1}  \leq e^{-\frac{3}{2}e_0t}.
\endaligned
\end{equation}
\end{proposition}

\begin{proof}
The function $U^a$ is solution of \eqref{har} if and only if
$h^a:=U^a-W$ is solution of
\begin{equation*}
\aligned
\partial_t h^a + \mathcal{L} h^a = R(h^a).
\endaligned
\end{equation*}
By \eqref{approxerror},  $h^a_k:=U^a_k-W$ satisfies that
\begin{equation*}
\aligned \partial_t  h^a_k + \mathcal{L}  h^a_k  = R(h^a_k) +
\epsilon_k.
\endaligned
\end{equation*}
Therefore, $U^a$ satisfies \eqref{har} if and only if $e
:=U^a-U^a_k=h^a-h^a_k$ satisfies
$$\partial_t e +\LLL
e = R(h^a_k+e )- R(h^a_k)-\eps_k.$$ This may be rewritten
\begin{align}
\label{solve:h}  i\partial_t e+\Delta e =& V e+  i R(h^a_k+e)-i
R(h^a_k)-i\eps_k.
\end{align}

Let
\begin{equation*}
\aligned \MMM_k(e)(t):=-\int_t^{+\infty}e^{i(t-s)\Delta} \Big(i
Ve(s)-R\big(h^a_k(s)+e(s)\big)+R(h^a_k(s))+ \eps_k(s)\Big) ds.
\endaligned
\end{equation*}
It is easy to see that the  solution $U^a$ of \eqref{har} satisfying
\eqref{difference:Stri} for $t\geq t_0$ is equivalent to the
 fixed point of the following  integral  equation
\begin{equation}
\label{defM} \aligned
 \forall \; t\geq t_0,\quad
e(t)=\MMM_k(e)(t)\quad \text{ and }\quad  \|  e \|_{l(t,+\infty)}  \leq
e^{-\left(k+\frac{1}{2}\right)e_0 t}.
\endaligned
\end{equation}

Let us fix $k$ and $t_0$. Denote
\begin{align*}
E_{l}^k&:=\Big\{e\in Z(t_0,+\infty),\;\nabla e\in S(t_0,+\infty);\;
  \|e\|_{E_{l}^k}:=\sup_{t\geq t_0} e^{\left(k+\frac{1}{2}\right)e_0 t}  \| e \|_{l(t,+\infty)}  <\infty\Big\},\\
B_{l}^k&:=\big\{e\in Z(t_0,+\infty),\;\nabla e\in S(t_0,+\infty);\;
  \sup_{t\geq t_0} e^{\left(k+\frac{1}{2}\right)e_0 t}  \| e \|_{l(t,+\infty)} \leq 1\big\}.
\end{align*}
The space $E_{l}^k$ is a Banach space. In view of \eqref{defM}, it
is sufficient to show that if $t_0$ and $k$ are large enough, the
mapping $\MMM_k$ is a contraction on $B_l^k$.

Let $e \in B_{l}^k$, $k\geq 1$. By the Strichartz estimate, we have
\begin{align*}
\| \MMM_k(e) \|_{l(t,+\infty)}  \leq  C^*\Big(\|\nabla \big( V
 e \big) \|_{N(t,+\infty)}  +\|\nabla
\big(R(h^a_k+e)-R(h^a_k)\big)\|_{N(t,+\infty)}+\|\nabla
\eps_k\|_{N(t,+\infty)}\Big).
\end{align*}

We first estimate $\|\nabla \big(V e\big) \|_{N(t,+\infty)} $. Let
$\tau_0\in (0,1)$. By Lemma \ref{linearoperator:prelimestimate}, we
have
\begin{equation*}
\aligned \|\nabla \big(V e\big) \|_{N(t,t+\tau_0)} \leq C
\tau^{\frac13}_0
  \big\|  e\big\|_{l(t,t+\tau_0)} \leq C \tau^{\frac13}_0
  e^{-(k+\frac12)e_0t} \big\|e\big\|_{E^k_l } \leq  C \tau^{\frac13}_0
  e^{-(k+\frac12)e_0t}.
\endaligned
\end{equation*}
This together with Lemma \ref{summation} yields that
\begin{equation}\label{lin}
\aligned \|\nabla \big(V e\big) \|_{N(t,+\infty)} \leq \frac{  C
\tau^{\frac13}_0 }{1- e^{-(k+\frac12)e_0\tau_0} } \;
e^{-(k+\frac12)e_0t}.
\endaligned
\end{equation}

To estimate $\|\nabla
\big(R(h^a_k+e)-R(h^a_k)\big)\|_{N(t,+\infty)}$. By Lemma
\ref{linearoperator:prelimestimate}, we have
\begin{equation*}
\aligned  &\|\nabla \big(R(h^a_k+e)-R(h^a_k)\big)\|_{N(t,t+1)} \\
\leq &\; C \big\|e\big\|_{l(t,t+1)}\Big(
\big\|h^a_k\big\|_{l(t,t+1)} + \big\|e\big\|_{l(t,t+1)} +
\big\|h^a_k\big\|^2_{l(t,t+1)} +
\big\|e\big\|^2_{l(t,t+1)}\Big)\\
\leq &\; C e^{-(k+\frac32)e_0t}\big\|e\big\|_{E^k_l} \leq  C
e^{-(k+\frac32)e_0t}.
\endaligned
\end{equation*}
Thus, we obtain  by Lemma \ref{summation}
\begin{equation}\label{nonlin} \aligned \|\nabla
\big(R(h^a_k+e)-R(h^a_k)\big)\|_{N(t,+\infty)} \leq \frac{C }{1-
e^{-(k+\frac32)e_0}}\; e^{-(k+\frac32)e_0t}.
\endaligned
\end{equation}

Finally, we estimate $\|\nabla \eps_k\|_{N(t,+\infty)}$. By
$\eps_k=O(e^{-(k+1)e_0t})$ for large $t$, we have
\begin{equation}\label{error}
\aligned \|\nabla \eps_k\|_{N(t,+\infty)} \leq \frac{C}{(k+1)e_0}\;
e^{-(k+1)e_0t}.
\endaligned
\end{equation}

By \eqref{lin}, \eqref{nonlin} and \eqref{error}, we have
\begin{align*}
\| \MMM_k(e) \|_{l(t,+\infty)} \leq & C^* \left( \frac{  C
\tau^{\frac13}_0 }{1- e^{-(k+\frac12)e_0\tau_0} } + \frac{C
e^{-e_0t} }{1- e^{-(k+\frac32)e_0}}  +
\frac{Ce^{-\frac12e_0t}}{(k+1)e_0}\; \right) \;
e^{-(k+\frac12)e_0t}.
\end{align*}
If choosing a small $\tau_0,$ and a large $ k_0$ and $t_0$ such that
\begin{equation*}
\aligned C^* C \tau^{\frac13}_0 =\frac18,\;\;
e^{-(k_0+\frac12)e_0\tau_0} \leq \frac12,\;\;\text{and}\;\; C^*
\left( \frac{C e^{-e_0t_0} }{1- e^{-(k_0+\frac32)e_0}}  +
\frac{Ce^{-\frac12e_0t_0}}{(k_0+1)e_0}\; \right) \leq \frac14,
\endaligned
\end{equation*}
we have
\begin{align}\label{Str1}
\| \MMM_k(e) \|_{E^k_l} \leq \frac12, \qquad k\geq k_0,\quad t\geq t_0.
\end{align}

By the similar analysis,  we can obtain for any $e, f\in B_{l}^k$,
$k\geq 1$,
\begin{align}
  &\| \MMM_k(e)-\MMM_k(f) \|_{l(t,+\infty)}\nonumber
\\
\leq & C^*\Big(\|\nabla \big( V e- V f\big)\|_{N(t,+\infty)}
+\|\nabla
(R(h^a_k+e)-R(h^a_k+f))\|_{N(t,+\infty)}\Big) \nonumber\\
\leq &  C^* \left( \frac{  C \tau^{\frac13}_0 }{1-
e^{-(k+\frac12)e_0\tau_0} } + \frac{C e^{-e_0t} }{1-
e^{-(k+\frac32)e_0}}   \right) e^{-(k+\frac12)e_0t}
\;\big\|e-f\big\|_{E^k_l}  \nonumber\\
\leq &  \frac12 e^{-(k+\frac12)e_0t}
\big\|e-f\big\|_{E^k_l}.\label{Str2}
\end{align}
For $k\geq k_0$, and large $t_0$, \eqref{Str1} and \eqref{Str2} show
that $\MMM_k$ is a contraction map on $B^k_l$. Thus for $k\geq k_0$,
\eqref{har} has an unique solution $U^a$ satisfying
\eqref{difference:Stri} for $t\geq t_0$.

Next we show that $U^a$ does not depend on $k$. Since the above
proceeding still remains valid for the larger $t_0$, the uniqueness
still holds in the class of solution of \eqref{har} satisfying
\eqref{difference:Stri} for $t\geq t'_0$, where $t'_0$ is any real
number larger than $t_0$.  Let $k<\widetilde{k}$, and $U^a$ and
$\widetilde{U}^a$ be the solutions of \eqref{har} constructed for $k
$ and $\widetilde{k}$ respectively. The uniqueness of the fixed
point show that $U^a=\widetilde{U}^a$ for large $t$, then by the
uniqueness of \eqref{har}, we have $U^a=\widetilde{U}^a$.

Finally, we prove \eqref{difference:Energy}. Note that
$$U^a-U^a_k=e=\MMM_k(e).$$ By the Strichartz estimate, we have
\begin{align*}
&\big\|U^a- U^a_k \big\|_{\dot H^1} =  \big\|e\big\|_{\dot H^1}\\
\leq & C^*\Big(\|\nabla \big(V e \big)\|_{N(t,+\infty)}  +\|\nabla
\big(R(h^a_k+e)-R(h^a_k)\big)\|_{N(t,+\infty)}+\|\nabla
\eps_k\|_{N(t,+\infty)}\Big) \\
\leq  & e^{-(k+\frac12)e_0t}.
\end{align*}
This together with the fact $U^a_k = W +
ae^{-e_0t}\mathcal{Y}_{+}+ O(e^{-2e_0t})$ in $\dot H^1$ yields
\eqref{difference:Energy}. \end{proof}

\subsection{Construction of
$W^{\pm}$.}

Note that \eqref{difference:Energy} and the conservation of energy,
we have
$$E(U^a)=E(W).$$
In addition, we have  by \eqref{difference:Energy}
\begin{equation*}
\aligned \big\|\nabla U^a(t) \big\|^2_{2}= \big\|\nabla W \big\|^2_2
+ 2a e^{-e_0 t}\int_{\R^d} \nabla W \cdot \nabla \mathcal{Y}_1 \;dx+
O\left(e^{-\frac32e_0t}\right)\quad \text{as} \; t\rightarrow
+\infty.
\endaligned
\end{equation*}
By \eqref{WYrelation}, replacing $\mathcal{Y}_+$ by $-\mathcal{Y}_+$
if necessary, we may assume that
\begin{equation*}
\aligned \int \nabla W \cdot \nabla \mathcal{Y}_1  \; dx > 0.
\endaligned
\end{equation*}
This  implies that $\big\|\nabla U^a(t) \big\|^2_{2} - \big\|\nabla
W \big\|^2_2$ enjoys the same sign with $a$ for large positive time.
Thus, by Lemma \ref{energytrapp}, $\big\|\nabla U^a (t_0)
\big\|^2_{2} - \big\|\nabla W \big\|^2_2$ has the same  sign to that
of $a$.

Let
\begin{equation}\label{Definition:threshold}
\aligned   W^{+}(t,x)=  U^{+1}(t+t_0, x), \quad W^{-}(t,x)=
U^{-1}(t+t_0, x),
\endaligned
\end{equation}
which yields two radial solutions $W^{\pm}(t,x)$ of \eqref{har} for
$t\geq 0$, and satisfies
\begin{equation*}
\aligned E(W^{\pm}(t))=E(W), \; \big\|\nabla W^-(0)\big\|_2 <
\big\|\nabla W \big\|_2 , \; \big\|\nabla W^+(0)\big\|_2 >
\big\|\nabla W \big\|_2,
\endaligned
\end{equation*}
and
\begin{equation*}
\aligned  \big\|W^{\pm}(t)-W\big\|_{\dot H^1} \leq C e^{-e_0t}, \;\;
t\geq 0.
\endaligned
\end{equation*}
To complete the proof of Theorem \ref{threholdsolution}, it remains
to show the behavior of $W^{\pm}$ for the negative time, we will do
it in Section \ref{subs:blowup} and Section \ref{S:sub:scattering}.
\qed

%
%
%
%

\section{Modulation}\label{S:modulation}
For the radial function $u\in \dot H^1(\R^d)$, we define the
gradient variant of the solution away from $W$ as
\begin{equation*}
\aligned \delta(u) = \left|\int_{\R^d} \Big( |\nabla u(x)|^2 -
|\nabla W(x)|^2\Big) dx\right|.
\endaligned
\end{equation*}
By Proposition \ref{P:static stability}, we know that if
\begin{equation}\label{GSenergy}
\aligned E(u)=E(W)
\endaligned
\end{equation}
and $\delta(u)$ is small enough, then there exists
$\widetilde{\theta}$ and $\widetilde{\mu}$ such that
\begin{equation*}
\aligned u_{\widetilde{\theta}, \widetilde{\mu}} = W +
\widetilde{u}, \;\; \text{with}\;  \big\|\widetilde{u}\big\|_{\dot
H^1} \leq \varepsilon(\delta(u)),
\endaligned
\end{equation*}
where $\varepsilon(\delta) \rightarrow 0$ as $\delta\rightarrow 0$.
The goal in this section is that for the solution of \eqref{har}
with small gradient variant, we choose parameters
$\widetilde{\theta}$ and $\widetilde{\mu}$ to show the linear
dependence between these parameters, their derivatives and
$\delta(u)$.

From the implicit theorem and the variational characterization
of $W$ in Proposition \ref{P:static stability}, we have the
following orthogonal decomposition.
\begin{lemma}\label{choice:spatialtranslation}
There exist  $\delta_0>0$ and a positive function $\epsilon(\delta)$
defined for $0<\delta\leq \delta_0$, which tends to $0$ when $\delta
$ tends to $0$, such that for all radial $u$ in $\dot H^1(\R^d)$
satisfying \eqref{GSenergy},  there exists a couple $(\theta, \mu)\in
\R\times (0, +\infty)$   such that $v=u_{\theta,\mu}$ satisfies
\begin{equation}
\aligned v\bot iW, \quad v \bot \widetilde{W}.
\endaligned
\end{equation}
The parameters $(\theta, \mu)$ are unique in $\R/2\pi\Z \times \R,
$, and the mapping $u\mapsto (\theta, \mu)$ is $C^1$.
\end{lemma}

\begin{proof} From Proposition \ref{P:static stability}, there
exist a function $\epsilon $ and $\theta_1, \mu_1>0$ such that
$u_{\theta_1, \mu_1}=W+g$ and
\begin{equation}\label{est:variational}
\big\| u_{\theta_1, \mu_1} -W\big\|_{\dot H^1} \leq \epsilon \left(
\delta(u)\right).
\end{equation}

Now we consider the functional
\begin{align*}
J(\theta, \mu, f)=&  \left(J_0(\theta, \mu, f), J_1(\theta, \mu, f)
\right) = \left( (f_{\theta, \mu}, iW)_{\dot H^1}, (f_{\theta,\mu},
\widetilde{W})_{\dot H^1}\right)\\
= & \left( (e^{i\theta}\mu^{-\frac{d-2}{2}} f(x/\mu), iW)_{\dot
H^1}, (e^{i\theta}\mu^{-\frac{d-2}{2}} f(x/\mu),
\widetilde{W})_{\dot H^1}\right).
\end{align*}
By the simple computations, we have
\begin{align*}
J(0,1, W)=0, \quad \left| \frac{\partial J}{\partial(\theta, \mu)}
(0,1, W)\right| = \left| \begin{array}{cc} \big\|\nabla W \big\|^2_{L^2} & 0\\
0 &  - \big\|\nabla \widetilde{W} \big\|^2_{L^2} \end{array} \right|
\not = 0.
\end{align*}
Thus by the implicit theorem, there exist $\epsilon_0, \eta_0 > 0$,
such that  for any $h\in \dot H^1$ with $\big\| h- W \big\|_{\dot
H^1} < \epsilon_0$, there exists a unique
$\left(\widetilde{\theta}(h), \widetilde{\mu}(h)\right) \in C^1$
satisfying $|\widetilde{\theta}| + |\widetilde{\mu}-1 |< \eta_0 $
and
\begin{align*}
J(\widetilde{\theta}, \widetilde{\mu}, h)=0.
\end{align*}
By \eqref{est:variational}, this implies that there exists a unique
$\left(\widetilde{\theta}_1(u), \widetilde{\mu}_1(u)\right)$ such
that
\begin{align*}
J\left(\widetilde{\theta}_1 + \theta_1, \widetilde{\mu}_1 \mu_1, u
\right)= J\left(\widetilde{\theta}_1, \widetilde{\mu}_1,
u_{\theta_1, \mu_1} \right) = 0.
\end{align*}
This completes the proof by taking $\theta =\widetilde{\theta}_1 +
\theta_1$ and $\mu = \widetilde{\mu}_1 \mu_1$.
\end{proof}

Let $u(t)$ be a radial solution of \eqref{har} satisfying
\eqref{GSenergy} and write
\begin{equation}\label{def:delta}
\aligned \delta(t) := \delta(u(t))= \left|\int_{\R^d}\Big( |\nabla
u(t, x)|^2 - |\nabla W(x)|^2 \Big) dx\right|.
\endaligned
\end{equation}

Let $D_{\delta_0}$ be the open set of all times $t$ in the domain of
existence of $u$ such that  $\delta(t)< \delta_0$. On
$D_{\delta_0}$, by Lemma \ref{choice:spatialtranslation}, there
exists $C^1$ functions $\theta(t)$ and $\mu(t)$ with  the following
decomposition
\begin{equation}\label{modulcomp}
\aligned   u_{\theta(t), \mu(t)}(t, & x ) = \Big( 1+ \alpha(t) \Big)
W(x) + h(t,x),
\endaligned
\end{equation}
\begin{equation*}
\aligned  1+ \alpha(t) =& \frac{1}{\big\|W\big\|^2_{\dot
H^1}}\left(u_{\theta(t), \mu(t)}, W\right)_{\dot H^1},\quad h \in
H^{\bot}\cap \dot H^1_{rad}.
\endaligned
\end{equation*}
In Section \ref{S:convergence:sup} and Section
\ref{S:convergence:sub}, we will make use of additional conditions
to show that $u$ converges exponentially  to $W$ in $\dot H^1$,  up to
the constant modulation parameters.

As a consequence of the coercivity of $\Phi$ on $H^{\bot}\cap \dot
H^1_{rad}$ in Proposition \ref{coerH}, we can identify the scaling
and phase parameters on the interval with small gradient variant
$\delta(t)$, which can be controlled by the gradient variant.

\begin{lemma}\label{estmodulpara}Let $u$ be a radial
solution of \eqref{har}
satisfying \eqref{GSenergy}. Then taking a smaller $\delta_0$ if
necessary, the following estimates hold for $t \in D_{\delta_0}$:
\begin{align}\label{modulpara}
\big| \alpha(t)\big| \approx \big\|   \alpha(t)W(\cdot) +
h(t,\cdot)\big\|_{\dot H^1}\approx & \big\|h(t,\cdot)\big\|_{\dot
H^1} \approx
\delta(t) ,\\
\label{modulparaderiva}  \big| \alpha'(t)\big| + \big|
\theta'(t)\big| +\left|\frac{\mu'(t)}{\mu(t)}\right|\leq &\; C
\mu^2(t)\delta(t).
\end{align}
\end{lemma}

\begin{proof}Denote
\begin{equation*}
\aligned \widetilde{\delta}(t) := \big| \alpha(t)\big| + \big\|
\alpha(t) W  + h(t) \big\|_{\dot H^1}+  \big\|h(t) \big\|_{\dot H^1}
+\delta(t).
\endaligned
\end{equation*}
 We first show that
 \begin{equation}\label{modulpara-variant}
 \aligned
\widetilde{\delta}(t) \lesssim \big| \alpha(t)\big|.
 \endaligned
 \end{equation}
Indeed, by \eqref{modulcomp}, $E(u)=E(W)$ and $t\in D_{\delta_0}$,
we have
\begin{equation}\label{linenerg-u}
\aligned \Phi(\alpha(t) W + h(t)) = O\left(\big\|\nabla
\big(\alpha(t) W + h(t)\big)\big\|^3_2 \right).
\endaligned
\end{equation}
By \eqref{B-orthEnerHbot}, we know that $W$ and $h(t)$ are
$B$-orthogonality, and $B(f,f)=\Phi(f)$, thus
\begin{equation}\label{whorth}
\aligned  \Phi(\alpha(t) W + h(t))= -\alpha(t)^2 \big|\Phi(W)\big| +
\Phi(h(t)).
\endaligned
\end{equation}
By \eqref{linenerg-u}, \eqref{whorth} and the coercivity of $\Phi$
on $H^{\bot}\cap \dot H^1_{rad}$, we have
\begin{equation}\label{h-control}
\aligned \big\|h(t)\big\|^2_{\dot H^1} \approx \Phi(h(t)) \approx
\alpha(t)^2 + O\left(\big\|\nabla \big(\alpha(t) W +
h(t)\big)\big\|^3_2 \right) \lesssim \alpha(t)^2 +
\widetilde{\delta}^3(t).
\endaligned
\end{equation}
By the orthogonality of $W$ and $h$ in $\dot H^1$, we have
\begin{equation}\label{Wh-control}
\aligned \big\|\nabla \big(\alpha(t) W + h(t)\big)\big\|^2_2 =
\alpha(t)^2\big\|\nabla W \big\|^2_2 + \big\|\nabla  h(t)\big\|^2_2
\lesssim \alpha(t)^2 + \widetilde{\delta}^3(t).
\endaligned
\end{equation}
and
\begin{equation}\label{delta}
\aligned \delta(t) =\left| \int_{\R^d} \big|\nabla u(t)\big|^2  -
\big| \nabla W\big|^2\; dx \right| = \left| \big(2 \alpha(t) +
\alpha(t)^2\big) \big\| \nabla W \big\|^2_2 + \big\|\nabla h (t)
\big\|^2_2 \right|.
\endaligned
\end{equation}
If $\delta_0$ is small, then $\big\|\nabla \big(\alpha(t) W +
h(t)\big)\big\|_2$ and $\big|\alpha(t)\big|$ are small (by the
orthogonality of $W$ and $h$ in $\dot H^1$), thus we have
\begin{equation*}
\aligned \delta(t) \lesssim \big| \alpha (t)\big| +
\widetilde{\delta}^3(t).
\endaligned
\end{equation*}
Therefore, we obtain that
\begin{equation*}
\aligned \widetilde{\delta}(t) \lesssim \big| \alpha(t)\big| +
\widetilde{\delta}^{\frac32}(t) + \widetilde{\delta}^3(t).
\endaligned
\end{equation*}
By the continuity argument, we have \eqref{modulpara-variant}. By
\eqref{h-control}, \eqref{Wh-control} and \eqref{delta} once again,
we have
\begin{equation*}
\aligned \big| \alpha(t)\big|  \approx \big\|h(t)\big\|_{\dot H^1}
\approx \big\|   \alpha W + h\big\|_{\dot H^1} \approx \delta(t).
\endaligned
\end{equation*}

Next we prove \eqref{modulparaderiva}. Denote
\begin{equation*}
\aligned \widetilde{\widetilde{\delta}}(t):= \big| \alpha'(t)\big| +
\big| \theta'(t)\big| +\left|\frac{\mu'(t)}{\mu(t)}\right|+
\mu^2(t)\delta(t).\endaligned
\end{equation*}
Let \begin{equation*} \aligned  v(t,y)=u_{\theta(t), \mu(t)}(t,y)=
e^{i\theta(t)} \mu(t)^{-\frac{d-2}{2}} u\left(t,
\frac{y}{\mu(t)}\right),
\endaligned
\end{equation*} then
\begin{equation*}
\aligned u(t,x)=e^{-i\theta(t)}\mu^{\frac{d-2}{2}}(t)\; v\big(t,
\mu(t)x\big),
\endaligned
\end{equation*}
and \eqref{har} may be written
\begin{equation*}
\aligned i\partial_t v + \mu^2\Delta v +
\mu^2\big(\frac{1}{|x|^{4}}* |v|^2 \big) v + \theta'v +
i\frac{\mu'}{\mu}\left(\frac{d-2}{2} + y\cdot \nabla \right)v=0.
\endaligned
\end{equation*}
For $t\in D_{\delta_0}$, by \eqref{modulcomp}, we have
$v=\big(1+\alpha \big) W + h$, $h\in H^{\bot}$. One easily verifies that  $h$ satisfies
that
\begin{equation*}
\aligned i\partial_t h + \mu^2(t) \Delta h + i \alpha'(t) W +
\theta'(t)W + i\frac{\mu'(t)}{\mu(t)} \widetilde{W} = O\Big(\mu^2(t)
\delta(t) + \delta(t) \widetilde{\widetilde{\delta}}(t)\Big)\;
\text{in}\; \dot H^1.
\endaligned
\end{equation*}
Note that $h\in H^{\bot}$ for $t\in D_{\delta_0}$, we have
$\partial_t h \in H^{\bot}$, i.e.
\begin{equation*}
\aligned \Re \int \partial_t h \Delta W = \Im \int \partial_t h
\Delta W = \Re \int \partial_t h \Delta  \widetilde{W}= 0.
\endaligned
\end{equation*}
This together with \eqref{modulpara} implies that
\begin{align*}
 \big| \alpha'(t)  \big| =  O\Big(\mu^2(t) \delta(t) + \delta(t)
\widetilde{\widetilde{\delta}}(t)\Big),   \;\; \big|\theta'(t) \big|
=
  O\Big(\mu^2(t) \delta(t) + \delta(t)
\widetilde{\widetilde{\delta}}(t)\Big),
\\
\left|\frac{\mu'(t)}{\mu(t)} \right| =     O\Big(\mu^2(t) \delta(t)
+ \delta(t) \widetilde{\widetilde{\delta}}(t)\Big).
\qquad\qquad\qquad
\end{align*}
Therefore, we obtain that
\begin{equation*}
\aligned \widetilde{\widetilde{\delta}}(t)= O\; \Big(\mu^2(t)
\delta(t) + \delta(t) \widetilde{\widetilde{\delta}}(t)\Big).
\endaligned
\end{equation*}
This  yields the result if $\delta_0$ is chosen small enough.
\end{proof}

\section{Convergence to $W$ for the supercritical threshold solution}\label{S:convergence:sup}
In this section, we will show the dynamics of $W^+$  in  Theorem
\ref{threholdsolution} in the negative time,  and it is also the
first step in the proof of case (c) of Theorem \ref{classification}.

\begin{proposition}\label{expdecay:supercase}
Consider a radial solution $u \in  H^1(\R^d)$ of \eqref{har} such
that
\begin{equation}\label{data:super}
\aligned E(u)=E(W), \quad \big\| \nabla u_0\big\|_{L^2}>
\big\|\nabla W\big\|_{L^2},
\endaligned
\end{equation}
which is globally defined for the positive times. Then there exist
$\theta_0\in \R/(2\pi\Z)$, $\mu_0 \in (0, +\infty)$, $c, C>0$ such
that
\begin{equation}\label{asymptotic:supercritical case}
\aligned \forall \; t\geq 0, \quad \big\|u-W_{\theta_0,
\mu_0}\big\|_{\dot H^1} \leq Ce^{-ct},
\endaligned
\end{equation}
and the negative time of existence of $u$ is finite.
\end{proposition}

\subsection{Exponential convergence of the gradient variant}
\begin{lemma}\label{L:firder} Under the assumptions of Proposition
\ref{expdecay:supercase},  there exists $C$ such that for any
$R>0$, and all $t\geq 0$,
\begin{align}
\label{Fderiv}   \big| \partial_t V_R(t) \big| \leq   CR^2 \;
\delta(t),
\end{align}
where $V_R(t)$ is defined as in \eqref{localV}.
\end{lemma}
\begin{proof} By Lemma \ref{L:local virial}, we have
\begin{align*}
\partial_t V_R(t)=& 2\Im \int_{\R^d} \overline{u}\; \nabla u \cdot \nabla
\phi_R \; dx .
\end{align*}

Note that $\big| \nabla \phi_R \big| \lesssim R^2/|x|$,  we have by
the Hardy inequality
\begin{equation*}
\aligned \big|\partial_t V_R(t)\big| \lesssim R^2
\big\|u(t)\big\|^2_{\dot H^1}.
\endaligned
\end{equation*}
Thus by Lemma \ref{energytrapp} and the definition of $\delta(t)$,
it suffices to show \eqref{Fderiv} when $\delta(t) \leq \delta_0$,
where $\delta_0$ comes from  Section \ref{S:modulation}. In this case, by
Lemma \ref{estmodulpara}, we can write $u$ as
\begin{align*}
u_{\theta(t), \mu(t)}(t,x)= W + \widetilde{u}, \;\;
\big\|\widetilde{u}\big\|_{\dot H^1} \lesssim
\delta(t)\Longrightarrow u(t,x)= \big( W +
\widetilde{u}\big)_{-\theta(t), 1/\mu(t)}.
\end{align*}
Thus we have
\begin{align*}
\partial_t V_R(t) =& \;  2 \Im \int_{\R^d} \overline{\Big( W + \widetilde{u}\Big)_{-\theta(t), 1/\mu(t)} }\nabla \Big(W + \widetilde{u} \Big)_{-\theta(t),
1/\mu(t)} \nabla \phi_R \; dx \\
=& \; \frac{2R}{\mu(t)} \Im \int_{\R^d} \overline{\Big( W +
\widetilde{u}\Big)} \nabla \Big(W + \widetilde{u} \Big) \cdot \nabla
\phi  \left(
\frac{x}{R\mu(t)}\right) \; dx \\
= & \;  2R^2 \Im \int_{\R^d} \frac{1}{R \mu(t)} \Big( W \nabla
\widetilde{u} +
 \overline{ \widetilde{u}} \nabla W + \overline{\widetilde{u} }  \nabla
 \widetilde{u} \Big) \cdot \nabla \phi  \left(
\frac{x}{R\mu(t)}\right) \; dx.
\end{align*}
By the definition of $\phi$ and the Hardy inequality, we have
\begin{align*}
\big|\partial_t V_R(t)\big| \lesssim R^2 \big( \big\|\widetilde{u}
\big\|_{\dot H^1} +   \big\|\widetilde{u} \big\|^2_{\dot H^1} \big)
\lesssim R^2 \delta(t).
\end{align*}
This completes the proof. \end{proof}

\begin{lemma}\label{L:secderiv} Under the assumptions of Proposition
\ref{expdecay:supercase},  there exist $C>0$ and $R_1\geq 1$
such that for $R\geq R_1$, and all $t\geq 0$,
\begin{align}
\label{Secderiv} \partial^2_t V_R(t) \leq & - 4 \delta(t),\\
\label{fderiv:positive}
\partial_t V_R(t) >&  \; 0.
\end{align}
\end{lemma}
\begin{proof}
By Lemma \ref{L:local virial}, we have
\begin{align*}
\partial^2_t V_R(t)=&\; 8 \int_{\R^d} \big|\nabla u (t,x)\big|^2 dx -8
\iint_{\R^d \times \R^d} \frac{|u(t,x)|^2|u(t,y)|^2}{|x-y|^4} \;
dxdy + A_R\big(u(t)\big),
\end{align*}
where $ A_R\big(u(t)\big)$ is defined in Lemma \ref{L:local virial}.

By $E(u)=E(W) = \big\|\nabla W \big\|^2_{L^2}/4$, we have
\begin{align*}
8 \int_{\R^d} \big|\nabla u\big|^2 dx -8 \iint_{\R^d\times\R^d}
\frac{|u(t,x)|^2|u(t,y)|^2}{|x-y|^4} \; dxdy=-8 \delta(t).
\end{align*}
To prove \eqref{Secderiv}, it suffices to show that for $R \geq R_1$
\begin{align*}
  A_R\big(u(t)\big)  \leq 4
\delta(t).
\end{align*}
We divide it into three steps:

\noindent{\bf Step 1: General bound on $A_R(u(t))$.} We claim that
there exists $C>0$ such that for any  $R>0$, $t\geq 0$,
 \begin{align}\label{errcontrol1:sup}
A_R\big(u(t)\big) \leq \frac{C}{R^2} + \frac{C}{R^{\frac{4d-4}{d}}}
\big\|u(t)\big\|^{\frac{4}{d}}_{\dot H^1}
 + \frac{C}{R^2}
 \big\|u(t)\big\|^2_{\dot H^1}.
\end{align}
Indeed,  choosing suitably $\phi$ such that
\begin{align*}
\phi''(r) \leq 2,\; r\geq 0,\;\text{and} \; \big| -\Delta \Delta
\phi_R \big| \lesssim  \frac{1}{R^2},
\end{align*}
we conclude by the definition $A_R(u)$ in Lemma \ref{L:local virial}
\begin{align*}
A_R\big(u(t)\big) \leq C\frac{\big\|u(t)\big\|^2_{L^2}}{R^2}
 +   8 \iint_{\R^d\times\R^d} & \left[ \Big(1-\frac12 \frac{R}{|x|}\phi'\big(\frac{|x| }{R
}\big) \Big)x -  \Big(1-\frac12 \frac{R}{|y|}\phi'\big(\frac{|y|
}{R }\big) \Big)y \right]\\
\times&\;   \frac{(x-y)}{|x-y|^6}\;|u(t,x)|^2|u(t,y)|^2 \;dxdy.
\end{align*}
Let  $\chi_{\Omega}$ be the characteristic function of the set
$\Omega$. Note that
\begin{align*}
\left|  \Big(1-\frac12 \frac{R}{|x|}\phi'\big(\frac{|x| }{R }\big)
\Big)x -  \Big(1-\frac12 \frac{R}{|y|}\phi'\big(\frac{|y| }{R }\big)
\Big)y  \right| \lesssim & \chi_{\{ |x|\geq R \} \cup \{|y|\geq R\}
} \big|x-y\big|,\end{align*} and the radial Sobolev inequality in
\cite{Strauss:Soliton wave in higher Dim, tao:book}
\begin{align*}
\big\|\chi_{\{|\cdot|\geq R\}} u(t,
\cdot)\big\|^4_{L^{\frac{2d}{d-2}}} \lesssim &
\frac{1}{R^{\frac{4d-4}{d}}} \big\|u(t)\big\|^{\frac{4}{d}}_{\dot
H^1} \big\|u(t)\big\|^{\frac{4d-4}{d}}_{L^2},
\end{align*}
we have
\begin{align*}&\iint_{\R^d\times\R^d} \chi_{\{|x|\geq R \} \cup
\{|y|\geq R\} }   \frac{1}{|x-y|^4}  |u(t,x)|^2|u(t,y)|^2 \; dxdy \\
\lesssim & \iint_{|x|\approx|y|\gtrsim  R}   \frac{1}{|x-y|^4} \;
|u(t,x)|^2|u(t,y)|^2 \; dxdy  \\& + \iint_{\max(|x|,|y|)\gg \min(|x|,|y|) \atop \max(|x|,|y|) \geq R} \frac{1}{|x-y|^4} \; |u(t,x)|^2|u(t,y)|^2 \; dxdy \\
\lesssim &\big\|\chi_{\{|\cdot|\geq R\}} u(t,
\cdot)\big\|^4_{L^{\frac{2d}{d-2}}} + \frac{1}{R^2}
\big\|u(t)\big\|^2_{L^2}\big\|u(t)\big\|^2_{\dot H^1}\\
\lesssim &
\frac{1}{R^{\frac{4d-4}{d}}}\big\|u(t)\big\|^{\frac{4d-4}{d}}_{L^2}
\big\|u(t)\big\|^{\frac{4}{d}}_{\dot H^1}
 + \frac{1}{R^2}
\big\|u(t)\big\|^2_{L^2}\big\|u(t)\big\|^2_{\dot H^1}.
\end{align*}
This yields  \eqref{errcontrol1:sup} by the mass conservation.

\noindent{\bf Step 2: Refined bound on $A_R(u(t))$ when $\delta(t)$
small.} We claim that there exist $\delta_1$,  $C>0$ such that for
any $R > 1$, and $t\geq 0$ with $\delta(t)\leq \delta_1$,
\begin{align}\label{errcontrol2:sup}
\Big| A_R(u(t)) \Big| \leq C \left( \frac{1}{ R ^{\frac{d-2}{2}}}
\delta(t) + \delta(t)^2\right).
\end{align}

To do so, we first show that for small $\delta_1$,
\begin{align}\label{sf:sup}
\mu_-:=\inf \Big\{\mu(t), t\geq 0, \delta(t)\leq \delta_1 \Big\}
> 0.
\end{align}
Indeed, by \eqref{modulcomp} and Lemma \ref{estmodulpara}, we have
\begin{align*} u_{\theta(t), \mu(t)} (t)= W + V,\; \text{with}\;\;
\big\|V(t)\big\|_{\dot H^1} \approx \delta(t).
\end{align*}
By the mass conservation, we have
\begin{align*}
\big\|u_0\big\|^2_{L^2} \geq & \int_{|x|\leq \mu(t)} \big| u(t,x)
\big|^2 dx = \frac{1}{\mu(t)^2} \int_{|x|\leq 1} \big| u_{\theta(t),
\mu(t)}\big|^2 dx \\
\gtrsim & \frac{1}{\mu(t)^2} \left(\int_{|x|\leq 1}\big|W\big|^2 dx
- \int_{|x|\leq 1}\big|V(t) \big|^2 dx \right).
\end{align*}
This together with
\begin{align*}
\big\|V(t)\big\|_{L^2(|x|\leq 1)} \lesssim
\big\|V(t)\big\|_{L^{\frac{2d}{d-2}}(|x|\leq 1)} \lesssim
\big\|V(t)\big\|_{\dot H^1} \lesssim \delta(t)
\end{align*}
implies that
\begin{align*}
\big\|u_0\big\|^2_{L^2} \geq  \frac{1}{\mu(t)^2} \left(
\int_{|x|\leq 1} \big|W\big|^2 dx -  C \delta(t)^2 \right).
\end{align*}
Choosing sufficiently small $\delta_1$, one easily gets
\eqref{sf:sup}. By \eqref{modulcomp} and the change of variable, we
have
\begin{align*}
\Big| A_R(u(t)) \Big| = \Big| A_R\big((W+V)_{-\theta(t),1/\mu(t) }
\big) \Big| = \Big| A_{R\mu(t)}\big( W+V   \big) \Big|.
\end{align*}
Note that
\begin{align*}
A_{R\mu(t)}(W)=0,\;\text{and}\;\; \big\|\nabla W\big\|_{L^2(|x|\geq
\rho)}\approx \big\| W\big\|_{L^{\frac{2d}{d-2}}(|x|\geq
\rho)}\approx \rho^{-\frac{d-2}{2}}\; \text{for}\; \rho\geq 1,
\end{align*}
we have
\begin{align*}
  \Big| A_R(u(t)) \Big|   = & \Big| A_{R\mu(t)}(W+V)-
A_{R\mu(t)}(W) \Big|
\\
\lesssim &\;  \int_{|x|\geq R\mu(t)}\big| \nabla W \cdot \nabla V(t)
\big| + \big| \nabla V(t)\big|^2
 dx \\
& + \int_{R\mu(t)\leq |x|\leq 2R\mu(t)}
\frac{1}{\big(R\mu(t)\big)^2} \left(\big|
  W \cdot V(t) \big| + \big| V(t)\big|^2  \right) dx \\
  & + \iint_{\{|x|\geq R\mu(t) \} \atop \cup \{|y|\geq R\mu(t) \}}
  \frac{1}{|x-y|^4}
  \Big( \big|W(x)V(t,x)\big| \big|W(y) \big|^2 + \big|W(x)\big|^2\big|W(y)V(t,y)
  \big|
 \\
  & \qquad \qquad  \qquad    + \big|W(x)\big|^2\big|V(t,y)\big|^2+ \big|V(t,x)\big|^2\big|W(y)\big|^2 \\
  & \qquad \qquad  \qquad   + \big|W(x)V(t,x)\big| \big|V(t,y) \big|^2 + \big|V(t,x)\big|^2 \big|V(t,y) \big|^2\Big) \;
  dxdy\\
  \lesssim & \frac{\|V(t)\|_{\dot
  H^1} }{\big(R\mu(t)\big)^{\frac{d-2}{2}}} + \frac{\|V(t)\big\|_{\dot
  H^1}}{\big(R\mu(t)\big)^{d-2}} + \big\|V(t)\big\|^2_{\dot
  H^1} + \big\|V(t)\big\|^3_{\dot
  H^1} + \big\|V(t)\big\|^4_{\dot
  H^1}.
\end{align*}
This combining  with \eqref{sf:sup} implies that
\eqref{errcontrol2:sup} if $\delta_1$ is small enough.

\vskip0.2cm

 \noindent{\bf Step 3: Conclusion.} From \eqref{errcontrol2:sup},
 there exist $\delta_2 > 0 $ and $R_2\geq 1$ such that if $R\geq R_2$, $\delta(t)\leq
 \delta_2$,
\begin{align*}
\Big| A_R(u(t)) \Big| \leq C \left( R ^{-\frac{d-2}{2}} \delta(t) +
\delta(t)^2\right) \leq 4 \delta(t).
\end{align*}

Let
\begin{align*}
f_{R_3}(\delta):=\frac{C}{R_3^2} + \frac{C}{R_{3}^{\frac{4d-4}{d}}}
\big( \delta+ \big\|W\big\|^2_{\dot H^1}\big)^{\frac{2}{d}}
 + \frac{C}{R_3^2}
\big( \delta+ \big\|W\big\|^2_{\dot H^1}\big)-4\delta,
\end{align*}
where $C$ is given by \eqref{errcontrol1:sup}. For sufficient large
$R_3$, $f_{R_3}(\delta)$ is concave on $\delta$. Choosing $R_3 $
large enough such that $f_{R_3}(\delta_3)\leq 0$, and
$f'_{R_3}(\delta_3)\leq 0$, we have for any $R\geq R_3$, $\delta\geq
\delta_3$
\begin{align*}
f_{R}(\delta)\leq f_{R_3}(\delta) \leq 0.
\end{align*}
This  implies that $ A_R(u(t)) \leq 4 \delta(t)$, so we  conclude the
proof of \eqref{Secderiv} with $R_1=\max(R_2, R_3)$.

Finally, note that $V_R(t)>0$,   $\partial^2_t
V_R(t) < 0$ for $t>0$ and $u$ is defined on $[0, +\infty)$, we easily see that $\partial_t V_R(t) > 0$
by the
convexity analysis.
\end{proof}

\begin{lemma}\label{L:inteexpondecay} Under the assumptions of Proposition
\ref{expdecay:supercase}, there exist $c>0$, $C>0$   such that
for $R\geq R_1$ ( which is given in Lemma \ref{L:secderiv}), and all
$t\geq 0$,
\begin{align}
\label{expdecay:Inte:supercase}
 \int^{+\infty}_t  \delta(s)ds   \leq & \; C  e^{-ct}.
\end{align}
\end{lemma}

\begin{proof}
Fix $R\geq R_1$. By \eqref{Fderiv}, \eqref{Secderiv} and
\eqref{fderiv:positive}, we have \begin{align*} 4\int^T_t
\delta(s)ds \leq - \int^T_t
\partial^2_s V_R(s)ds =
\partial_tV_R(t) - \partial_tV_R(T) \leq CR^2 \delta(t).
\end{align*}
Let $T\rightarrow +\infty $, we have
\begin{align*}
 \int^{+\infty}_t  \delta(s)ds   \leq C  \delta(t).
\end{align*}
By the Gronwall inequality, we have \eqref{expdecay:Inte:supercase}.
\end{proof}
\subsection{Convergence of the modulation parameters.} Let us show that
\begin{equation}\label{diff:Gradient}
\aligned \lim_{t\rightarrow+\infty}\delta(t)=0.
\endaligned
\end{equation}
If \eqref{diff:Gradient} does not hold, by
\eqref{expdecay:Inte:supercase}, there exist two increasing
sequences $(t_n)_{n}$, $(t'_n)_{n}$ such that $t_n<t'_n$,
$\delta(t_n)\rightarrow 0$, $\delta(t'_n) = \varepsilon_1 $ for some
$0< \varepsilon_1 < \delta_0$, and
\begin{align*}
  \forall\; t\in (t_n, t'_n), \;
0< \delta(t) < \epsilon_1.
\end{align*}
On $[t_n, t'_n]$, $\theta(t)$ and $\mu(t)$ are well-defined by Lemma
\ref{choice:spatialtranslation}. By Proposition \ref{P:static
stability}, there exist $\theta_0$ and $\mu_0>0$ such that
\begin{equation}\label{approx:u:sup}
u(t_n)\longrightarrow W_{-\theta_0, 1/\mu_0} \quad \text{in}\quad \dot
H^1.
\end{equation}
By the mass conservation, there exists $u_{\infty} \in L^2$ such
that $ u(t) \rightharpoonup u_{\infty} $ in $ L^2 $, as
$t\rightarrow +\infty.$ Hence, we have
\begin{equation*}
u(t) \rightharpoonup W_{-\theta_0, 1/\mu_0} \quad  \text{in}\quad  L^2, \;
\; \text{as}\;\; t\rightarrow +\infty,
\end{equation*}
which implies that
\begin{equation}\aligned\label{fsf:bound interval}
 \displaystyle
 \mu(t)\;\text{is bounded on}\;
\bigcup_{n} [t_n, t'_n].
\endaligned
\end{equation}

By Lemma \ref{estmodulpara} and \eqref{expdecay:Inte:supercase}, we
have
\begin{align*}
\big|\alpha(t_n) \big|\approx & \;\delta(t_n) \rightarrow 0, \; \text{as}\; n\rightarrow +\infty, \\
\big|\alpha(t'_n)-\alpha(t_n)\big|= & \left|\int^{t'_n}_{t_n}
\alpha'(s) ds \right| \leq C \left|\int^{t'_n}_{t_n} \delta(s) ds
\right| \leq C e^{-ct_n}.
\end{align*}
Therefore $\alpha(t'_n) \rightarrow 0$ as $n\rightarrow +\infty$,
which contradicts the definition of $t'_n$.

Similarly by \eqref{diff:Gradient}, we have that $\mu(t)$ is bounded
for large $t$, and the parameter $\alpha(t)$ is well defined for
large $t$, and
\begin{align}
\delta(t) \approx &\big| \alpha(t) \big| =  \big| \alpha(t)
-\alpha(+\infty) \big| \notag\\
\leq & C \int^{+\infty}_t \big| \alpha'(s)\big| ds \leq C
\int^{+\infty}_t  \mu^2(s) \delta(s) ds \leq C
e^{-ct},\label{gradient:sup:expdecay}
\end{align}
which, together with Lemma \ref{estmodulpara}, implies that $
\frac{1}{\mu(t)^2} $ satisfies the Cauchy criterion as $t\rightarrow
+\infty$. Then we have
\begin{align*}
\lim_{t \rightarrow +\infty } \mu(t) = \mu_{+\infty} \in (0, +\infty
],
\end{align*}
This combining with  the boundness of  $\mu(t)$  precludes the case
$\lim_{t\rightarrow +\infty}\mu(t) = +\infty$, Hence we have
\begin{equation} \label{fsf:bound infinity}
 \lim_{t\rightarrow+\infty}\mu(t) =
\mu_{\infty} \in (0, +\infty).
\end{equation}
Thus by Lemma \ref{estmodulpara}, we have
\begin{equation*}
\aligned \big\|u-W_{\theta(t), \mu(t)}\big\|_{\dot H^1} + \big|
\alpha'(t)\big| + \big| \theta'(t)\big| \leq C \delta(t) \leq C
e^{-ct}.
\endaligned
\end{equation*}
This  yields \eqref{asymptotic:supercritical case}. \qed

\subsection{Blowup for the negative times}\label{subs:blowup}
It is a consequence of the positivity of $\partial_t V_R(t)$ in
\eqref{Fderiv} and the time reversal symmetry. Suppose that $u$ is
also global for the negative time. Applying Lemma \ref{L:firder},
Lemma \ref{L:secderiv} and Lemma \ref{L:inteexpondecay} to
$\overline{u}(-t)$, we know that they also hold for the negative
times.  Hence by \eqref{diff:Gradient}, we know that
\begin{equation*}
\lim_{t\rightarrow \pm \infty} \delta(t) =0.
\end{equation*}
By Lemma \ref{L:firder} and Lemma \ref{L:secderiv}, we know that
$\partial_t V_R(t)> 0$ and $
\partial_t V_R(t) \longrightarrow 0, \; \text{as}\; t\rightarrow \pm
\infty.
$
By Lemma \ref{L:secderiv}, we have
$
\partial^2_t V_R(t)<0.
$ This implies that $\partial_tV_R(t)\equiv 0$. It is a
contradiction, so we conclude the proof.\qed


\section{Convergence to $W$ for the subcritical threshold solution}\label{S:convergence:sub}
In this section, we consider the radial subcritical  threshold
solution $u$ of \eqref{har}. Similar to that in Section
\ref{S:convergence:sup}, the following proposition will give the
dynamics of $W^-$ of Theorem \ref{threholdsolution} in the negative
time and is also the first step in the proof of case (a) of Theorem
\ref{classification}.
\begin{proposition}\label{expdecay:subcase}
Let $u \in \dot H^1(\R^d)$ be a radial solution of \eqref{har}, and
$I=(T_-, T_+)$  denote its maximal interval of existence. Assume that
\begin{equation} \label{sub thresh case}
\aligned E(u_0)=E(W),  \quad \big\| \nabla u_0\big\|_{L^2} <
\big\|\nabla W\big\|_{L^2},
\endaligned
\end{equation}
Then $$I=\R.$$ Furthermore, if $u$ does not scatter for the positive
times, that is,
\begin{equation}\label{Noscattering}
\aligned \big\|u\big\|_{Z(0,+\infty)}=\infty,
\endaligned
\end{equation}
 then
there exist $\theta_0\in \R, \mu_0>0, c, C>0$ such that
\begin{equation}
\aligned \big\|u-W_{\theta_0,
\mu_0}\big\|_{\dot H^1} \leq Ce^{-ct}, \qquad \forall t\geq 0,
\endaligned
\end{equation}and
\begin{equation}\label{negscattering}
\aligned \big\|u\big\|_{Z(-\infty, 0)}< \infty.
\endaligned
\end{equation}
An analogues assertion holds on $(-\infty, 0]$.
\end{proposition}

\subsection{Compactness properties.}\label{subs:compact}
It is well known that the solution with
\begin{equation*}
\aligned E(u)<E(W),  \quad \big\| \nabla u_0\big\|_{L^2} <
\big\|\nabla W\big\|_{L^2},
\endaligned
\end{equation*}
is global well-posed and scatters in both time directions \cite{LiMZ:e-critical Har, MiaoXZ:09:e-critical radial Har}.
  By Lemma
\ref{energytrapp}, the concentration compactness principle and the
stability theory as the proof of Proposition 4.2 in
\cite{MiaoXZ:09:e-critical radial Har}, we can show:

\begin{lemma} \label{compactness}
Let $u$ be a radial solution of \eqref{har} satisfying
\begin{align*}
E(u_0)=E(W),  \;\; \big\| \nabla u_0\big\|_{L^2} < \big\|\nabla
W\big\|_{L^2}, \; \; \big\|u\big\|_{Z(0,T_+)}=+\infty.
\end{align*}   Then there
exists a continuous functions $\lambda(t)$ such that the set
\begin{equation}\label{Kcompact}
\aligned K:= \big\{ (u(t))_{\lambda(t)} , t\in [0,T_+) \big\}
\endaligned
\end{equation}
is  pre-compact in $\dot H^1(\R^d)$.
\end{lemma}

Let $u$ be a solution of \eqref{har}, and $\lambda(t)$ be as in
Lemma \ref{compactness}. Consider $\delta_0$ as in Section 4. The
parameters $\theta(t)$, $\mu(t)$ and $\alpha(t)$ are defined for
$t\in D_{\delta_0}= \{t: \delta(t)< \delta_0 \}$. By
\eqref{modulcomp} and Lemma \ref{estmodulpara}, there exists a
constant $C_0>0$ such that
\begin{equation*}
\aligned \int_{\mu(t)\leq |x|\leq
2\mu(t)}\big|\nabla u(t,x) \big|^2 dx = & \int_{\mu(t)\leq |x|\leq
2\mu(t)} \frac{1}{\mu(t)^d} \left| e^{i\theta(t)}\nabla u\left(t,\frac{x}{\mu(t)}\right) \right|^2 dx  \\
 \geq & \int_{1\leq |x|\leq 2}\big|
\nabla W\big|^2 -C_0 \delta(t), \quad  \forall \; t \in D_{\delta_0}.
\endaligned
\end{equation*}
Taking a smaller $\delta_0$ if necessary, we can assume that the
right hand side of the above inequality is bounded from below by a
positive constant $\varepsilon_0$ on $D_{\delta_0}$. Thus, we have
\begin{equation*}
\aligned \int_{\frac{\mu(t)}{\lambda(t)}\leq |x|\leq \frac{2
\mu(t)}{\lambda(t)}}\frac{1}{\lambda(t)^d}\left|\nabla
u\left(t,\frac{x}{\lambda(t)}\right) \right|^2 dx \geq \int_{1\leq
|x|\leq 2}\big| \nabla W\big|^2 -C_0 \delta(t), \quad  \forall \; t \in D_{\delta_0}.
\endaligned
\end{equation*}
By the compactness of $\overline{K}$, it follows that for any $t\in
D_{\delta_0}$, we have $\big| \mu(t) \big| \sim \big| \lambda(t)
\big| $. As a consequence, we may modify $\lambda(t)$ such that $K$
defined by \eqref{Kcompact} remains pre-compact in $\dot H^1$ and
\begin{equation}\label{para-uniform-SGD}
\aligned \forall \; t\in D_{\delta_0}, \; \lambda(t)=\mu(t).
\endaligned
\end{equation}

As a consequence of Lemma \ref{compactness}, we have
\begin{corollary}\label{globalpositive}Let $u$ be a radial solution of \eqref{har} satisfying \eqref{sub thresh
case} and  not scatter for the positive times. Then
\begin{enumerate}
\item[\rm (a)] $T_+=+\infty.$

\item[\rm (b)] $\displaystyle \lim_{t \rightarrow \infty} \sqrt{t} \lambda(t) =\infty$;
\end{enumerate}
\end{corollary}

As a direct consequence of (a) in Corollary \ref{globalpositive}, we
have

\begin{corollary}\label{global}
Let $u$ be a radial solution of \eqref{har} with the maximal existence
interval $I$ such that
\begin{equation*} \aligned E(u_0)\leq E(W), \quad \big\|\nabla
u_0\big\|_2 \leq \big\|\nabla W \big\|_{2},
\endaligned
\end{equation*}
then $I=\R.$
\end{corollary}
\begin{proof}
If $\big\|\nabla u_0\big\|_2 = \big\|\nabla W\big\|_2$, then by
\eqref{convexity}, we have $E(u_0)=E(W)$, then by the varialtional
characterization of $W$, we have $u_0=\pm W_{\theta_0,\lambda_0}$
for some $\theta_0\in [0, 2\pi), \lambda_0>0$. By uniqueness of
\eqref{har}, $u$ is only the stationary solution $\pm W_{\theta_0,
\lambda_0}$, which is globally defined.

If $\big\|\nabla u_0\big\|_2 < \big\|\nabla W\big\|_2, E(u_0)<E(W)$,
then from \cite{MiaoXZ:09:e-critical radial Har}, we have the
solution $u$ is global wellposed and scatters.

Now we consider the case $\big\|\nabla u_0\big\|_2 < \big\|\nabla
W\big\|_2, E(u_0)=E(W)$. If $\big\|u\big\|_{Z(I)}<+\infty$, then by
the finite blowup criterion, we know that $u$ is a global solution.
If $\big\|u\big\|_{Z([0, T_+))}=+\infty$, then by Corollary
\ref{globalpositive} (a), we have $T_+ =+\infty$. The same result
holds for the negative time.\end{proof}

\begin{proof}[Proof of Corollary \ref{globalpositive}]
 We show (a) by contradiction. Assume that
 $$T_+(u_0)< +\infty.$$ For this case we can show that
$$\lambda(t) \rightarrow +\infty ,\quad
\text{as}\; t \rightarrow T_+(u_0). $$
 Then using the
localized mass argument, the almost finite propagation speed and the
compactness property of $\overline{K}$ in $\dot H^1$ as Step 2 of
Lemma 2.8 in \cite{DuyMerle:NLS:ThresholdSolution}, we can show that
$$u_0 \in L^2.$$
Moreover, we have $u \equiv 0,$ which contradicts the assumption
that $E(u_0)=E(W)$ or that $u$ blows up at finite time $T_+>0$.

We also show (b) by the compactness argument. Assume that (b) does
not hold. Then there exists a sequence $t_n\rightarrow +\infty$ such
that
\begin{equation*}
\aligned \lim_{t_n\rightarrow +\infty} \sqrt{t_n} \lambda(t_n)
=\tau_0 \in[0, +\infty).
\endaligned
\end{equation*}

Consider
\begin{equation*}
\aligned w_n (s,y) = \lambda(t_n)^{-\frac{d-2}{2}} u
\left(t_n+\frac{s}{\lambda(t_n)^2}, \frac{y}{\lambda(t_n)} \right).
\endaligned
\end{equation*}
By the compactness of $\overline{K}$, up to a subsequence, there
exists a function $w_0\in \dot H^1$ such that
\begin{equation*}
\aligned w_n(0) \rightarrow w_0 \; \; \text{in}\; \dot H^1
\endaligned
\end{equation*}

Let $w$ be the solution of \eqref{har} with initial data $w_0$. Note
that
\begin{equation*}
\aligned E(u_0)=E(W),\quad \big\|\nabla u(t_n)\big\|_{2} <
\big\|\nabla W\big\|_{2},
\endaligned
\end{equation*}
we have
\begin{equation*}
\aligned E(w_0)=E(W),\quad \big\|\nabla w_0\big\|_{2} \leq
\big\|\nabla W\big\|_{2}.
\endaligned
\end{equation*}
Thus by Corollary \ref{global}, we have that
$T_{-}(w_0)=T_{+}(w_0)=\infty$. By Theorem \ref{T:local} and
$-\sqrt{t_n} \lambda(t_n)\rightarrow -\tau_0$, we have
\begin{equation*}
\aligned
 \lambda(t_n)^{-\frac{d-2}{2}}u_0\left( \frac{y}{\lambda(t_n)}
\right) =w_n\big( -t_n \lambda(t_n)^2, y \big) \rightarrow
w(-\tau^2_0, y), \;\;\text{in}\; \dot H^1.
\endaligned
\end{equation*}
Since $\lambda(t_n) \rightarrow 0$, we have
\begin{equation*}
\aligned
 \lambda(t_n)^{-\frac{d-2}{2}}u_0\left(\frac{y}{\lambda(t_n)}
\right) \rightharpoonup 0, \;\;\text{in}\; \dot H^1.
\endaligned
\end{equation*}
Thus $w(-\tau^2_0)=0$, which contradicts $E(w)=E(W)>0$. \end{proof}

\subsection{Convergence in the ergodic mean.}
\begin{lemma}\label{meanconverg}
Let $u$ be a radial solution of \eqref{har} satisfying  \eqref{sub
thresh case} and  \eqref{Noscattering}. Then
\begin{equation}
\aligned \lim_{T\rightarrow +\infty}
\frac{1}{T}\int^{T}_{0}\delta(t) dt =0,
\endaligned
\end{equation}
where $\delta(t)$ is defined by \eqref{def:delta}.
\end{lemma}

\begin{proof} By Lemma \ref{L:local virial} and the Hardy inequality, we have
\begin{align*}
\big| \partial_t V_R(t)\big| \leq & C R^2.
\end{align*}
From  $E(u)=E(W)=\big\|\nabla W\big\|^2_{L^2}/4$, we have
\begin{align*}
\partial^2_t V_R(t)=&  -8 \delta(t) + A_R(t),
\end{align*}
where $ A_R\big(u(t)\big)$ is defined in Lemma \ref{L:local virial}.
By the compactness of $\overline{K}$ in $\dot H^1$, for any
$\epsilon >0$, there exists $\rho_{\epsilon} >0$ such that
\begin{align*}
\int_{|x|\geq \frac{\rho_{\epsilon}}{\lambda(t)}} \big| \nabla
u(t,x) \big|^2 dx \leq \epsilon.
\end{align*}

Note that
\begin{align*}
\left|  \Big(1-\frac12 \frac{R}{|x|}\phi'\big(\frac{|x| }{R }\big)
\Big)x -  \Big(1-\frac12 \frac{R}{|y|}\phi'\big(\frac{|y| }{R }\big)
\Big)y  \right| \lesssim & \chi_{\{ |x|\geq R \} \cup \{|y|\geq R\}
} \big|x-y\big|,\end{align*} and
\begin{align*}
&\iint \chi_{\{|x|\geq R\}\cup\{|y|\geq R\}}
\frac{1}{\big|x-y\big|^4} \big|u(t,x) \big|^2 \big|u(t,y) \big|^2\;
dxdy\\
\lesssim &  \big\|u(t)\big\|^2_{L^{\frac{2d}{d-2}} (|x|\geq R )}
\big\|u(t)\big\|^2_{\dot H^1}  \lesssim
\big\|u(t)\big\|^2_{L^{\frac{2d}{d-2}} (|x|\geq R )},
\end{align*}
we have for $\displaystyle R \geq  \rho_{\epsilon}/\lambda(t)$
\begin{align}\label{smalremainder}
 \forall \; t\geq 0,\;  \big| A_R(t) \big| \leq \epsilon.
\end{align}

Fix $\epsilon$, choose $\epsilon_0$ and $M_0$ such that
\begin{align*}
2C \epsilon^2_0 = \epsilon,\;\; M_0\epsilon_0 \geq \rho_{\epsilon}.
\end{align*}
By Corollary \ref{globalpositive} (b), there exists $t_0\geq 0$ such
that
\begin{align*}
\forall \; t\geq t_0, \;\; \lambda(t) \geq \frac{M_0}{\sqrt{t}}.
\end{align*}
For $T\geq t_0$, let
$R:=\epsilon_0 \sqrt{T}$.
For $t\in [t_0, T]$, we have
\begin{align*}
R \geq \epsilon_0 \sqrt{T} \; \frac{ M_0 }{\sqrt{t} \lambda(t)} =
\frac{\sqrt{T}}{\sqrt{t}}\; \frac{M_0 \epsilon_0}{ \lambda(t)} \geq
\frac{\rho_{\epsilon}}{\lambda(t)},
\end{align*}
this  yields that
\begin{align*}
8\int^T_t \delta(s) ds \leq & \int^T_t \partial^2_s V_R(s) ds +
\big| A_R(s) \big| \big( T -t_0\big) \leq   2CR^2 + \epsilon T = 2
\epsilon T.
\end{align*}
Let $T\rightarrow +\infty$, we have
\begin{align*}
\lim_{T\rightarrow+\infty} \frac{1}{T} \int^T_0 \delta(s) ds \leq
\frac{\epsilon}{4}.
\end{align*}
This completes the proof.
\end{proof}

\begin{corollary}\label{seqconv}
Let $u$ be a radial solution of \eqref{har} satisfying  \eqref{sub
thresh case} and  \eqref{Noscattering}.  Then there exists a
sequence $t_n$ such that $t_n \rightarrow +\infty$ and
\begin{equation}
\aligned \lim_{n\rightarrow +\infty} \delta(t_n)=0.
\endaligned
\end{equation}
\end{corollary}


\subsection{Exponential convergence.}\label{subs:expconv}

\begin{lemma}\label{virialestimate-delta:positive} Let $u$ be a
radial solution of \eqref{har} satisfying  \eqref{sub thresh case},
\eqref{Noscattering} and \eqref{para-uniform-SGD}, and $\lambda(t)$
be as in Lemma \ref{compactness}. Then there exists a constant $C$
such that if $\; 0\leq \sigma < \tau$,
\begin{equation*}
\aligned \int^{\tau}_{\sigma} \delta (t) dt \leq C \left(
\sup_{\sigma\leq t\leq \tau}\frac{1}{\lambda(t)^2} \right) \Big(
\delta(\sigma) + \delta(\tau) \Big).
\endaligned
\end{equation*}
\end{lemma}
\begin{proof} For $R>0$, Let us consider the function $V_R(t)$ defined as in
\eqref{localV}.

By the same estimate as that in Lemma \ref{L:firder}, there is a
constant $C_0$ independent of $t\geq 0$ such that
\begin{align}\label{Fderiv:subcase}
\Big| \partial_t V_R(t) \Big| \leq C_0 R^2 \delta(t).
\end{align}

Now we show that if $\epsilon>0$, there exists $R_{\epsilon}$ such
that for any $R \geq R_{\epsilon}/\lambda(t)$, then
\begin{align}\label{Sderiv:subcase}
\partial^2_t V_R(t) \geq  (8-\epsilon) \delta(t).
\end{align}
Indeed, by Lemma \ref{L:local virial} and $E(u)=E(W) =  \big\|\nabla
W \big\|^2_{L^2}/4$, we have
\begin{align*}
\partial^2_t V_R(t)=&\; 8 \int_{\R^d} \big|\nabla u (t,x)\big|^2 dx -8 \iint_{\R^d\times\R^d}
\frac{|u(t,x)|^2|u(t,y)|^2}{|x-y|^4} \; dxdy + A_R\big(u(t)\big)\\
= & \; 8 \delta(t) + A_R\big(u(t)\big),
\end{align*}
where $A_R\big(u(t)\big)$ is defined in Lemma \ref{L:local virial}.

To prove \eqref{Sderiv:subcase}, it suffices to show that if
$\epsilon>0$, there exists $R_{\epsilon}$ such that for any $R \geq
R_{\epsilon}/\lambda(t)$
\begin{align*}
  \Big| A_R(t) \Big|\lesssim \epsilon
\delta(t).
\end{align*}

For the case $\delta(t)\geq \delta_0$,  as the estimate
\eqref{smalremainder}, we can use the compactness of $\overline{K}$
in $\dot H^1$ to show that for any $t\geq 0$, $\epsilon>0$, there
exists $\rho_{\epsilon}>0$, such that for any $R\geq
\rho_{\epsilon}/\lambda(t)$,
\begin{align*}
\big|A_R(t) \big|\leq \epsilon\lesssim \epsilon \delta(t).
\end{align*}

For the case  $\delta(t)< \delta_0$, similar to the proof of Step 2 in
Lemma \ref{L:secderiv}, we can show that for any $t\geq 0$,
$\rho>1$, there exists $C>0$ such that for any $R\geq
\rho/\lambda(t)$,
\begin{align*}
\big|A_R(t) \big|\leq C\left( \rho^{-\frac{d-2}{2}} \delta(t)+
\delta(t)^2\right).
\end{align*}
This  implies  \eqref{Sderiv:subcase} if we choose $\rho$ large
enough.

By \eqref{Sderiv:subcase}, there exists $R_2$ such that for any
$R\geq R_2/\lambda(t)$
\begin{align*}
\partial^2_t V_R(t) \geq 4 \delta(t).
\end{align*}

Finally, if taking $R = R_2 \displaystyle \sup_{\sigma\leq t \leq
\tau} \left(\frac{1}{\lambda(t)}\right),$ and integrating between
$\sigma$ and $\tau$, we have
\begin{align*}
4\int^{\tau}_{\sigma} \delta(t) dt \leq \int^{\tau}_{\sigma}
\partial^2_t V_R(t) dt = \partial_t V_R(\tau)-\partial_t V_R(\sigma)
\leq C R^2 \Big(\delta(\tau) + \delta(\sigma) \Big).
\end{align*}
This finishes the proof of Lemma \ref{virialestimate-delta:positive}. \end{proof}

\begin{lemma}\label{control:parameter}
Let $u$ be a radial solution of \eqref{har} satisfying \eqref{sub
thresh case},  \eqref{Noscattering} and \eqref{para-uniform-SGD},
and $\lambda(t)$ be as in Lemma \ref{compactness}. Then there is a
constant $C_0>0$ such that for any $ \sigma, \tau>0$ with
$\sigma+\frac{1}{\lambda(\sigma)^2} \leq \tau$, we have
\begin{equation}\label{parcon}
\aligned \left|\frac{1}{\lambda(\tau)^2}
-\frac{1}{\lambda(\sigma)^2} \right|\leq C_0 \int^{\tau}_{\sigma}
\delta(t) dt.
\endaligned
\end{equation}
\end{lemma}

\begin{proof} We divide it into three steps:

\noindent{\bf Step 1: Local constancy of $\lambda(t)$.}  By the
compactness of $\overline{K}$, one easily  show that there exists
$C_1>0$  such that \begin{align}\label{sf:local constancy} \forall\;
\sigma, \tau\geq 0, \big| \tau - \sigma  \big| \leq
\frac{1}{\lambda(\sigma)^2} \Longrightarrow
\frac{\lambda(\tau)}{\lambda(\sigma)} +
\frac{\lambda(\sigma)}{\lambda(\tau)} \leq C_1.
\end{align}
Indeed, for any two sequences
$\tau_n, \sigma_n \geq 0$ such that
\begin{align*}
\big| \tau_n-\sigma_n\big|\leq \frac{1}{\lambda(\sigma_n)^2}.
\end{align*}
Up to subsequence, we may assume that
\begin{align*}
\lim_{n\rightarrow +\infty} \lambda^2(\sigma_n) \big(
\tau_n-\sigma_n\big) = s_0 \in [-1, 1].
\end{align*}

Consider
\begin{align*}
v_n(s,y)=
 \Big(\lambda(\sigma_n)\Big)^{-\frac{d-2}{2}}u\left(\frac{s}{\lambda(\sigma_n)^2}+\sigma_n,
\frac{y}{\lambda(\sigma_n)}\right).
\end{align*}
By the compactness of $\overline{K}$, up to subsequence, there
exists $v_0\in \dot H^1$ such that
\begin{align*}
v_n(0) \longrightarrow v_0 \;\; \text{in}\;\; \dot H^1,\;\;
\text{as}\;\; n\rightarrow +\infty.
\end{align*}
Thus $E(v_0)=E(W)$ and $\big\|\nabla v_0 \big\|_{L^2} < \big\|\nabla
W\big\|_{L^2}$. Let $v$ be the solution of \eqref{har} with $v_0$.
By Corollary \ref{global} and Theorem \ref{T:local}, $v$ is globally
defined and
\begin{align*}
 \Big(\lambda(\sigma_n)\Big)^{-\frac{d-2}{2}} u\left(\tau_n,
\frac{y}{\lambda(\sigma_n)}\right)=v_n\left(  \lambda(\sigma_n)^2
\big( \tau_n-\sigma_n\big) , y \right) \longrightarrow v(s_0,y)
\;\text{in}\; \dot H^1 \;\;\text{as}\;\; n\rightarrow +\infty.
\end{align*}
In addition, by the compactness of $\overline{K}$, we know that $
\Big(\lambda(\tau_n)\Big)^{-\frac{d-2}{2}} u\left(\tau_n,
\frac{y}{\lambda(\tau_n)}\right) $ converges in $\dot H^1$. Thus
$\lambda(\tau_n)/\lambda(\sigma_n) +
\lambda(\sigma_n)/\lambda(\tau_n)$ is bounded.

\vskip0.2cm

\noindent{\bf Step 2: Control of the variations of $\delta(t)$. }Let
$\delta_0$ be as in Lemma \ref{modulpara}. Using the local constancy
of $\lambda(t)$ and the compactness of $\overline{K}$, we will show
that for any $\tau>0$, if \begin{align*} \sup_{t\in\left[\tau,
\tau+\frac{1}{\lambda(\tau)^2}\right]} \delta(t)>\delta_0,
\end{align*}then there exists $\delta_1>0$ such that
\begin{align}\label{local control:delta:sub}
\inf_{t\in\left[\tau, \tau+\frac{1}{\lambda(\tau)^2}\right]}
\delta(t)>\delta_1.
\end{align}
Indeed, if not, we may find sequences $\tau_n$, $t_n$ and $t'_n$
such that
\begin{align*}
t_n, t'_n \in \left[\tau_n, \tau_n
+\frac{1}{\lambda(\tau_n)^2}\right],\quad  \delta(t_n)\rightarrow 0,\quad
\delta(t'_n)>\delta_0.
\end{align*}

Consider
\begin{align*}
v_n(s,y)=  \lambda(t_n)
^{-\frac{d-2}{2}}u\left(\frac{s}{\lambda(t_n)^2}+t_n,
\frac{y}{\lambda(t_n)}\right).
\end{align*}
From the compactness of $\overline{K}$ and $\delta(t_n)\rightarrow 0$,
we may assume that
\begin{align*}
v_n(0)\longrightarrow W_{\lambda_0}\; \text{in}\; \dot
H^1\;\;\text{as}\; n\rightarrow +\infty.
\end{align*}
By Step 1, we know that $\lambda(t_n)/\lambda(\tau_n)\leq C$, thus
$\big|\lambda(t_n)^2\big(t_n -t'_n \big)\big| < C$ for some constant
$C>0$. Up to a subsequence, we may assume that
\begin{align*}
\lim_{n\rightarrow +\infty} \lambda(t_n)^2 \big( t_n-t'_n\big) = s_0
\in [-C, C].
\end{align*}
By Theorem \ref{T:local}, we know that
\begin{align*}
v_n\Big(  \lambda(t_n)^2 \big( t_n-t'_n\big), y \Big)=
 \lambda(t_n) ^{-\frac{d-2}{2}} u\left(t'_n,
\frac{y}{\lambda(t_n)}\right) \longrightarrow W_{\lambda_0}
\;\text{in}\; \dot H^1 \;\;\text{as}\;\; n\rightarrow +\infty.
\end{align*}
This  contradicts  $\delta(t'_n)> \delta_0$.

\vskip0.2cm

\noindent{\bf Step 3: End of the proof. } We first show that there
exists $C>0$ such that
\begin{align}\label{parcon:subinterval}
0\leq \sigma \leq \widetilde{\sigma} \leq \widetilde{\tau} \leq \tau
= \sigma + \frac{1}{C^2_1\lambda(\sigma)^2} \Longrightarrow
\left|\frac{1}{\lambda(\widetilde{\tau})^2}
-\frac{1}{\lambda(\widetilde{\sigma})^2} \right|\leq C
\int^{\tau}_{\sigma} \delta(t) dt.
\end{align}
where $C_1\geq 1$ is the constant defined in Step 1. Indeed, if
$\delta(t) \leq \delta_0$ on $[\sigma, \tau]$, then by Lemma
\ref{estmodulpara} and \eqref{para-uniform-SGD}, we have
\begin{align*}
\left| \frac{1}{\lambda(\widetilde{\tau})^2}
-\frac{1}{\lambda(\widetilde{\sigma})^2} \right| =\left|
\int^{\widetilde{\tau}}_{\widetilde{\sigma}}
\frac{\lambda'(t)}{\lambda(t)^3} dt \right| \leq
\int^{\widetilde{\tau}}_{\widetilde{\sigma}}\left|
\frac{\mu'(t)}{\mu(t)^3} \right|  dt \leq C \int^{\tau}_{\sigma}
\delta(t) dt.
\end{align*}
Otherwise if there exists $t\in [\sigma, \tau]$ such that
$\delta(t)> \delta_0$, then by Step 2, we know that $\delta(t) \geq
\delta_1$ for all $t\in [\sigma, \tau]$. Note that
\begin{align*}
\big| \widetilde{\sigma}- \widetilde{\tau}\big| \leq
\frac{1}{C^2_1\lambda(\sigma)^2} \leq
\frac{1}{\lambda(\widetilde{\sigma})^2}.
\end{align*}
By Step 1, we obtain
\begin{align*}
\left| \frac{1}{\lambda(\widetilde{\tau})^2}
-\frac{1}{\lambda(\widetilde{\sigma})^2} \right| \leq
\frac{2C^2_1}{\lambda(\widetilde{\sigma})^2} \leq
\frac{2C^4_1}{\lambda(\sigma)^2} = 2C^5_1 \big| \tau-\sigma \big|
\leq \frac{2C^5_1}{\delta_1}\int^{\tau}_{\sigma} \delta(t) dt.
\end{align*}

Dividing   $[\sigma, \tau]$ into small subintervals, we can complete
the proof. \end{proof}

\begin{lemma} Let $u$ be a
radial solution of \eqref{har} satisfying  \eqref{sub thresh case},
\eqref{Noscattering} and \eqref{para-uniform-SGD}, and $\lambda(t)$
be as in Lemma \ref{compactness}. Then there exists $C$   such that
for all $t\geq 0$,
\begin{align}\label{ExpDecay:Inte}
 \int^{+\infty}_t  \delta(s)ds   \leq & \; C  e^{-ct}.
\end{align}
\end{lemma}

\begin{proof} We first show that $\frac{1}{\lambda(t)^2}$ is bounded. By
Lemma \ref{virialestimate-delta:positive} and Lemma
\ref{control:parameter}, there exists a constant $C_0>0$ such that
for all $0\leq \sigma \leq s < t\leq \tau$ with
$s+\frac{1}{\lambda(s)^2}< t$, we have
\begin{align}\label{FSF-bound}
\left| \frac{1}{\lambda(s)^2} -\frac{1}{\lambda(t)^2} \right| \leq
C_0 \sup_{\sigma\leq t\leq \tau}\left(\frac{1}{\lambda(t)^2}\right)
\Big(\delta(\sigma)+\delta(\tau)\Big).
\end{align}
By Corollary \ref{seqconv}, there exist $t_n\rightarrow +\infty$ and
$n_0\in \N$ such that for $n\geq n_0$,
\begin{align*}
\delta(t_n)\leq \frac{1}{4C_0}.
\end{align*}
Take $\sigma=s=t_{n_0}$, $\tau=t_n$, then
\begin{align*}
\forall \; t \in ( t_{n_0}+\frac{1}{\lambda(t_{n_0})^2}, t_n]
\Longrightarrow \left| \frac{1}{\lambda(t_{n_0})^2}
-\frac{1}{\lambda(t)^2} \right| \leq \frac{1}{2} \sup_{ t \geq
t_{n_0}} \frac{1}{\lambda(t)^2}.
\end{align*}
Note that $t_n \rightarrow +\infty$ as $n\rightarrow +\infty$, we
have
\begin{align*}
\sup_{ t \geq t_{n_0}+\frac{1}{\lambda(t_{n_0})^2}}
\frac{1}{\lambda(t)^2} \leq &  \frac{1}{2} \sup_{ t \geq t_{n_0}}
\frac{1}{\lambda(t)^2} +  \frac{1}{\lambda(t_{n_0})^2}.
\end{align*}
Thus
\begin{align*}
\sup_{ t \geq t_{n_0}+\frac{1}{\lambda(t_{n_0})^2}}
\frac{1}{\lambda(t)^2} \leq  \sup_{t_{n_0}\leq  t \leq
t_{n_0}+\frac{1}{\lambda(t_{n_0})^2}} \frac{1}{\lambda(t)^2} +
\frac{2}{\lambda(t_{n_0})^2}.
\end{align*}
This  shows the boundness of $\frac1{\lambda(t)^2}$.

By Lemma \ref{virialestimate-delta:positive} and the boundness of
$\frac{1}{\lambda(t)^2}$, we have for $t+\frac{1}{\lambda(t)^2}<t_n$
\begin{align*}
\int^{t_n}_t \delta(s) ds \leq C \Big(\delta(t)+\delta(t_n) \Big).
\end{align*}
Let $n\rightarrow +\infty$, we obtain
\begin{align*}
\int^{\infty}_t \delta(s) ds \leq C  \delta(t).
\end{align*}
This together with the Gronwall inequality yields
\eqref{ExpDecay:Inte}.\end{proof}

\subsection{Convergence of $\lambda(t)$}
Now by Lemma \ref{control:parameter} and \eqref{ExpDecay:Inte}, we
have for $\sigma+\frac{1}{\lambda(\sigma)^2} < \tau$
\begin{align*}
\left|\frac{1}{\lambda(\sigma)^2} -\frac{1}{\lambda(\tau)^2} \right|
\leq C e^{-c\sigma}.
\end{align*}
By means of  the Cauchy criteria of convergence at infinity, there exists
$\lambda_{\infty}\in (0, +\infty]$ such that
\begin{align*}
\left|\frac{1}{\lambda(t)^2} -\frac{1}{\lambda_{\infty}^2} \right|
\leq C e^{-ct}.
\end{align*}

Now we show that
\begin{align}\label{ScaFun:limit}
\lambda_{\infty} \in (0, + \infty).
\end{align}
Assume that $\lambda_{\infty}=+\infty$. Let $0\leq \sigma \leq s$.
By \eqref{ExpDecay:Inte}, there exists a  sequence $t_n$  such that
$$\delta(t_0) \leq \frac{1}{2C_0},\;\;\delta(t_n)\rightarrow 0 \; \text{ as }\; t_n \rightarrow +\infty.$$
 For larger $n$, we
have $s+\frac{1}{\lambda(s)^2}<t_n$. By Lemma
\ref{virialestimate-delta:positive}, we obtain
\begin{align*}
\left|\frac{1}{\lambda(s)^2} -\frac{1}{\lambda(t_n)^2} \right| \leq
C_0 \sup_{\sigma\leq t \leq t_n} \Big(\frac{1}{\lambda(t)^2} \Big)
\big(\delta(\sigma)+\delta(t_n) \big).
\end{align*}
Let $n\rightarrow +\infty$, we have
\begin{align*}
\sup_{t \geq \sigma} \frac{1}{\lambda(t)^2}    \leq C_0
\delta(\sigma) \sup_{t \geq \sigma}  \frac{1}{\lambda(t)^2}.
\end{align*}
If taking $\sigma=t_0$, we have
\begin{align*}
\forall\; t \geq t_0,\;\; \lambda(t) \equiv +\infty.
\end{align*}
This  contradicts the continuity of $\lambda(t)$ on $\R$.

\subsection{Convergence of the modulation parameters.}

By \eqref{ExpDecay:Inte}, there exists $t_n\rightarrow +\infty$ such
that
$$\delta(t_n) \rightarrow 0.$$
Fix such $\{t_n\}_{n\in \N}$. In the similar proof as leading to
\eqref{diff:Gradient}, one easily sees that
\begin{equation}\label{deltalimit}
\aligned \lim_{t\rightarrow +\infty} \delta(t)=0.
\endaligned
\end{equation}

Now for large $t$, $\alpha(t)$ is well defined by Lemma
\ref{estmodulpara}, \eqref{para-uniform-SGD}, \eqref{ScaFun:limit},
and \eqref{deltalimit}. Furthermore,  we also have
\begin{align}
\delta(t) \approx& \big| \alpha(t) \big| = \big| \alpha(t)
-\alpha(+\infty) \big| \notag\\
\leq & C \int^{+\infty}_t \big| \alpha'(s)\big| ds \leq C
\int^{+\infty}_t  \mu^2(s) \delta(s) ds \leq C e^{-ct}.
\label{expdecay:pointwise}
\end{align}

Finally, by Lemma \ref{estmodulpara}, \eqref{para-uniform-SGD},
\eqref{ScaFun:limit}, and the Cauchy criteria of convergence, there
exists $\theta_{\infty}$ such that for large $t$
\begin{align*}
\Big| \theta(t) - \theta_{\infty}\Big| \leq C \int^{\infty}_t \big|
\theta'(s) \big| ds \leq C \int^{+\infty}_t  \mu^2(s) \delta(s) ds
\leq C e^{-ct}.
\end{align*}
\subsection{Scattering for the negative times.}
\label{S:sub:scattering} By Corollary \ref{global}, we know that $u$
is globally well defined. Assume that $u$ does not scatter for the
negative time. Then by the analogue estimates as those in Subsection
\ref{subs:compact}-\ref{subs:expconv}, we have
\begin{enumerate}
\item[\rm (a)] there exists $\lambda(t)$, defined for $t\in \R$, such
that the set \begin{align*} K:=\Big\{\big(u(t)\big)_{\lambda(t)},\;
t\in \R \Big\}
\end{align*}
is pre-compact in $\dot H^1$.

\item[\rm (b)] there exists an decreasing sequence $t'_n \rightarrow
-\infty$ as $n\rightarrow +\infty$, such that
\begin{align*}
\lim_{n\rightarrow +\infty} \delta(t'_n)=0.
\end{align*}

\item[\rm (c)] there is a constant $C>0$ such that if $\; -\infty <
\sigma < \tau < \infty$,
\begin{equation*}
\aligned \int^{\tau}_{\sigma} \delta (t) dt \leq C \left(
\sup_{\sigma\leq t\leq \tau}\frac{1}{\lambda(t)^2} \right) \big(
\delta(\sigma) + \delta(\tau) \big).
\endaligned
\end{equation*}

\item[\rm (d)] there is a constant $C>0$
such that for any $\sigma, \tau $ with
$\sigma+\frac{1}{\lambda(\sigma)^2} \leq \tau$, we have
\begin{equation*}
\aligned \left|\frac{1}{\lambda(\tau)^2}
-\frac{1}{\lambda(\sigma)^2} \right|\leq C \int^{\tau}_{\sigma}
\delta(t) dt.
\endaligned
\end{equation*}
\end{enumerate}
From (b)-(d), we know that
\begin{align*}
\lim_{t\rightarrow \pm \infty} \delta(t)=0.
\end{align*}
Note that  $\frac{1}{\lambda(t)^2}$ is bounded for $t\in \R$, we
easily verify  that
\begin{align*}
\int^{\tau}_{\sigma} \delta(s) \;ds\leq C
\Big(\delta(\sigma)+\delta(\tau) \Big).
\end{align*}
Let $\sigma \rightarrow -\infty$ and $\tau\rightarrow +\infty$, we
have
$$\delta(t)=0,\;  \forall \; t\in \R.$$
Thus $u=W$ up to the symmetry of \eqref{har}, it  contradicts
$\big\|u_0\big\|_{\dot H^1} < \big\|W\big\|_{\dot H^1}$. \qed


\section{Uniqueness and the classification result}\label{S:uniqueness}

\subsection{Exponentially small solutions of the linearized equation.}
Recall the notation of Section \ref{S:existence}, in particular the
operator $\mathcal{L}$ and its eigenvalues and eigenfunctions.

Consider the radial functions $v(t)$ and $g(t)$
\begin{equation*}
\aligned v\in C^0([t_0, +\infty), \dot H^1),\;\;  g\in C^0([t_0,
+\infty), L^{\frac{2d}{d+2}})\;\;\text{and}\;\; \nabla g \in N(t_0,
+\infty)
\endaligned
\end{equation*}
satisfying
\begin{align}\label{linEquat}
 \partial_t v + \mathcal{L}v =g, & \;\; (x,t)\in \R^d  \times [t_0,
+\infty), \\
\label{baseestimate} \big\|v(t)\big\|_{\dot H^1} \leq C e^{-c_1t}, &
\;\;\;\; \big\|g(t)\big\|_{L^{\frac{2d}{d+2}}(\R^d)}+  \big\| \nabla
g(t)\big\|_{N(t, +\infty)} \leq  Ce^{-c_2 t},
\end{align}
where $0<c_1<c_2$. For any $c >0$, we denote by $c^-$ a positive
number arbitrary close to $c$ and such that $0<c^{-}<c$.

By the Strichartz estimate and Lemma \ref{summation}, we have
\begin{lemma}\label{vstrich}
Under the above assumptions, then for any $(q,r)$ with $\frac2q=d
(\frac12 -\frac1r ) $, $q\in [2, +\infty]$, we have
\begin{align*}
\big\|v\big\|_{l(t,+\infty)}+\big\|\nabla v\big\|_{L^q(t,+\infty;
L^r)} \leq C e^{-c_1t}.
\end{align*}
\end{lemma}

By the coercivity of the linearized energy in $G_{\bot}$, we will
show that $v$ must decay almost as fast as $g$, except in the
direction $\YYY_+$ where the decay is $e^{-e_0t}$.

\begin{proposition}\label{improved decay}
The following estimates hold.
\begin{enumerate}
\item[\rm(a)] if $c_2 \leq e_0$ or $e_0 < c_1$, then
\begin{equation}
\aligned \big\|v(t)\big\|_{\dot H^1}  \leq C e^{-c_2^{-}t},
\endaligned
\end{equation}
\item[\rm(b)] If $c_2 > e_0$, then there exists $a\in \R$
such that
\begin{equation}
\aligned \big\|v(t)-ae^{-e_0t}\mathcal{Y}_+\big\|_{\dot H^1}  \leq C
e^{-c_2^{-}t}.
\endaligned
\end{equation}
\end{enumerate}
\end{proposition}

\begin{proof} We decompose $v$ as
\begin{equation}
\label{h:dec} \aligned v(t)=\alpha_+(t)\YYY_+ +\alpha_-(t)\YYY_- +
\beta(t)iW +\gamma(t) \widetilde{W} + v_{\bot}(t), \;\; v_{\bot}(t)
\in G_{\bot}\cap \dot H^1_{rad}.
\endaligned
\end{equation}
By \eqref{B-orthEigen}, we can normalize the eigenfunction
$\YYY_{\pm}$ such that
$$B(\YYY_+,\;\YYY_-)=1.$$
By Remark \ref{r:bilinear form prop}, we have
\begin{align}
\label{alphapm} &\alpha_-(t)=  B(v,\; \YYY_+),\;\;
\alpha_+(t)= B(v,\; \YYY_-),\\
&\label{beta} \beta(t)= \frac{1}{\big\|W\big\|^2_{\dot H^1}}
\Big(v-\alpha_+(t)\YYY_+-\alpha_-(t)\YYY_-,\; iW \Big)_{\dot H^1}, \\
&\label{gamma} \gamma(t)= \frac{1}{\big\|W\big\|^2_{\dot H^1}}
\Big(v-\alpha_+(t)\YYY_+-\alpha_-(t)\YYY_-,\;
\widetilde{W}\Big)_{\dot H^1}.
\end{align}

\noindent{\bf Step 1: Differential equations on the coefficients.}
We first show that
\begin{align}
\label{alphapmder} & \frac{d}{dt}  \left(e^{-e_0t}\alpha_-\right) =
e^{-e_0t}B(g,\; \YYY_+),\;\; \frac{d}{dt} \left(e^{
e_0t}\alpha_+\right) =
e^{e_0t}B(g,\; \YYY_-),\\
\label{betagammader} & \frac{d}{dt} \beta(t)= \frac{\big(iW,\;
\widetilde{v} \big)_{\dot H^1}}{\big\|W\big\|^2_{\dot H^1}},\;\;\;\;
\frac{d}{dt} \gamma(t) = \frac{\big(\widetilde{W},\;  \widetilde{v} \big)_{\dot H^1}}{\big\|\widetilde{W}\big\|^2_{\dot H^1}},\\
\label{Phivder} &\frac{d}{dt}\Phi(v(t)) = 2 B(g,v),
\end{align}
where
\begin{equation*}
\aligned \widetilde{v}=g-B(\YYY_-,\; g)\YYY_+-B(\YYY_+,\; g)\YYY_-
-\mathcal{L}v_{\bot}.
\endaligned
\end{equation*}

Indeed, by \eqref{linEquat} and \eqref{alphapm} and Remark
\ref{r:bilinear form prop}, we have
\begin{align}\label{alphand}
 \alpha'_-(t)= B(\partial_t v, \; \YYY_+)=& B(-\mathcal{L}v, \; \YYY_+)
+ B(g,\; \YYY_+)\\
=& e_0B(v,\; \YYY_+) + B(g,\; \YYY_+) = e_0\alpha_-(t) + B(g,\;
\YYY_+), \notag
\end{align}
and
\begin{align}\label{alphapd}
\alpha'_+(t)= B(\partial_t v, \; \YYY_-)=& B(-\mathcal{L}v, \;
\YYY_-)
+ B(g,\; \YYY_-)\\
=& -e_0B(v,\; \YYY_-) + B(g,\; \YYY_-) = -e_0\alpha_+(t) + B(g,\;
\YYY_-). \notag
\end{align}
This jointed with \eqref{alphand} implies \eqref{alphapmder}.
\vskip0.2cm

By \eqref{linEquat}, \eqref{h:dec}, \eqref{beta}, \eqref{gamma},
\eqref{alphand} and \eqref{alphapd}, we have
\begin{align*}
\frac{d}{dt} \beta(t)=& \frac{1}{\big\|W\big\|^2_{\dot H^1}}
\Big(\partial_t v-\alpha'_+(t)\YYY_+-\alpha'_-(t)\YYY_-,\; iW
\Big)_{\dot H^1}\\
=& \frac{1}{\big\|W\big\|^2_{\dot H^1}} \Big(g-\mathcal{L}
v-\alpha'_+(t)\YYY_+-\alpha'_-(t)\YYY_-  ,\; iW
\Big)_{\dot H^1}\\
=& \frac{1}{\big\|W\big\|^2_{\dot H^1}} \Big( g-B(g,\;
\YYY_-)\YYY_+-B(g,\; \YYY_+)\YYY_-+\mathcal{L}v_{\bot} ,\; iW
\Big)_{\dot H^1}:=  \frac{\big(\widetilde{v},\; iW \big)_{\dot
H^1}}{\big\|W\big\|^2_{\dot H^1}},
\end{align*}
and
\begin{align*}
\frac{d}{dt} \gamma(t)=& \frac{1}{\big\|\widetilde{W}\big\|^2_{\dot
H^1}} \Big(\partial_t
v-\alpha'_+(t)\YYY_+-\alpha'_-(t)\YYY_-,\;\widetilde{ W}
\Big)_{\dot H^1}\\
=& \frac{1}{\big\|\widetilde{W}\big\|^2_{\dot H^1}} \Big( g-B(g,\;
\YYY_-)\YYY_+-B(g,\; \YYY_+)\YYY_-+\mathcal{L}v_{\bot} ,\;
\widetilde{W} \Big)_{\dot H^1}:=  \frac{\big(\widetilde{v},\;
\widetilde{ W} \big)_{\dot H^1}}{\big\|\widetilde{W}\big\|^2_{\dot
H^1}},
\end{align*}

By \eqref{linEquat}, we have
\begin{align*}
\frac{d}{dt}\Phi(v) = \frac{d}{dt} B(v,\; v) =& 2B(v, \partial_t v)
 =   2B(v, -\mathcal{L}v) + 2 B(v,\; g)=2B(v,\; g).
\end{align*}

\noindent{\bf Step 2: Decay estimate on $\alpha_{\pm}(t)$.} We now
claim that there exists a real number $a\in \R$, such that
\begin{align}\label{alphan:bound}
\big|\alpha_-(t)\big| \leq &\; C e^{-c_2t}, \\
\label{alphap:bound} \big|\alpha_+(t) \big| \leq& \;  Ce^{-c^-_2t} \quad \text{if}\; e_0\leq c_1  \;\text{or}\; c_2 \leq e_0, \\
\label{alphap:bound2}  \big| \alpha_+(t)-ae^{-e_0t} \big| \leq  & \;
Ce^{-c_2t} \quad \text{if}\; c_1 \leq e_0 < c_2.
\end{align}

By the definition of \eqref{bilinear form}, we have
\begin{align*}
2B(g,\; h)=\;\; & \int_{\R^d} (L_+ g_1) h_1 \; dx + \int_{\R^d} (L_- g_2) h_2 \; dx\\
=\;\; & \int_{\R^d} \nabla g_1 \nabla h_1 \; dx -\int_{\R^d} \Big(\frac{1}{|x|^{4}}*|W|^2\Big) g_1 h_1 \; dx - 2\int_{\R^d} \Big(\frac{1}{|x|^{ 4}}* (Wg_1)\Big) W h_1\; dx\\
+  & \int_{\R^d} \nabla g_2 \nabla h_2\; dx-\int_{\R^d}
\Big(\frac{1}{|x|^{ 4}}*|W|^2\Big)g_2 h_2 \; dx.
\end{align*}
Hence, for any time-interval $I$ with $|I|<\infty$, we have
\begin{align*}
\int_I \Big| B(g(t),\; \YYY_\pm) \Big| dt \lesssim & \;
|I|^{\frac13} \big\| \nabla g \big\|_{N(I)} \big\|\nabla \YYY_{\pm}
\big\|_{L^{\frac{6d}{3d-4}}_x(\R^d)}\\
&\;  +  |I| \cdot \big\|  g
\big\|_{L^{\infty}_IL^{\frac{2d}{d+2}}_x(\R^d)}
\big\|\YYY_{\pm}\big\|_{L^{\infty}_{I}L^{\frac{2d}{d-2}}_x(\R^d)}
\big\| W \big\|^2_{L^{\frac{2d}{d-4}}_x(\R^d)}.
\end{align*}
This together with \eqref{baseestimate} implies that
\begin{align*}
\int^{t+1}_t \Big| e^{-e_0s}B(g(s), \; \YYY_+)\Big| \; ds \leq C
e^{-e_0t} e^{-c_2t}.
\end{align*}
By Lemma \ref{summation}, we have
\begin{align}\label{alphan:conv}
\int^{+\infty}_t \Big| e^{-e_0s}B(g(s), \; \YYY_+)\Big| \; ds \leq C
e^{-(e_0+c_2)t}.
\end{align}
By \eqref{baseestimate}, we know that $e^{-e_0t}\alpha_-(t)$ tends
to 0 when $t$ goes to infinity. Integrating the equation on
$\alpha_-$ in \eqref{alphapmder} between $t$ and $+\infty$, we have
\begin{align*}
\big|\alpha_-(t)\big| \leq C e^{-c_2t}.
\end{align*}

Now we prove \eqref{alphap:bound}. First consider the case $e_0 <
c_1$. By \eqref{baseestimate}, we know that $e^{e_0t}\alpha_+(t)$
tends to 0 when $t$ goes to infinity. By \eqref{baseestimate} and
the similar estimate as \eqref{alphan:conv}, we also have
\begin{align*}
\int^{+\infty}_t \Big| e^{e_0s}B(g(s), \; \YYY_-)\Big| \; ds \leq C
e^{(e_0-c_2)t}.
\end{align*}
Integrating the equation on $\alpha_+$ in \eqref{alphapmder} between
$t$ and $+\infty$, we have
\begin{align*}
\big|\alpha_+(t)\big| \leq C e^{-c_2t}.
\end{align*}

If $c_1\leq e_0 < c_2$, by \eqref{baseestimate}, we have
\begin{align*}
\int^{t+1}_t \Big| e^{e_0s}B(g(s),\; \YYY_-)\;\Big| \; ds \leq C
e^{e_0t} e^{-c_2t},
\end{align*}
which together with Lemma \ref{summation} implies that
\begin{align*}
\int^{+\infty}_{t_0} \Big| e^{e_0s}B(g(s),\; \YYY_-)\;\Big| \; ds
\lesssim e^{e_0t_0} e^{-c_2t_0}< \infty.
\end{align*}
By \eqref{alphapmder}, we know that $e^{e_0t}\alpha_{+}(t)$
satisfies the Cauchy criterion as $t\rightarrow +\infty$, then
\begin{align*}
\lim_{t\rightarrow+\infty} e^{e_0t}\alpha_{+}(t) =a,
\end{align*}
and
\begin{align*}
\big|e^{e_0t}\alpha_+(t) - a\big| \leq C e^{e_0t} e^{-c_2t}.
\end{align*}
This  shows \eqref{alphap:bound2}.

If $c_1 < c_2\leq e_0$, integrating the equation on $\alpha_+$ in
\eqref{alphapmder} between $0$ and $t$, we have
\begin{align*}
\alpha_+(t)=e^{-e_0t}\alpha_+(0) + e^{-e_0t} \int^t_0
e^{e_0s}B(g(s),\; \YYY_-)\; ds,
\end{align*} by \eqref{baseestimate}, we know that \begin{align*} \left|\int^t_0
e^{e_0s}B(g(s),\; \YYY_-)\; ds \right|\leq \left\{
\begin{array}{rl}
C e^{(e_0-c_2)t}, & \;\text{if}\; c_2 < e_0,\\
C t  , & \; \text{if}\; c_2=e_0.
\end{array} \right.
\end{align*}
This  yields \eqref{alphap:bound} in this case.

\noindent{\bf Step 3: Proof of the case $c_2\leq e_0$, or $c_2 >
e_0, a=0$.} By Step 2, we know in this case that
\begin{align}\label{alphapm:optcontrol}
\big| \alpha_+(t)\big| + \big| \alpha_-(t)\big| \lesssim
e^{-c^-_2t}.
\end{align}

Since
\begin{align*}
\int^{t+1}_{t} B(g(s), v(s)) ds \leq C e^{-(c_1+c_2)t},
\end{align*} we have by Lemma \ref{summation}
\begin{align*}
\int^{+\infty}_{t} B(g, v) ds \leq C e^{-(c_1+c_2)t}.
\end{align*}
By \eqref{Phivder} and $\big|\Phi(v(t)) \big|\lesssim
\big\|v(t)\big\|^2_{\dot H^1} \rightarrow 0$ as $t\rightarrow
+\infty$, we have
\begin{align*} \Big|\Phi(v(t))\Big| \leq C e^{-(c_1+c_2)t}.
\end{align*}

Note that
\begin{align*}
\Phi(v) = B(v, v) = B(v_{\bot}, v_{\bot}) + 2 \alpha_+\alpha_-,
\end{align*}
we obtain from Lemma \ref{spectral} and \eqref{alphapm:optcontrol}
\begin{align}\label{vorth:optcontrol}
\big\|v_{\bot}(t)\big\|^2_{\dot H^1} \lesssim \Big| B(v_{\bot},
v_{\bot}) \Big| \leq C e^{-(c_1+c_2)t}.
\end{align}

Now we turn to estimate the decay of $\beta(t)$. First by
\eqref{beta} and Step 2, we know that
\begin{align*}
\big| \beta(t)\big| \rightarrow 0,\; \text{as} \; t\rightarrow
+\infty.
\end{align*}
In addition, by $\mathcal{L}^* (i \Delta W)= L_+ ( \Delta W ) \in
L^{\frac{2d}{d+2}}(\R^d)$, we also have
\begin{align*}
\int^{t+1}_{t} \left|\big( \widetilde{v}, iW \big)_{\dot H^1}
\right| ds \lesssim &\; e^{-c_2t} + \int^{t+1}_t  \left| \big(iW,
\mathcal{L} v_{\bot}(s)
\big)_{\dot H^1}  \right| ds \\
\lesssim & \; e^{-c_2t} + \int^{t+1}_t \int_{\R^d}
\Big|\mathcal{L}^* (i \Delta W) \overline{ v_{\bot}}(s) \Big| \;
 dx ds \\
 \lesssim & \;  e^{-c_2t} + \big\|v_{\bot}(t)\big\|_{L^{\infty}\dot
 H^1}  \lesssim e^{-\frac{c_1+c_2}{2}t}.
\end{align*}
This,  together with \eqref{betagammader} and Lemma \ref{summation},
implies that
\begin{align}\label{beta:optcontrol}
\big| \beta(t)\big|  \lesssim e^{-\frac{c_1+c_2}{2}t}.
\end{align}

 By the analogue
analysis, we can obtain the estimate of $\gamma(t)$
\begin{align}\label{gamma:optcontrol}
\big| \gamma(t)\big|  \lesssim e^{-\frac{c_1+c_2}{2}t}.
\end{align}
Thus $v$, $g$ satisfies \eqref{linEquat} and \eqref{baseestimate}
with $c_1$ replaced by $c'_1=\frac{c_1+c_2}{2}$. An iteration
argument gives
\begin{align*}
\big\|v(t)\big\|_{\dot H^1} \leq C e^{-c^-_2t}.
\end{align*}
Hence we  conclude the proof of the case $c_2\leq e_0$ or $c_2
> e_0, a=0$.

\noindent{\bf Step 4: Proof of the case $c_2>e_0$, and $a\not= 0$.}
By Step 2, if $c_1> e_0$, we can take $a=0$, so that we may assume
that $c_1\leq e_0$. Let
\begin{align*}
\overline{v} (t) := v(t) -ae^{-e_0t}\YYY_+,
\end{align*}
then
\begin{align*}
\partial_t \overline{v} (t) + \mathcal{L} \overline{v} (t) = g(t),
\quad \big\|\overline{v} (t) \big\|_{\dot H^1} \leq C e^{-c_1t}
\end{align*}
and by \eqref{alphap:bound2}, we have
\begin{align*}
\lim_{t\rightarrow+\infty}e^{e_0t}\overline{\alpha}_+(t) =0,
\end{align*}
where $\overline{\alpha}_+(t) = B(\overline{v}(t), \YYY_-)$ is the
coefficient of $\YYY_+$ in the decomposition of
$\overline{\alpha}_+(t)$. Thus $\overline{\alpha}_+(t)$ and $g$
satisfy all the assumptions of Step 3, hence, we have
\begin{align*}
\big\|v(t)-a e^{-e_0t}\YYY_+\big\|_{\dot H^1} \leq C e^{-c^-_2t}.
\end{align*}
This completes the proof.\end{proof}

\subsection{Uniqueness.}

\begin{proposition}\label{uniqueness}
Let $u$ be a radial solution of \eqref{har}, defined on $[t_0,
+\infty)$, such that $E(u)=E(W)$ and
\begin{equation}\label{uconvGSexponential}
\aligned \exists c, C>0,\; \big\|u-W \big\|_{\dot H^1} \leq C
e^{-ct}, \forall\; t\geq t_0 .
\endaligned
\end{equation}
Then there exists a unique $a \in \R$ such that $u=U^a$ is the
solution of \eqref{har} defined in Proposition
\ref{existence:thresholdsolution}.
\end{proposition}

\begin{proof} Let $u=W+v$ be a radial solution of \eqref{har} for $t\geq
t_0$ satisfying \eqref{uconvGSexponential}. Then $v$ satisfies the
linearized equation \eqref{linearequat}, that is,
\begin{equation}
\aligned
\partial_t v + \mathcal{L} v=R(v),
\endaligned
\end{equation}

\noindent{\bf Step 1: } We will show that there exists $A\in \R$
such that
\begin{align}\label{uconvappsol:e0decay}
\forall \eta>0,\; \big\|v(t)-ae^{-e_0t}\YYY_+\big\|_{\dot H^1} +
\big\|v(t)-ae^{-e_0t}\YYY_+\big\|_{l(t, +\infty)} \leq C_{\eta}
e^{-2^-e_0t}.
\end{align}
Indeed, we only need show
\begin{align}\label{vdecay:e0}
  \big\|v(t)\big\|_{\dot H^1} \leq C e^{-e_0t}, &
\;\; \big\|R(v(t))\big\|_{L^{\frac{2d}{d+2}}(\R^d)}+  \big\| \nabla
R(v(t))\big\|_{N(t, +\infty)} \leq  Ce^{-2e_0 t},
\end{align}
which together with Proposition \ref{improved decay} gives
\eqref{uconvappsol:e0decay}. By Lemma
\ref{linearoperator:prelimestimate} and Lemma \ref{linearstrich} it
suffices to show the first estimate.

By \eqref{uconvGSexponential}, Lemma
\ref{linearoperator:prelimestimate} and Lemma \ref{linearstrich}, we
have
\begin{align*}
\big\|R(v(t))\big\|_{L^{\frac{2d}{d+2}}(\R^d)}+  \big\| \nabla
R(v(t))\big\|_{N(t, +\infty)} \leq  Ce^{-2c t}.
\end{align*}
Then Proposition \ref{improved decay} gives that
\begin{align*}
  \big\|v(t)\big\|_{\dot H^1} \leq C \Big(  e^{-e_0t} + e^{-\frac32ct}\Big).
\end{align*}
If $ \dfrac32\;c\geq e_0$, we obtain the first inequality in
\eqref{vdecay:e0}. If not, an iteration argument gives the first
inequality in \eqref{vdecay:e0}.

\noindent{\bf Step 2: } We will use the induction argument to show
that
\begin{align}\label{uconvappsol:fastdecay}
\forall\; m >0,\; \big\|u(t)-U^a(t)\big\|_{\dot H^1} +
\big\|u(t)-U^a(t)\big\|_{l(t, +\infty)} \leq   e^{-mt}.
\end{align}

By Proposition \ref{existence:thresholdsolution} and Step 1,
\eqref{uconvappsol:fastdecay} holds with $m=\frac32e_0$. Let us
assume that \eqref{uconvappsol:fastdecay} holds with $m=m_1 > e_0$,
we will show that it also holds for $m=m_1 + \frac12 e_0$, which
yields \eqref{uconvappsol:fastdecay} by iteration.

Let  $u^a(t)=U^{a}-W$, then we have
\begin{align*}
\partial_t \big( v-u^a \big) + \mathcal{L} \big( v-u^a \big) =
-R(v)+ R(u^a).
\end{align*}
By assumption, we have \begin{align*} \big\|v(t)-u^a(t)\big\|_{\dot
H^1} + \big\|v(t)-u^a(t)\big\|_{l(t, +\infty)} \leq C  e^{-m_1t}.
\end{align*}
By Lemma \ref{linearoperator:prelimestimate} and Lemma
\ref{summation}, we have
\begin{align*}
\big\|R(v(t))-R(u^a(t))\big\|_{L^{\frac{2d}{d+2}}(\R^d)}+  \big\|
\nabla \big( R(v(t))-R(u^a(t))\big)\big\|_{N(t, +\infty)} \leq
Ce^{-(m_1 +e_0) t}.
\end{align*}
By Proposition \ref{improved decay} and Lemma \ref{vstrich}, we have
\begin{align*}
\forall m >0,\; \big\|v(t)-u^a(t)\big\|_{\dot H^1} +
\big\|v(t)-u^a(t)\big\|_{l(t, +\infty)} \leq C
e^{-(m_1+\frac{3}{4}e_0)t}\leq   e^{-(m_1+\frac{1}{2}e_0)t}.
\end{align*}

\noindent{\bf Step 3: End of the proof. } Using
\eqref{uconvappsol:fastdecay} with $m=(k_0 +1)e_0$, (where $k_0$ is
given by Proposition \ref{existence:thresholdsolution}), we get that
for large $t>0$,
\begin{align*}
 \big\|u(t)-U^a_{k_0}(t)\big\|_{l(t, +\infty)} \leq   e^{-(k_0+\frac12)t}.
\end{align*}
By the uniqueness in Proposition \ref{existence:thresholdsolution},
we get that $u=U^a$, so we complete the proof of Proposition \label{uniqueness}. \end{proof}

\begin{corollary}\label{cor:uniqueness}
For any $a\not =0$, there exists $T_a \in \R$  such that
\begin{equation}\label{ua:classification}
\left\{ \begin{array}{rl} U^a(t)=W^{+}(t-T_a), & \text{if}\;\; a>0,\\
U^a(t)=W^{-}(t-T_a), & \text{if}\; \; a<0.
\end{array} \right.
\end{equation}
\end{corollary}

\begin{proof}Let $a\not =0$ and choose $T_a$ such that
$|a|e^{-e_0T_a}=1$. By \eqref{difference:Energy}, we have
\begin{align}
&\big\|U^a(t+T_a)-W - \text{sign} (a) e^{-e_0t}\YYY_+\big\|_{\dot
H^1} \nonumber \\
= & \big\|U^a(t+T_a)-W-a
e^{-e_0(t+T_a)}\YYY_+\big\|_{\dot H^1} \leq   e^{-\frac32 e_0
(t+T_a)} \leq C e^{-\frac32 e_0 t}. \label{ua:translation}
\end{align}
In addition, $U^a(t+T_a)$ satisfies the assumption of Proposition
\ref{uniqueness}, this implies that there exists $\widetilde{a}$
such that \begin{align*} U^a(t+T_a) =U^{\widetilde{a}}(t).
\end{align*}
By \eqref{ua:translation} and \eqref{difference:Energy}, we know
that  $\widetilde{a}=1$ if $a>0$ and $\widetilde{a}=-1$ if $a=-1$,
this completes the proof.\end{proof}

\subsection{Proof of Theorem \ref{classification}.} Let us first show
(a). Let $u$ be a radial solution to \eqref{har} on $[-T_-, T_+]$
such that, replacing if necessary $u(t)$ by $\overline{u}(-t)$
$$E(u_0)=E(W), \big\|\nabla u_0 \big\|_{L^2} < \big\|\nabla W
\big\|_{L^2},\;\text{and}\; \big\|u\big\|_{Z(0, T_+)}=+\infty.$$
Then by Proposition \ref{expdecay:subcase}, we know that
$T_-=T_+=\infty$ and there exist $\theta_0\in \R, \mu_0>0$ and $c,
C>0$ such that
\begin{align*} \big\|u(t)-W_{\theta_0, \mu_0}\big\|_{\dot H^1} \leq
C e^{-ct}.
\end{align*}
This shows that $u_{-\theta_0, \frac{1}{\mu_0}}$ fulfills the
assumptions of Proposition \ref{uniqueness}. By $\big\|u\big\|_{\dot
H^1} < \big\| W\big\|_{\dot H^1}$ and Corollary
\ref{cor:uniqueness}, we know that there exist  $a<0$ and $T_a$ such
that
\begin{align*}
u_{-\theta_0, \frac{1}{\mu_0}}(t)=U^a(t) = W^{-}(t-T_a),
\end{align*}
which  shows (a).

\vskip0.2cm
(b) is a direct consequence of the variational characterization of $W$.

\vskip0.2cm

The proof of (c) is similar to that of (a). Let $u$ be a radial
solution of \eqref{har} defined on $[0, +\infty)$  (Replacing if
necessary $u(t)$ by $\overline{u}(-t)$) such that
\begin{align*}
E(u)=E(W), \; \big\| \nabla u_0\big\|_{L^2}> \big\|\nabla
W\big\|_{L^2},\; \text{and}\; u_0 \in L^2.
\end{align*}
Due to Proposition \ref{expdecay:supercase}, there exist
$\theta_0\in \R, \mu_0>0$ and $c, C>0$ such that \begin{align*}
\big\|u(t)-W_{\theta_0, \mu_0}\big\|_{\dot H^1} \leq C e^{-ct}.
\end{align*}
Similar to the proof of (a), we know that there exist  $a>0$ and
$T_a$ such that
\begin{align*}
u_{-\theta_0, \frac{1}{\mu_0}}(t)=U^a(t) = W^{+}(t-T_a).
\end{align*}
This shows (c). The proof of Theorem \ref{classification} is thus
complete. \qed

%
%
%
%

\appendix

\section{Uniqueness of the ground state in $L^{\frac{2d}{d-2}}$}\label{S:movplan}
In this Appendix, we will use the moving plane method to show the
uniqueness of the positive solution of the following nonlocal
elliptic equation in $L^{\frac{2d}{d-2}}$

\begin{equation}\label{ellipticwithnonlocal}
- \Delta \omega  =  \left(\frac{1}{|x|^4}*|\omega|^2\right)\omega
,\quad (t,x)\in \R \times \R^d, d \geq 5.
\end{equation}

\begin{proposition}\label{T:uniqueness:diff form}
The positive solution $\omega(x)$ of \eqref{ellipticwithnonlocal} in
$ L^{\frac{2d}{d-2}}(\R^d)$ is radially symmetry and decreasing
about some point $x_0$ and unique. Therefore the positive solution
$u(x)$ of \eqref{ellipticwithnonlocal} in $
L^{\frac{2d}{d-2}}(\R^d)$  has the form
\begin{equation}\label{standform}
c_0\left( \frac{t}{t^2 + |x-x_0|^2} \right)^{\frac{d-2}{2}}
\end{equation}
with some positive constant $c_0$ and $t$.
\end{proposition}
\begin{remark}\begin{enumerate}
\item The method of the moving planes was invented by Alexanderov in
\cite{Ale:58:MovPlan}. Later, it was further developed by Serrin
\cite{Ser:71:MovPlan}, Gidas, Ni and Nirenberg
\cite{GidNN:81:MovPlan}, Caffarelli, Gidas and Spruck
\cite{CafGS:89:MovPlan} in the study of the classification of
semilinear elliptic equation $ - \Delta \omega =
\omega^{\frac{d+2}{d-2}}, x\in \R^d $. Subsequently, Chen and Li
\cite{CheL:91:MovPlan} and Li \cite{Li:96:MovPlan} simplified its
proof. Recently, Wei and Xu \cite{WeiX:99:MovPlan} and Chen, Li and
Ou \cite{CheLO:06:MovPlan} generalize the classification result to
the solutions of higher order conformally invariant equations $
\big(- \Delta \big)^{\alpha/2}\omega =
\omega^{\frac{d+\alpha}{d-\alpha}},\; x\in \R^d,\; 0<\alpha <d. $ Li
\cite{Li:04:MovSph} use the method of moving spheres to obtain the
same classification result as that in \cite{CheLO:06:MovPlan}. For
other applications, please refer to \cite{ChaY:97:MovPlan,
CheL:97:MovPlan, CheL:05:MovPlan, CheLO:05:MovPlan, FraL:09:uniq1,
FraL:09:uniq2, MaZ:10:MovPlan, Len:08:uniq}.

\item The uniqueness still holds for the positive solution in
$L^{\frac{2d}{d-2}}_{loc}(\R^d)$ by the analogue analysis as that in
\cite{CheLO:06:MovPlan}.

\item The results still hold for the fractional Laplacian equation
($0<\alpha<d$)
\begin{equation*}
\big(- \Delta\big)^{\alpha/2} \omega =
\left(\frac{1}{|x|^{2\alpha}}*|\omega|^2\right)\omega ,\quad
(t,x)\in \R \times \R^d.
\end{equation*}That is, the positive solution in $L^{2d/(d-\alpha)}_{loc}$ is radially symmetry and decreasing about some
point $x_0$ and unique.
\end{enumerate}
\end{remark}

To do so, we first show the covariance of
\eqref{ellipticwithnonlocal} under the Kelvin transform. Denote
$K\omega$ the Kelvin transform of $\omega$, that is
\begin{equation*}
K\omega=\frac{1}{|x|^{n-2}}\omega\left(\frac{x}{|x|^2}\right).
\end{equation*}
\begin{lemma} Let $\omega(x)$ be a solution of
\eqref{ellipticwithnonlocal}, then $w=K\omega$ is still a solution
of \eqref{ellipticwithnonlocal}.
\end{lemma}
\begin{proof} Note that
\begin{align*}
\omega(x)=\frac{1}{|x|^{d-2}}w\left(\frac{x}{|x|^2}\right),
\end{align*}
It is well known that
\begin{equation*}\label{linear-term}
\Delta \omega(x) = |x|^{-2-d}( \Delta w
)\left(\frac{x}{|x|^2}\right).
\end{equation*}

Now we check the nonlinearity,
\begin{align*}
\left(\frac{1}{|\cdot|^4}*|\omega|^2\right)(x) = &
\int_{\R^d}\frac{|\omega(y)|^2}{|x-y|^4} \; dy =
\int_{\R^d}\frac{\big|w(\frac{y}{|y|^2})\big|^2}{|x-y|^4
|y|^{2(d-2)}} \; dy
\\
= & \int_{\R^d}\frac{\big|w(y)\big|^2}{|x-\frac{y}{|y|^2}|^4} \;
|y|^{2(d-2)}\; |y|^{-2d} \; dy\\
= &
|x|^{-4}\int_{\R^d}\frac{\big|w(y)\big|^2}{|\frac{x}{|x|^2}-y|^4} \;
 \; dy = |x|^{-4}
 \left(\frac{1}{|\cdot|^4}*|w|^2\right)\left(\frac{x}{|x|^2}\right)
\end{align*}
where we use the following identity in the fourth step
\begin{align*}
 |y|^4  \left| x- \frac{y}{|y|^2}  \right|^4 =  |x|^4
 \left| \frac{x}{|x|^2} - y \right|^4 .
\end{align*}
Therefore, by \eqref{ellipticwithnonlocal}, we have
\begin{align*}
|x|^{-2-d}( -\Delta w )\left(\frac{x}{|x|^2}\right) = |x|^{-4}
 \left(\frac{1}{|\cdot|^4}*|w|^2\right)\left(\frac{x}{|x|^2}\right)
 \times \frac{1}{|x|^{d-2}}w\left(\frac{x}{|x|^2}\right).
\end{align*}
This shows that $w$ is also the solution of
\eqref{ellipticwithnonlocal}, that is, the equation
\eqref{ellipticwithnonlocal} is covariant under the Kelvin
transform.
\end{proof}

Now, we transfer \eqref{ellipticwithnonlocal} into the equivalent
integral system \eqref{integralsystem}, then make use of the moving
plane method in its global form to show that the positive solution
of \eqref{ellipticwithnonlocal} in $L^{\frac{2d}{d-2}}$ is radially
symmetric and monotone decreasing about some point $x_0\in \R^d$. For this purpose, we
first introduct some notations. For $x=(x_1, x_2, \ldots, x_d)\in \R^d$,
$\lambda \in \R$, we define $x^{\lambda}=(2\lambda- x_1, x_2,
\ldots, x_d),$ and
\begin{align*}
\omega_{\lambda}(x)=\omega(x^{\lambda}), \quad
v_{\lambda}(x)=v(x^{\lambda}).
\end{align*}
Let $\Sigma_{\lambda}=  \{x=(x_1, x_2, \ldots, x_d)\in \R^d, x_1
\geq \lambda\}$, we denote
\begin{align*}
\Sigma^{\omega}_{\lambda}= &  \{x \in \Sigma_{\lambda}, \omega(x)<
\omega_{\lambda}(x)\};\;\; \overline{\Sigma^{\omega}_{\lambda}}=
\{x \in
\Sigma_{\lambda}, \omega(x) \leq \omega_{\lambda}(x)\}, \\
\Sigma^v_{\lambda}= &  \{x \in \Sigma_{\lambda}, v(x)<
v_{\lambda}(x)\}.
\end{align*}
We denote the complement of $ \Sigma_{\lambda}$ in $\R^d$ by
$\Sigma_{\lambda}^c $, and the reflection of
$\Sigma^{\omega}_{\lambda}$ about the plane $x_1=\lambda$ by $\big(
\Sigma^{\omega}_{\lambda}\big)^*$.

Let $v(x)=|x|^{-4}*|\omega|^2$, then \eqref{ellipticwithnonlocal} is
equivalent to
\begin{align}\label{integralsystem}
\omega(x)=\displaystyle \int_{\R^d}
\frac{\omega(y)v(y)}{|x-y|^{d-2}} \; dy, \quad v(x)=\displaystyle
\int_{\R^d} \frac{\omega(y)^2 }{|x-y|^{4}} \; dy.
\end{align}
By $\omega \in L^{\frac{2d}{d-2}}$ and Lemma \ref{L:hardy}, we know
that $v\in L^{d/2}$. Therefore, it suffices to show the radial
symmetry and monotone decreasing of the positive solution
$(\omega,v)$ of \eqref{integralsystem} in $L^{2d/(d-2)}\times
L^{d/2}$ in order to show the radial symmetry and monotone
decreasing of the positive solution $\omega$ of
\eqref{ellipticwithnonlocal} in $L^{2d/(d-2)}$.

To do so, we decompose $\omega_{\lambda}(x)$, $\omega(x)$ in
$\Sigma_{\lambda}$ and $v_{\lambda}(x)$, $v(x)$ in
$\Sigma_{\lambda}$ as follows:

\begin{lemma}\label{L:decomp}
For any solution $(\omega,v)$ of \eqref{integralsystem}, we have
\begin{align}
\omega_{\lambda}(x)-\omega(x) = & \int_{\Sigma_{\lambda}}  \left(
\frac{1}{|x-y|^{d-2}} -\frac{1}{|x^{\lambda}-y|^{d-2}} \right)\Big(
\omega_{\lambda}(y) v_{\lambda}(y) - \omega(y) v(y) \Big)\; dy, \label{decom-u} \\
v_{\lambda}(x)-v(x) = & \int_{\Sigma_{\lambda}}  \left(
\frac{1}{|x-y|^{4}} -\frac{1}{|x^{\lambda}-y|^{4}} \right)\Big(
\omega_{\lambda}(y)^2  - \omega(y)^2 \Big)\; dy. \label{decom-v}
\end{align}
\end{lemma}

\begin{proof} By \eqref{integralsystem} and the fact that $\left|x-y^{\lambda}\right|=\left|x^{\lambda}-y\right|$, we have
\begin{align}
\omega(x)= &  \int_{\R^d} \frac{\omega(y)v(y)}{|x-y|^{d-2}} \; dy
 =   \int_{\sum_{\lambda}} \frac{\omega(y)v(y)}{|x-y|^{d-2}} \; dy +   \int_{\sum^c_{\lambda}} \frac{\omega(y)v(y)}{|x-y|^{d-2}} \;
 dy \nonumber\\
 = &  \int_{\sum_{\lambda}} \frac{\omega(y)v(y)}{|x-y|^{d-2}} \; dy +   \int_{\sum_{\lambda}} \frac{\omega(y^{\lambda})v(y^{\lambda})}{|x-y^{\lambda}|^{d-2}} \;
 dy \nonumber\\
  = &  \int_{\sum_{\lambda}} \frac{\omega(y)v(y)}{|x-y|^{d-2}} \; dy +   \int_{\sum_{\lambda}} \frac{\omega(y^{\lambda})v(y^{\lambda})}{|x^{\lambda}-y|^{d-2}} \;
 dy \nonumber\\
 = &  \int_{\sum_{\lambda}} \left( \frac{\omega(y)v(y)}{|x-y|^{d-2}} +  \frac{\omega_{\lambda}(y)v_{\lambda}(y)}{|x^{\lambda}-y|^{d-2}} \right)\;
 dy. \label{dec:u:sigma}
\end{align}
This implies that
\begin{align}\label{dec:u-lamb:sigma}
\omega_{\lambda}(x)= \omega(x^{\lambda})= &  \int_{\sum_{\lambda}}
\left( \frac{\omega(y)v(y)}{|x^{\lambda}-y|^{d-2}} +
\frac{\omega_{\lambda}(y)v_{\lambda}(y)}{|x-y|^{d-2}} \right) \;
 dy.
\end{align}

By \eqref{dec:u:sigma} and \eqref{dec:u-lamb:sigma}, we get
\eqref{decom-u}. By the same way, we can show \eqref{decom-v}.
\end{proof}

Based on the above preliminaries, we can prove that

\begin{proposition}\label{P:rad-dec}
Let $(\omega,v)$ be the positive solution of \eqref{integralsystem}
in $ L^{2d/(d-2)} \times L^{d/2}$. Then $\omega$ and $v$ are
radially symmetry and decreasing about some point $x_0\R^d$.
\end{proposition}

\begin{proof} The proof consist of three steps.
\begin{enumerate}
\item[Step 1:] We show that there exists an $N>0$ such that for
any $\lambda<-N$, we have
\begin{equation}\label{basicfact}
\omega(x)\geq \omega_{\lambda}(x); \; \text{and}\; v(x) \geq
v_{\lambda}(x), \forall \; x\in \Sigma_{\lambda}.
\end{equation}
\item[Step 2:] We move the plane continuously from $\lambda<-N$ to
the right as long as \eqref{basicfact} holds.  We show that if the
plane stops at $x_1=\lambda_0$ for some $\lambda_0$, then
$\omega(x)$ and $v(x)$ must be symmetric and monotone about the
plane $x_1=\lambda_0$; otherwise, we can move the plane all the way
to the right.
\item[Step 3:] By Step 1, we know that the plane cannot move
all through to the right in Step 2. That is,  the plane will
eventually stop at some finite point. In fact, by the similar
analysis as that in Step 1 and Step 2, there exists a large $M$,
such that for $\lambda
> M$, we have
\begin{equation*}
\omega(x)\leq \omega_{\lambda}(x); \; \text{and}\; v(x) \leq
v_{\lambda}(x), \forall \; x\in \Sigma_{\lambda}.
\end{equation*}
Now  we can move the plane continuously  from  $\lambda> M$ to the
left as long the above fact holds. The planes moved from the left
and the right will eventually meet at some point. Finally, since the
direction of $x_1$ can be chosen arbitrarily, we deduce that
$\omega(x)$ and $v(x)$ must be radially symmetric and decreasing
about some point.
\end{enumerate}

According to the above analysis, we only need to show the facts in
Step 1 and Step 2.

\noindent{\bf Step 1: } For the sufficiently negative
value of $\lambda$, we can show that  $\Sigma^{\omega}_{\lambda}$ and
$\Sigma^v_{\lambda}$ must be both empty. In fact, for any $x\in
\Sigma_{\lambda}$, we have
\begin{align*}
\omega_{\lambda}(x) - \omega(x) = & \int_{\sum_{\lambda}}  \left(
\frac{1}{|x-y|^{d-2}} -\frac{1}{|x^{\lambda}-y|^{d-2}} \right)\Big(
\omega_{\lambda}(y) v_{\lambda}(y) - \omega(y) v(y) \Big)\; dy \\
\leq & \int_{  \{y\in\; \Sigma_{\lambda}:\; \omega v \leq
\omega_{\lambda}v_{\lambda} \}}
 \frac{1}{|x-y|^{d-2}} \Big( \omega_{\lambda}(y) v_{\lambda}(y) - \omega(y)
v(y) \Big)\; dy \\
= & \int_{  \{y\in\; \Sigma_{\lambda}:\; \omega v \leq
\omega_{\lambda}v_{\lambda} \}}
 \frac{   \omega(y)\Big(  v_{\lambda}(y) -
v(y) \Big) +  v_{\lambda}(y)\Big(  \omega_{\lambda}(y) -
\omega(y) \Big)  }{|x-y|^{d-2}}\; dy \\
= &  \int_{  \{y\in\; \Sigma_{\lambda}:\; \omega v \leq
\omega_{\lambda}v_{\lambda}, \atop \omega <\omega_{\lambda}, v<
v_{\lambda} \}}
 \frac{   \omega(y)\Big(  v_{\lambda}(y) -
v(y) \Big) +  v_{\lambda}(y)\Big(  \omega_{\lambda}(y) -
\omega(y) \Big)  }{|x-y|^{d-2}}\; dy \\
 &+ \int_{  \{y\in\; \Sigma_{\lambda}:\; \omega v \leq
\omega_{\lambda}v_{\lambda}, \atop \omega < \omega_{\lambda}, v \geq
v_{\lambda} \}}
 \frac{  \omega(y)\Big(  v_{\lambda}(y) -
v(y) \Big) +  v_{\lambda}(y)\Big(  \omega_{\lambda}(y) -
\omega(y) \Big)  }{|x-y|^{d-2}}\; dy \\
& + \int_{  \{y\in\; \Sigma_{\lambda}:\; \omega v \leq
\omega_{\lambda}v_{\lambda}, \atop \omega \geq \omega_{\lambda}, v<
v_{\lambda} \}}
 \frac{   \omega(y)\Big(  v_{\lambda}(y) -
v(y) \Big) +  v_{\lambda}(y)\Big( \omega_{\lambda}(y) -
\omega(y) \Big)  }{|x-y|^{d-2}}\; dy \\
\leq & \int_{ \sum^v_{\lambda} }
 \frac{\omega(y)}{|x-y|^{d-2}}  \Big(  v_{\lambda}(y) -
v(y) \Big)\; dy + \int_{ \sum^{\omega}_{\lambda} }
 \frac{v_{\lambda}(y)}{|x-y|^{d-2}}  \Big( \omega_{\lambda}(y) -
\omega(y) \Big)\; dy.
\end{align*}
Hence, by Lemma \ref{L:hardy} and H\"{o}lder's inequality, we obtain
\begin{align}
\big\| \omega_{\lambda} - \omega
\big\|_{L^{\frac{2d}{d-2}}(\Sigma^{\omega}_{\lambda})}  \leq  &\; C
\big\| \omega
\big(v_{\lambda}-v\big)\big\|_{L^{\frac{2d}{d+2}}(\Sigma^v_{\lambda})}
+ C \big\|v_{\lambda}\big( \omega_{\lambda} - \omega \big)
\big\|_{L^{\frac{2d}{d+2}}(\Sigma^{\omega}_{\lambda})} \nonumber \\
\leq & \; C
\big\|\omega\big\|_{L^{\frac{2d}{d-2}}(\Sigma^v_{\lambda})}
\big\|v_{\lambda}-v \big\|_{L^{\frac{d}{2}}(\Sigma^v_{\lambda})} + C
\big\|v_{\lambda}\big\|_{L^{\frac{d}{2}}(\Sigma^{\omega}_{\lambda})}
\big\|\omega_{\lambda} - \omega
\big\|_{L^{\frac{2d}{d-2}}(\Sigma^{\omega}_{\lambda})}.\label{est:u}
\end{align}

In the same argument as above, for $x\in \Sigma_{\lambda}$, we also have
\begin{align*}
v_{\lambda}(x)-v(x) \leq & 2 \int_{ \Sigma^{\omega}_{\lambda} }
 \frac{\omega_{\lambda}(y)}{|x-y|^{4}}  \Big(  \omega_{\lambda}(y) -
\omega(y) \Big)\; dy,
\end{align*}
and
\begin{align}
\big\| v_{\lambda} - v\big\|_{L^{\frac{d}{2}}(\Sigma^v_{\lambda})}
\leq & C \big\| \omega_{\lambda}
\big(\omega_{\lambda}-\omega\big)\big\|_{L^{\frac{d}{d-2}}(\Sigma^{\omega}_{\lambda})}
 \leq  C
\big\|\omega_{\lambda}\big\|_{L^{\frac{2d}{d-2}}(\Sigma^{\omega}_{\lambda})}
\big\|\omega_{\lambda}- \omega
\big\|_{L^{\frac{2d}{d-2}}(\Sigma^{\omega}_{\lambda})}.
\label{est:v}
\end{align}
Hence, taking \eqref{est:v} into \eqref{est:u}, we obtain
\begin{align}
&\big\|  \omega_{\lambda} -  \omega
\big\|_{L^{\frac{2d}{d-2}}(\Sigma^{\omega}_{\lambda})} \nonumber \\
\leq &\;
C\big\|\omega\big\|_{L^{\frac{2d}{d-2}}(\Sigma^v_{\lambda})}
\big\|\omega_{\lambda}\big\|_{L^{\frac{2d}{d-2}}(\Sigma^{\omega}_{\lambda})}
\big\|\omega_{\lambda}-\omega\big\|_{L^{\frac{2d}{d-2}}(\Sigma^{\omega}_{\lambda})}
+ C
\big\|v_{\lambda}\big\|_{L^{\frac{d}{2}}(\Sigma^{\omega}_{\lambda})}
\big\|\omega_{\lambda} -\omega
\big\|_{L^{\frac{2d}{d-2}}(\Sigma^{\omega}_{\lambda})}  \nonumber\\
\leq & \; C
\big\|\omega\big\|_{L^{\frac{2d}{d-2}}(\Sigma_{\lambda})}
\big\|\omega_{\lambda}\big\|_{L^{\frac{2d}{d-2}}(\Sigma_{\lambda})}
\big\|\omega_{\lambda}-\omega\big\|_{L^{\frac{2d}{d-2}}(\Sigma^{\omega}_{\lambda})}
+ C \big\|v_{\lambda}\big\|_{L^{\frac{d}{2}}(\Sigma_{\lambda})}
\big\|\omega_{\lambda} - \omega
\big\|_{L^{\frac{2d}{d-2}}(\Sigma^{\omega}_{\lambda})} \nonumber\\
\leq &   \left( C
\big\|\omega\big\|_{L^{\frac{2d}{d-2}}(\Sigma_{\lambda})}
\big\|\omega \big\|_{L^{\frac{2d}{d-2}}(\Sigma^c_{\lambda})}  +  C
\big\|v \big\|_{L^{\frac{d}{2}}(\Sigma^c_{\lambda})} \right)
\big\|\omega_{\lambda} - \omega
\big\|_{L^{\frac{2d}{d-2}}(\Sigma^u_{\lambda})}. \label{est:u:final}
\end{align}

By $\omega\in L^{2d/(d-2)}$ and $v\in L^{d/2}$ and the dominated
convergence theorem, we can choose $N$ sufficiently large such that
$\lambda< -N,$   and
\begin{equation*}
C \big\|\omega\big\|_{L^{\frac{2d}{d-2}}(\Sigma_{\lambda})}
\big\|\omega \big\|_{L^{\frac{2d}{d-2}}(\Sigma^c_{\lambda})} +  C
\big\|v \big\|_{L^{\frac{d}{2}}(\Sigma^c_{\lambda})} \leq \frac12.
\end{equation*}
Thus \eqref{est:u:final} implies that
\begin{equation*}
\big\| \omega_{\lambda} - \omega
\big\|_{L^{\frac{2d}{d-2}}(\Sigma^{\omega}_{\lambda})} =0.
\end{equation*}
This implies that $\Sigma^{\omega}_{\lambda}$ must be the set with
zero measure, hence must be empty, up to a set with zero measure. By
\eqref{est:v}, $\Sigma^v_{\lambda}$ also must be empty.

\noindent{\bf Step 2: } We move the plane $x_1=\lambda$ to the right
as long as \eqref{basicfact} holds. Suppose that at some
$\lambda_0$, we have
\begin{align*}
\omega(x) & \geq \omega_{\lambda_0}(x), v(x)\geq v_{\lambda_0}(x),
\; \text{on}\;
\Sigma_{\lambda_0}, \\
\text{but}\; \omega(x) & \not \equiv \omega_{\lambda_0}(x)
\;\text{or}\; v(x) \not \equiv v_{\lambda_0}(x) \; \text{on}\;
\Sigma_{\lambda_0}.
\end{align*}
We show that the plane can be moved further to the right. More
precisely,  there exists $\epsilon=\epsilon(n,\omega,v)$ such that
$\omega(x)\geq \omega_{\lambda}(x)$ and $v(x)\geq v_{\lambda}(x)$ on
$\Sigma_{\lambda}$ for all $\lambda \in [\lambda_0, \lambda_0 +
\epsilon)$.

In the case
\begin{equation*}
v(x) \not \equiv v_{\lambda_0}(x) \; \text{on}\; \Sigma_{\lambda_0}.
\end{equation*}
By \eqref{decom-u}, we have that $\omega(x)> \omega_{\lambda_0}(x)$
in the interior of $\Sigma_{\lambda_0}$. Note that
\begin{equation*}\big|\; \overline{\Sigma^{\omega}_{\lambda_0}}
\;\big|=0,\; \text{and}\; \lim_{\lambda\rightarrow \lambda_0}
\Sigma^{\omega}_{\lambda} \subseteq
\overline{\Sigma^{\omega}_{\lambda_0}}.
\end{equation*}

From \eqref{est:u} and \eqref{est:v}, we have
\begin{align*}
\big\| \omega_{\lambda} - \omega
\big\|_{L^{\frac{2d}{d-2}}(\Sigma^{\omega}_{\lambda})} \leq &\;
\left( C \big\|\omega\big\|_{L^{\frac{2d}{d-2}}(\Sigma_{\lambda})}
\big\|\omega
\big\|_{L^{\frac{2d}{d-2}}\big((\Sigma^{\omega}_{\lambda})^*\big)} +
C \big\|v
\big\|_{L^{\frac{d}{2}}\big((\Sigma^{\omega}_{\lambda})^*\big)}
\right) \big\|\omega_{\lambda} - \omega
\big\|_{L^{\frac{2d}{d-2}}(\Sigma^{\omega}_{\lambda})}.
\end{align*}
By $\omega\in L^{2d/(d-2)}$ and $v\in L^{d/2}$ and $\big|\;
\overline{\Sigma^{\omega}_{\lambda_0}} \; \big|=0$ and the dominated
convergence theorem, we can choose $\epsilon$ sufficiently small,
such that for all $\lambda \in [\lambda_0, \lambda_0 + \epsilon)$,
we have
\begin{equation*}
 C
\big\|\omega\big\|_{L^{\frac{2d}{d-2}}(\Sigma_{\lambda})}
\big\|\omega
\big\|_{L^{\frac{2d}{d-2}}\big((\Sigma^{\omega}_{\lambda})^*\big)} +
C \big\|v
\big\|_{L^{\frac{d}{2}}\big((\Sigma^{\omega}_{\lambda})^*\big)} \leq
\frac12.
\end{equation*}
This implies that $\big\| \omega_{\lambda} - \omega
\big\|_{L^{\frac{2d}{d-2}}(\Sigma^{\omega}_{\lambda})}=0$. Hence,
$\Sigma^{\omega}_{\lambda}$ must be empty for all $\lambda \in
[\lambda_0, \lambda_0 + \epsilon)$. It also implies that
$\Sigma^v_{\lambda}$ also is empty for all $\lambda \in [\lambda_0,
\lambda_0 + \epsilon)$.

In the case
\begin{equation*}
\omega(x) \not \equiv \omega_{\lambda_0}(x) \; \text{on}\;
\Sigma_{\lambda_0}.
\end{equation*}
By \eqref{decom-v}, we have that $v(x)> v_{\lambda_0}(x)$ in the
interior of $\Sigma_{\lambda_0}$. By the above analysis, we know
that $\Sigma^{\omega}_{\lambda}$ and $\Sigma^v_{\lambda}$ must also
be empty for all $\lambda \in [\lambda_0, \lambda_0 + \epsilon)$.
This completes the proof.
\end{proof}

Now we use the elliptic regularity theory to show that

\begin{proposition}\label{P:regularity}
Assume that $\omega$ is a positive solution of
\eqref{ellipticwithnonlocal} in $L^{2d/(d-2)}$. Then $\omega$ is
uniformly bounded in $\R^d$. Furthermore, $\omega$ is $C^{\infty}$
and
\begin{equation}\label{asympbehav}
\lim_{|x| \rightarrow +\infty} |x|^{d-2} \omega(x) = \omega_{\infty}
\end{equation}
for some positive number $\omega_{\infty}$.
\end{proposition}

\begin{proof}
{\bf Step 1: } We first show that $\omega$ is uniformly bounded and
continuous. For $A>0$, we define
\begin{equation*}
\Omega = \{x\in \R^n, \omega(x)>A\}  \quad \text{and}\quad \omega_{A}(x)=\begin{cases} \omega(x), & \Omega, \\
0, & \R^d\backslash \Omega.
\end{cases}
\end{equation*}
Hence \begin{equation}\label{est:small part} \omega - \omega_A \in
L^{2d/(d-2)}\cap L^{\infty},\;\text{for any}\; A>0.
\end{equation}

Since $\omega$ is a solution of \eqref{ellipticwithnonlocal}, we
have
\begin{equation*}
\omega(x) = \int_{\R^d}
\frac{\big(|\cdot|^{-4}*|\omega|^2\big)\omega(y)}{|x-y|^{d-2}} \;dy,
\; \forall \; x \in \R^d.
\end{equation*}
This  implies that for any $ x \in \Omega$
\begin{align*}
\omega_A(x) = & \int_{\R^d}
\frac{\big(|\cdot|^{-4}*|\omega|^2\big)\omega(y)}{|x-y|^{d-2}} \;dy, \\
= & \int_{\R^d}
\frac{\big(|\cdot|^{-4}*|\omega_A|^2\big)\omega_A(y)}{|x-y|^{d-2}}
\;dy + \int_{\R^d}
\frac{\big(|\cdot|^{-4}*|\omega-\omega_A|^2\big)\omega_A(y)}{|x-y|^{d-2}} \;dy \\
& + \int_{\R^d}
\frac{\big(|\cdot|^{-4}*|\omega_A|^2\big)(\omega-\omega_A)(y)}{|x-y|^{d-2}}
\;dy + \int_{\R^d}
\frac{\big(|\cdot|^{-4}*|\omega-\omega_A|^2\big)(\omega-\omega_A)(y)}{|x-y|^{d-2}}
\;dy.
\end{align*}

For any $r \geq \frac{2d}{d-2}$,  we have by Lemma \ref{L:hardy}
\begin{align}
\big\|\omega_A\big\|_{L^r} \leq &\; C
\big\|\omega_A\big\|^2_{L^{2d/(d-2)}} \big\|\omega_A\big\|_{L^r} + C
\big\|\omega_A\big\|_{L^{2d/(d-2)}}\big\|\omega-\omega_A\big\|_{L^{2d/(d-2)}}
\big\|\omega-\omega_A\big\|_{L^r} \nonumber\\
& +  C \big\|\omega_A\big\|^2_{L^{2d/(d-2)}}
\big\|\omega-\omega_A\big\|_{L^r} +  C
\big\|\omega-\omega_A\big\|^2_{L^{2d/(d-2)}}
\big\|\omega-\omega_A\big\|_{L^r}. \label{est:big part}
\end{align}

On one hand, from  $\omega\in L^{2d/(d-2)}$,  we can choose $A$
sufficiently large, such that
\begin{equation}\label{nonlinearterm}
C\big\|\omega_A\big\|^2_{L^{2d/(d-2)}}\leq \frac12.
\end{equation}
On the other hand, by  $\omega\in L^{2d/(d-2)}$ and \eqref{est:small
part}, we  easily verify that
\begin{equation}\label{resterm}
\big\|\omega_A\big\|_{L^{2d/(d-2)}} +
\big\|\omega-\omega_A\big\|_{L^{2d/(d-2)}} +
\big\|\omega-\omega_A\big\|_{L^r} \leq C(A).
\end{equation}

Taking \eqref{nonlinearterm} and \eqref{resterm} into \eqref{est:big
part}, we have for any $r\geq \frac{2d}{d-2}$
\begin{equation*}
\big\|\omega_A\big\|_{L^r} \leq  \frac{1}{2}
\big\|\omega_A\big\|_{L^r} + C(A),
\end{equation*}
which implies that $\omega_A\in L^r$ for any $r\geq \frac{2d}{d-2}$.
Therefore, we have $\omega\in L^r$ for any $r\geq \frac{2d}{d-2}$.
By Lemma \ref{L:hardy}, we have
\begin{equation*}
-\Delta \omega = (|\cdot|^{-4}*|\omega|^2)\omega \in L^{p}, \qquad
\text{for any}\quad p\geq \frac{2d}{d+2}.
\end{equation*}
From the $L^p$ theory and the Sobolev embedding theorem in
\cite{Ste:70:book}, we know that $\omega$ is uniformly bounded  and
$C^{0, \gamma}$  for all $0<\gamma<1$. In fact, we also have $\omega
\in C^{\infty}$ from Theorem 4.4.8 in \cite{Caz:Semiell}.

\vskip0.2cm

\noindent{\bf Step 2: } We show that $\omega$ has the asymptotic
behavior at infinity. We prove it by contradiction. Consider the
Kelvin transform solution:
\begin{equation*}
U(x)=\frac{1}{|x|^{d-2}}\omega\left(\frac{x}{|x|^2}\right)\quad
\Longrightarrow \quad |x|^{d-2}\;\omega(x)=U\left(\frac{x}{|x|^2}\right).
\end{equation*}
Applying Proposition \ref{P:rad-dec} to  $U(x)$,  we conclude that
$U(x)$ must be radially symmetric about some point and continuity.
Hence
\begin{equation*}  \lim_{|x| \rightarrow +\infty}
|x|^{d-2} \omega(x) = U(0)>0.
\end{equation*}This completes the proof of
Proposition \ref{P:regularity}.
\end{proof}

Finally, the uniqueness in Proposition \ref{T:uniqueness:diff form}
comes from the scaling covariance of \eqref{ellipticwithnonlocal}
and the uniqueness of ODE theory. This finishes the proof of
Proposition \ref{T:uniqueness:diff form}.\qed

\section{Coercivity of $\Phi$ on $H^{\bot} \cap \dot H^1_{rad}$}\label{appen CoerH}
In this Appendix, we prove Proposition \ref{coerH}. We divide the
proof into two steps.

\noindent{\bf Step 1: Nonnegative.} We claim that for any  function
$h \in \dot H^1$, $\big(h, W\big)_{\dot H^1}=0$, there exists
\begin{equation*}
\aligned    \Phi(h) \geq  0.
\endaligned
\end{equation*}

Indeed, let
\begin{equation}\label{HLS-fomula}
\aligned    I(u)= \frac{\big\|\nabla u \big\|^4_2}{\big\|\nabla
W\big\|^4_2} - \frac{\displaystyle \iint_{\R^d\times\R^d}
\frac{\big|u(x)\big|^2 \big|u(y)\big|^2}{|x-y|^{4}} \ dxdy }{
\displaystyle \iint_{\R^d\times\R^d} \frac{\big|W(x)\big|^2
\big|W(y)\big|^2}{|x-y|^{4}} \ dxdy }.
\endaligned
\end{equation}
By  Lemma \ref{SharpConstant}, we have
\begin{equation*}
\aligned \forall \; u \in \dot H^1, \quad I(u) \geq 0.
\endaligned
\end{equation*}

Choosing  $\displaystyle h\in \dot H^1 $ such that $(W, h)_{\dot H^1}=0, \;
\alpha \in \R$, we consider the expansion of $I(W+\alpha h)$ in
$\alpha$ of order $2$. Note that
\begin{equation*}
\aligned \big\|\nabla (W+ \alpha h) \big\|^4_2 = \big\|\nabla W
\big\|^4_2  \left( 1 + 2 \alpha^2 \frac{\big\|h\big\|^2_{\dot
H^1}}{\big\|W\big\|^2_{\dot H^1} } + O(\alpha^4)\right),
\endaligned
\end{equation*}
and
\begin{align*}
 & \iint_{\R^d\times\R^d} \frac{\big|W(x)+ \alpha h(x) \big|^2 \big|W(y)+
\alpha h(y)\big|^2}{|x-y|^{4}} \ dxdy \\
 = &\iint_{\R^d\times\R^d} \frac{\big|W(x)
\big|^2 \big|W(y) \big|^2}{|x-y|^{4}} \ dxdy \\
&   + 2\alpha    \iint_{\R^d\times\R^d} \frac{\big|W(x) \big|^2
 W(y) h_1(y)+ \big|W(y) \big|^2
 W(x) h_1(x)}{|x-y|^{4}} \ dxdy
\\  &  + \; \alpha^2  \iint_{\R^d\times\R^d} \frac{\big|W(x) \big|^2
\big|h(y)\big|^2+ \big|W(y) \big|^2 \big|h(x)\big|^2+4 W(x)h_1(x)
W(y)h_1(y) }{|x-y|^{4}} \ dxdy   + O(\alpha^3) \\
 = &\iint_{\R^d\times\R^d} \frac{\big|W(x)
\big|^2 \big|W(y) \big|^2}{|x-y|^{4}} \ dxdy \\
& \times \left( 1 + 2 \alpha^2 \frac{\displaystyle
\iint_{\R^d\times\R^d} \frac{\big|W(x) \big|^2 \big|h(y)\big|^2+ 2
W(x)h_1(x) W(y)h_1(y) }{|x-y|^{4}} \ dxdy }{\big\|W\big\|^2_{\dot
H^1}} + O(\alpha^3)\right),
\end{align*}
one easily shows that
\begin{align*}
  I(W+\alpha h)  = &  \frac{ 2 \alpha^2}{\big\|W\big\|^2_{\dot
H^1}} \left( \displaystyle \big\|h\big\|^2_{\dot H^1}-
\iint_{\R^d\times\R^d} \frac{\big|W(x) \big|^2 \big|h(y)\big|^2+2
W(x)h_1(x) W(y)h_1(y)
}{|x-y|^{4}} \ dxdy \right) \\
& + O(\alpha^3)
\end{align*}
where we have used the facts that
\begin{align*}
\iint_{\R^d\times\R^d} \frac{\big|W(x) \big|^2
 W(y) h_1(y)}{|x-y|^{4}} \ dxdy =& \int_{\R^d} - \Delta W(y)\cdot h_1(y)dy
 =0 ,\\
 \iint_{\R^d\times\R^d} \frac{\big|W(y) \big|^2
 W(x) h_1(x)}{|x-y|^{4}} \ dxdy =& \int_{\R^d} - \Delta W(x)\cdot h_1(x)dy
 =0, \\
\iint_{\R^d\times\R^d} \frac{\big|W(x) \big|^2
\big|h(y)\big|^2}{|x-y|^4} \; dxdy = & \iint_{\R^d\times\R^d}
\frac{\big|W(y) \big|^2 \big|h(x)\big|^2}{|x-y|^4} \; dxdy,
 \\
\iint_{\R^d\times\R^d} \frac{\big|W(x) \big|^2 \big|W(y)
\big|^2}{|x-y|^{4}} \ dxdy = &  \int_{\R^d} \big|\nabla W \big|^2
dx= C^{-4}_*.
\end{align*}

Since $I(W+\alpha h) \geq 0$ for all $\alpha \in \R$,  we have
\begin{equation*}
\aligned \Phi(h)= \int_{\R^d} \big|\nabla h \big|^2 dx -
\iint_{\R^d\times\R^d} \frac{\big|W(x) \big|^2 \big|h(y)\big|^2 + 2
W(x)h_1(x) W(y)h_1(y)}{|x-y|^{4}} \ dxdy \geq 0.
\endaligned
\end{equation*}

\noindent{\bf Step 2: Coercivity.} We show that there exists a
constant $c_*>0$ such that for any radial function $h\in H^{\bot}$
\begin{equation*}
\aligned  \Phi (h) \geq c_* \big\|h\big\|^2_{\dot H^1}.
\endaligned
\end{equation*}

We rewrite  $\Phi(h) = \Phi_1 (h_1) + \Phi_2(h_2)$,  where
\begin{align*}
 \Phi_1 (h_1)
: = & \frac12 \int_{\R^d} (L_+h_1) h_1  \; dx\\
=& \frac12 \int_{\R^d} \big| \nabla h_1 \big| ^2 \;dx- \frac12
\iint_{\R^d\times\R^d} \frac{|W(x)|^2|h_1(y)|^2}{|x-y|^4}\;dxdy  \\
& \qquad \qquad \qquad - \iint_{\R^d\times\R^d}
\frac{W(x)h_1(x)W(y)h_1(y)}{|x-y|^4}\; dxdy,
\end{align*}
and
\begin{align*}
 \Phi_2(h_2) :=  & \frac12 \int_{\R^d} (L_-h_2) h_2 \; dx \\
 =& \frac12 \int_{\R^d}
\big| \nabla h_2 \big| ^2 \; dx- \frac12 \iint_{\R^d\times\R^d}
\frac{|W(x)|^2|h_2(y)|^2}{|x-y|^4}\; dxdy.
\end{align*}
By Step 1, $L_+$ is nonnegative on $\{W\}^{\bot}$ (in the sense of
$\dot H^1$) and $L_-$ is nonnegative. We will deduce the coercivity
by the compactness argument \cite[Proposition
2.9]{Wein:85:Modulationary stability}.

We first show that there exists a constant $c$ such that
\begin{equation*}
\aligned   \Phi_1 (h_1) \geq c \big\| h_1 \big\|^2_{\dot H^1},
\endaligned
\end{equation*}
 for any
radial real-valued $\dot H^1$-function $h_1 \in \{W,
\widetilde{W}\}^{\bot}$ (in the sense of $\dot H^1$).
Assume that the above inequality does not hold, then there exists a
sequence of real-valued radial $\dot H^1$-functions $\{f_n\}_n$ such
that
\begin{equation} \label{H-noncoercivity}
\aligned f_n \in H^{\bot}, \quad \lim_{n \rightarrow +\infty}
\Phi_1(f_n) =0, \quad \big\|f_n\big\|_{\dot H^1} =1.
\endaligned
\end{equation}
Extracting a subsequence from $(f_n)$, we may assume that
\begin{equation*}
\aligned f_n \rightharpoonup f_* \; \text{in} \; \dot H^1.
\endaligned
\end{equation*}
The weak convergence of $f_n \in H^{\bot}$ to $f_*$ implies that
$f_*\in H^{\bot}$. In addition, by compactness, we have
\begin{equation*}
\aligned \iint_{\R^d\times\R^d} \frac{|W(x)|^2|f_n(y)|^2}{|x-y|^4}
\; dxdy & \longrightarrow \iint_{\R^d\times\R^d}
\frac{|W(x)|^2|f_*(y)|^2}{|x-y|^4} \; dxdy, \\
 \iint_{\R^d\times\R^d}
\frac{W(x)f_n(x)W(y)f_n(y)}{|x-y|^4} \;dxdy & \longrightarrow
\iint_{\R^d\times\R^d} \frac{W(x)f_*(x)W(y)f_*(y)}{|x-y|^4}\;dxdy.
\endaligned
\end{equation*}
Thus by the Fatou lemma, \eqref{H-noncoercivity} and Step 1, we have
\begin{equation*}
\aligned 0 \leq \Phi_1(f_*)\leq  \liminf_{n\rightarrow +\infty}
\Phi_1(f_n)=0.
\endaligned
\end{equation*}
So we   conclude that $f_*$ is the solution to the following minimizing problem
\begin{align*}
 0= \min_{f\in \Omega\backslash \{0\}   }
 \frac{\displaystyle\int_{\R^d}  L_+ f \cdot f \; dx}{\big\|f\big\|^2_{\dot H^1}},\quad \Omega=\left\{f\in \dot
H^1_{rad}, (f, W)_{\dot H^1}= (f, \widetilde{W})_{\dot
H^1}=0\right\}.
\end{align*}
Thus for some Langrange multipliers $\lambda_0, \lambda_1$, we have
\begin{align*}
L_+f_*=\lambda_0 \Delta W + \lambda_1 \Delta \widetilde{W}.
\end{align*}
Note that $(W, \widetilde{W})_{\dot H^1}=0$ and
$L_+(\widetilde{W})=0$, we have
\begin{align*}
0=\int_{\R^d} f_* L_+(\widetilde{W})=\int_{\R^d} (L_+
f_*)\widetilde{W} =\lambda_1 \big\|\widetilde{W}\big\|^2_{\dot H^1}
\Longrightarrow \lambda_1=0.
\end{align*}
This  tells us that
\begin{align*}
L_+f_*=\lambda_0 \Delta W =-\lambda_0 \big(  |x|^{-4}*|W|^2 \big) W
= \frac{\lambda_0}{2} L_+W.
\end{align*}
By $\text{Null}(L_+) = \text{span}\{\widetilde{W}\}$ in Lemma
\ref{keyassumption},  there exists $\mu_1$ such that
\begin{align*}
f_*= \frac{\lambda_0}{2} W + \mu_1 \widetilde{W}.
\end{align*}
Using $f_* \in H^{\bot}$, we get $\mu_1=0$ and
\begin{align*}
f_*= \frac{\lambda_0}{2} W.
\end{align*}
This  implies that
\begin{align*}
0=\Phi_1(f_*)=  \frac{\lambda^2_0}{4} \Phi_1(W) =
\frac{\lambda^2_0}{4} \Phi (W) \leq 0  \Longrightarrow \lambda_0 =0.
\end{align*}
Therefore, we have
\begin{equation*}
\aligned   f_*=0,\;\; \text{and}\;\; f_n \rightharpoonup 0 \;
\text{in} \; \dot H^1 .
\endaligned
\end{equation*}
Now, by  compactness, we have
\begin{equation*}
\aligned \iint_{\R^d\times\R^d} \frac{|W(x)|^2|f_n(y)|^2}{|x-y|^4}
\longrightarrow 0, \quad
 \iint_{\R^d\times\R^d}
\frac{W(x)f_n(x)W(y)f_n(y)}{|x-y|^4} \longrightarrow 0.
\endaligned
\end{equation*}
By $\Phi_1(f_n)\rightarrow 0$ in \eqref{H-noncoercivity}, we get
that
\begin{equation*}
\aligned \big\|\nabla f_n \big\|_2 \rightarrow 0,
\endaligned
\end{equation*}
which contradicts $\big\|f_n\big\|_{\dot H^1} =1$ in
\eqref{H-noncoercivity}.

Using the same argument, we can show that there exists a constant
$c$, such that for any real-valued radial $\dot H^1$-function $h_2
\in \{W\}^{\bot}$, we have $$\Phi_2 (h_2) \geq c \big\| h_2
\big\|^2_{\dot H^1}. $$ This completes the proof of Proposition \ref{coerH}. \qed

\section{Spectral properties of the linearized operator}
\label{appen-spectralprop} In this Appendix, we will give the proof
of Proposition \ref{spectral}.

\subsection{Existence, symmetry of the eigenfunctions} Note that
$$\overline{\mathcal{L}v}=-\mathcal{L}\overline{v},$$ so that if
$e_0>0$ is an eigenvalue of $\mathcal{L}$ with the eigenfunction
$\mathcal{Y}_+$, $-e_0$ is also an eigenvalue with eigenfunction
$\mathcal{Y}_-=\overline{\mathcal{Y}}_+$.

Now we show the existence of $\mathcal{Y}_+$. Let $\mathcal{Y}_1=\Re
\mathcal{Y}_+, \mathcal{Y}_2 = \Im \mathcal{Y}_+$, it suffices to
show that
\begin{equation}\label{eigenf-system}
  -L_- \mathcal{Y}_2 =\; e_0 \mathcal{Y}_1, \quad
L_+ \mathcal{Y}_1 = \; e_0 \mathcal{Y}_2.
\end{equation}

From the proof of the coercivity property of $\Phi$ on $H^{\bot}\cap
\dot H^1_{rad}$ in Proposition \ref{coerH}, we know that $L_-$ on
$L^2$ with domain $H^2$ is self-adjoint and nonnegative, By Theorem
3.35 in \cite[Page 281]{kato:book}, it has a unique square root $
(L_-)^{1/2} $ with domain $H^1$.

Assume that there exists a function $f\in H^4_{rad}$ such that
\begin{equation}\label{single-eigenf}
\aligned \mathcal{P} f=-e^2_0 f_1, \quad \mathcal{P} := \big(
L_-\big)^{1/2} \big(L_+\big) \big( L_-\big)^{1/2}.
\endaligned
\end{equation}
Then taking
\begin{equation*}
\aligned \mathcal{Y}_1 : = \big( L_-\big)^{1/2} f, \quad
\mathcal{Y}_2 := \frac{1}{e_0} \big(L_+\big) \big( L_-\big)^{1/2} f
\endaligned
\end{equation*}
would yield a solution of \eqref{eigenf-system}, which implies the
existence of the radial $\mathcal{Y}_{\pm}$ by the rotation
invariance of the operator $\mathcal{L}$.

It suffices to show that the operator $\mathcal{P}$ on $L^2$ with
domain $H^4_{rad}$ has a strictly negative eigenvalue. Since
$\mathcal{P}$ is a relatively  compact, self-adjoint, perturbation
of $\big(-\Delta\big)^2$, then by the Weyl theorem
\cite{HisSig:96:book, kato:book}, we know that
\begin{equation*}
\aligned \sigma_{ess}(\mathcal{P})= [0, +\infty).
\endaligned
\end{equation*}
We only need to show that $\mathcal{P}$ has at least one
negative eigenvalue $-e^2_0$.
\begin{lemma}
\begin{equation*}
\aligned
 \sigma_{-}(\mathcal{P}):= \inf\{\big(\mathcal{P}f , f
\big)_{L^2}, f\in H^4_{rad}, \big\|f\big\|_{L^2}=1\}<0.
\endaligned
\end{equation*}
\end{lemma}
\begin{proof} Note that
\begin{equation*}
\aligned \big(\mathcal{P}f, f \big)_{L^2} = \big( L_+F ,
F\big)_{L^2}, \quad F: = \big(L_-\big)^{1/2}f,
\endaligned
\end{equation*}
it suffices to find $F$ such that there exists $g \in
H^4_{rad}$, $ F=(\Delta + V) g$ and
\begin{equation}\label{L+negative}
\aligned \big(L_+F, F \big)_{L^2} < 0.
\endaligned
\end{equation}

Since   $W \in L^2(\R^d)$, we have
\begin{equation*}
\aligned  \text{Ran}(L_-) ^{\bot} = \text{Null}(L_-)
=\text{span}\{W\},
\endaligned
\end{equation*}
Thus
\begin{equation}\label{RLn}
\aligned \text{Ran}(L_-) = \{f\in L^2, (f, W)_{L^2} =0\}.
\endaligned
\end{equation}

Note that $L_+$ is a self-adjoint compact perturbation of $-\Delta$
and
\begin{equation*}
\aligned \big(L_+W, W \big)_{L^2} = -2 \int |\nabla W|^2 \; dx < 0,
\endaligned
\end{equation*}
we easily see that  $L_+$ has a negative eigenvalue. Let $Z$ be the
eigenfunction for this eigenvalue (it is radial by the minimax
principle). Note that $L_+ \widetilde{W} =0$, then for any real
number $\alpha$, we have
\begin{equation}\label{ZW}
\aligned E_0: = \int_{\R^d} L_+(Z+\alpha \widetilde{W} ) \cdot (Z+
\alpha \widetilde{W} ) = \int_{\R^d} L_+  Z \cdot Z < 0.
\endaligned
\end{equation}
Since
$$\big(\widetilde{W} , W\big)_{L^2} \not = 0,$$
we can choose the real number $\alpha_1$ such that
\begin{equation*}
\aligned \big( Z + \alpha_1 \widetilde{W} , W \big)_{L^2} =0,
\endaligned
\end{equation*}
which means that
\begin{equation*}
\aligned \left( (L_- +1) (Z+\alpha_1 \widetilde{W} ) , W
\right)_{L^2} =& \left(
 Z+\alpha_1 \widetilde{W }  ,(L_-+1)  W \right)_{L^2}\\
 =& \big( Z+\alpha_1 \widetilde{W }
,  W \big)_{L^2}=0.
\endaligned
\end{equation*}

By \eqref{RLn}, for any $\epsilon>0$, there exists a function
$G_{\epsilon} \in H^2_{rad}$ such that
\begin{equation*}
\aligned \big\|  L_-  G_{\epsilon} -\big(L_-+1\big) (Z+\alpha_1
\widetilde{W })\big\|_{L^2} < \epsilon.
\endaligned
\end{equation*}
Taking
$$F_{\epsilon}:= \big(L_- +1 \big)^{-1} L_-   G_{\epsilon},$$
we obtain that
\begin{equation*}
\aligned \big\| \big(L_- + 1\big) \big(F_{\epsilon} - (Z + \alpha_1
\widetilde{W })\big)\big\|_{L^2} \leq \epsilon.
\endaligned
\end{equation*}
This implies that
\begin{equation*}
\aligned \big\| F_{\epsilon} - (Z+\alpha_1\widetilde{W}
)\big\|_{H^2} \leq \epsilon \big\| \big(L_- +1 \big)^{-1}\big\|_{L^2
\rightarrow L^2}.
\endaligned
\end{equation*}
Hence for some constant $C_0$, we have
\begin{equation*}
\aligned \left| \int_{\R^d}  L_+   F_{\epsilon} \cdot F_{\epsilon}
-\int_{\R^d}  L_+ (Z+\alpha_1\widetilde{W} ) \cdot
(Z+\alpha_1\widetilde{W}) \right| \leq C_0\epsilon.
\endaligned
\end{equation*}
By \eqref{ZW}, we have \eqref{L+negative} for $F=F_{\epsilon},
\epsilon=\frac{E_0}{2C_0}$. \end{proof}

\subsection{Decay of the eigenfunctions at infinity } By the bootstrap argument,
we know that $\YYY_{\pm} \in C^{\infty}\cap H^{\infty}$. In fact, we
have $\YYY_{\pm} \in \mathcal{S}$. By \eqref{eigenf-system}, it
suffices to show that $\YYY_1\in \mathcal{S}$. Note that $\YYY_1$
satisfies that
\begin{align*}
\left(e^2_0 + \Delta^2 \right)\YYY_1 = &  -2\Delta \left(
\frac{1}{|x|^4}*(W\YYY_1)  \cdot  W \right)- 2 \left(
\frac{1}{|x|^4}*(W\YYY_1) \right)  \left(
\frac{1}{|x|^4}*|W|^2\right)\cdot W\\
& - \Delta \left( \frac{1}{|x|^4}*|W|^2 \cdot \YYY_1\right)  -
\left( \frac{1}{|x|^4}*|W|^2 \cdot \Delta \YYY_1\right) - \left(
\frac{1}{|x|^4}*|W|^2 \right)^2 \cdot  \YYY_1.
\end{align*}
Thus,
\begin{align*} \left(e_0 - \Delta \right)^2 \YYY_1 =& - \Delta \left( \frac{1}{|x|^4}*|W|^2 \cdot \YYY_1\right)  -
\left( \frac{1}{|x|^4}*|W|^2 \cdot \Delta \YYY_1\right) - \left(
\frac{1}{|x|^4}*|W|^2 \right)^2 \cdot  \YYY_1-2e_0 \Delta \YYY_1\\
 &
-2\Delta \left( \frac{1}{|x|^4}*(W\YYY_1) \cdot  W \right)- 2 \left(
\frac{1}{|x|^4}*(W\YYY_1) \right)  \left(
\frac{1}{|x|^4}*|W|^2\right)\cdot W.
\end{align*}
Because of the existence of the nonlocal interaction on the right
hand side, the decay estimate in
\cite{DuyMerle:NLS:ThresholdSolution} does not work. From the Bessel
potential theory in \cite{Ste:70:book}, we know that the integral
kernel $G$ of the operator $\left(e_0 - \Delta \right)^{-2}$ is
\begin{align*}
G(x)=\frac{1}{(4\pi)^2}\int^{\infty}_0e^{-\frac{e_0}{4\pi}\delta}
e^{-\frac{\pi}{\delta}|x|^2} \delta^{\frac{-d+4}{2}}
\frac{d\delta}{\delta}.
\end{align*}
Hence we have
\begin{enumerate}
\item $G(x)=\frac{|x|^{-d+4}}{\gamma(4)}+o(|x|^{-d+4}),\; |x|\rightarrow
0$;
\item there exists $c>0$ such that
$$G(x)=o(e^{-c|x|}), \; |x|\rightarrow +\infty.$$
\end{enumerate}
Then the conclusion follows by the analogue estimates in
\cite{FroJL:07:effective dynamics, KriMR:mass-subcritical har}.\qed

\subsection{Coercivity of $\Phi$ on $G_{\bot}\cap \dot H^1_{rad}$.}
Let $f\in G_{\bot}\cap\dot H^1_{rad}$,  we now decompose $f,
\mathcal{Y}_+, \mathcal{Y}_-$ in the orthogonal sum $\dot H^1 = H
\oplus H^{\bot}$:
\begin{equation*}
\aligned f= \alpha W + \widetilde{h},  \;\;
\mathcal{Y}_+=&\; \eta iW + \xi \widetilde{W} +\zeta W + h_+ ,\\
\mathcal{Y}_- = & -\eta iW + \xi \widetilde{W}  + \zeta W + h_-,
\endaligned
\end{equation*}
where $\widetilde{h},\; h_+,\; h_- \in H^{\bot}\cap \dot H^1_{rad} $
and $h_- =\overline{h}_+$.

\noindent{\bf Step 1.} We first show that for any $f\in G_{\bot}$,
\begin{equation}\label{G formula}
\aligned \Phi(f) = - \frac{B(h_+, \widetilde{h}) B(h_-,
\widetilde{h})}{\sqrt{\Phi(h_+)} \sqrt{\Phi(h_-)}} +
\Phi(\widetilde{h})
\endaligned
\end{equation}
By $\Phi(\mathcal{Y}_{\pm})=0 $ and Remark \ref{r:bilinear form
prop}, we have
\begin{equation*}
\aligned \zeta^2 \Phi(W) + \Phi(h_+)=0,\;\;
\zeta^2 \Phi(W) + \Phi(h_-)=0,\\
\endaligned
\end{equation*}
Since $f\in G_{\bot}$, we have $ B(f, \mathcal{Y}_{\pm})= 0$, which
implies that
\begin{equation*}
\aligned \alpha \zeta \Phi(W) + B(\widetilde{h}, h_+) =0, \;\;
\alpha \zeta \Phi(W) + B(\widetilde{h}, h_-) =0.
\endaligned
\end{equation*}
Thus we have
\begin{equation*}
\aligned \Phi(f) = \alpha^2 \Phi(W)+ \Phi(\widetilde{h}) = -
\frac{B(h_+, \widetilde{h}) B(h_-, \widetilde{h})}{\sqrt{\Phi(h_+)}
\sqrt{\Phi(h_-)}} + \Phi(\widetilde{h}).
\endaligned
\end{equation*}

\noindent{\bf Step 2.} Next we show that
\begin{equation}\label{h-pm indep}
\aligned h_1:=\Re \; h_+ \not =0, \quad h_2:=\Im\; h_+ \not =0.
\endaligned
\end{equation}
In other words,  $h_+$ and $h_-$ are independent in the real Hilbert
space $\dot H^1$.

By conclusion (a) in Proposition \ref{spectral}, we have
\begin{equation}\label{eigen}
\aligned L_- \mathcal{Y}_2 =-e_0 \mathcal{Y}_1,\quad
L_+\mathcal{Y}_1 =e_0 \mathcal{Y}_2.
\endaligned
\end{equation}

We show \eqref{h-pm indep} by contradiction. First assume that
$h_2=0$, then from the decomposition of $\mathcal{Y}_+$,
$$\mathcal{Y}_2 \in \text{span}(W),$$
which is the null space of $L_-$.  Thus,
$\mathcal{L}_-\mathcal{Y}_2=0$,   together with \eqref{eigen},
implies that $\mathcal{Y}_1=0, \mathcal{Y}_2=0$. But it contradicts
the definition of $\mathcal{Y}_+$.

Similarly, assume that $h_1=0$, and by \eqref{eigen},
$L_+\widetilde{W} =0$ and the decomposition of $\mathcal{Y}_+$, we
get
\begin{align*}
  \mathcal{Y}_2 =& \; \frac{1}{e_0}L_+ \mathcal{Y}_1
=   \frac{1}{e_0}L_+ \big( \xi \widetilde{W}  + \zeta W \big)  =
\frac{\zeta}{e_0}L_+  W = -2  \frac{\zeta}{e_0}
\big(|\cdot|^{-4}* \big|W\big|^2 \big) W, \\
\mathcal{Y}_1 = &\;  -\frac{1}{e_0} L_-\mathcal{Y}_2  =
\frac{1}{e_0}\left( \Delta + |\cdot|^{-4}*\big|W \big|^2
\right)\mathcal{Y}_2   \\ = &\; \frac{\zeta}{e^2_0}\left( \Delta +
|\cdot|^{-4}*\big|W \big|^2 \right) \big(|\cdot|^{-4}* |W|^2 \big)W.
\end{align*}
A direct computation shows that $\mathcal{Y}_1 \not \in
\text{span}(W, \widetilde{W})$, which contradicts the decomposition
of $\mathcal{Y}_+$.

\noindent{\bf Step 3: Conclusion of the proof.}  Note that  $\Phi$
is positive definite on $H^{\bot}\cap \dot H^1_{rad}$, we
claim that there exists a constant $b \in (0,1)$ such that
\begin{equation}\label{h control}
\aligned \forall X \in H^{\bot}, \quad \left| \frac{B(h_+, X) B(h_-,
X)}{\sqrt{\Phi(h_+)} \sqrt{\Phi(h_-)}}   \right| \leq b \Phi(X).
\endaligned
\end{equation}
Indeed it is equivalent to show that, by the orthogonal
decomposition on $H^{\bot}$ related to $B$
\begin{equation*}
\aligned  \displaystyle b:=\max_{X\in span \{h_{\pm}\}\atop X
\not=0} \left(\frac{B(h_+, X) )}{\sqrt{\Phi(h_+)} \sqrt{\Phi(X)}}
 \right) \left(\frac{ B(h_-, X)}{\sqrt{\Phi(h_-)}\sqrt{\Phi(X) } }
 \right)< 1.
\endaligned
\end{equation*}
Applying twice Cauchy-Schwarz inequality with $B$, we get $b\leq 1$.
Furthermore, if $b=1$, there exists $X\not=0$ such that the two
Cauchy-Schwarz inequalities become  equalities and thus $X\in
\text{span}\{h_+\} \cap \text{span}\{h_-\} =0,$ which is a
contradiction. Thus $b<1$.

Now by the coercivity of $\Phi$ on $H^{\bot}\cap \dot H^1_{rad}$,
\eqref{G formula} and \eqref{h control}, we have
\begin{equation}
\aligned \Phi(f)\geq (1-b)\Phi(\widetilde{h}) \geq c_1 (1-b)
\big\|\widetilde{h} \big\|^2_{\dot H^1}.
\endaligned
\end{equation}
From the  decomposition of $f $, one easily see that  $\alpha^2 \Phi(W) +
\Phi(\widetilde{h}) = \Phi(f) \geq (1-b)\Phi(\widetilde{h})  $.
Since $\Phi(W)<0$, we have
\begin{equation*}
\aligned b\Phi(\widetilde{h})\geq \alpha^2 \big| \Phi(W) \big| =
\alpha^2  \big\|W\big\|^2_{\dot H^1}. \Longrightarrow C  \Phi(f)
\geq  \alpha^2 \big\|W\big\|^2_{\dot H^1}+ \big\|\widetilde{h}
\big\|^2_{\dot H^1} = \big\|f\big\|^2_{\dot H^1}.
\endaligned
\end{equation*}
The proof is  complete. \qed

\subsection{Characterization of the real spectrum of $\mathcal{L}$}
 Note that $\mathcal{L}$ is a compact perturbation of
$\begin{pmatrix} 0 & \Delta \\-\Delta & 0 \end{pmatrix}$, thus its
essential spectrum is $i\R$. Consequently, $0\in
\sigma(\mathcal{L})$ and $\sigma(\mathcal{L}) \cap \big(
\R\backslash\{0\}\big) $ contains only eigenvalues. Furthermore, we
have show that $\{\pm e_0\} \subset \sigma(\mathcal{L})$. It remains
to show that $\pm e_0$ are the only eigenvalues of $\mathcal{L}$ in
$\R\backslash\{0\} $.

We argue by contradiction. Assume that for some $f\in H^2$, we have
\begin{equation*}
\aligned \mathcal{L}f=e_1 f,\quad e_1 \in \R\backslash \{0, \pm
e_0\}.
\endaligned
\end{equation*}
By \eqref{B-pro}, we have  $$(e_1+e_0) B (f, \mathcal{Y}_+) =
(e_1-e_0) B (f, \mathcal{Y}_-)=0, \quad e_1 B(f,f) = -e_1 B(f, f).$$
Thus
\begin{equation*}
\aligned B (f, \mathcal{Y}_+) =   B (f, \mathcal{Y}_-)=0, \quad
B(f,f)=0.
\endaligned
\end{equation*}
Now we write
\begin{equation*}
\aligned f= \beta iW + \gamma \widetilde{W} + g, \quad g\in
G_{\bot}, \; \beta =  \frac{(f, iW)_{\dot
H^1}}{\big\|W\big\|^2_{\dot H^1}},\; \gamma =  \frac{(f,
\widetilde{W})_{\dot H^1}}{\big\|\widetilde{W}\big\|^2_{\dot H^1}}.
\endaligned
\end{equation*}
By the coercivity of $\Phi$ on $G_{\bot}$ in Proposition
\ref{spectral} and Remark \ref{r:bilinear form prop}, we have
\begin{equation*}
\aligned 0 = B(f,f) = B(g,g) \gtrsim \big\|g\big\|^2_{\dot H^1},
\Longrightarrow g=0.
\endaligned
\end{equation*}
thus we have
\begin{equation*}
\aligned e_1 f = \mathcal{L}f = \eta \mathcal{L}(iW) + \gamma
\mathcal{L}(W)=0.
\endaligned
\end{equation*}
Since $e_1 \not =0$, we have
$$f=0.$$
This  contradicts the definition of the eigenfunctions,  and so we concludes
the proof. \qed

%
%
%
%

\end{document}